\documentclass[11pt,reqno]{amsart}
\usepackage[normalem]{ulem}
\usepackage[margin=1in,letterpaper]{geometry}
\usepackage{graphicx}
\usepackage{amssymb}
\usepackage{amsthm}
\usepackage{mathrsfs}
\usepackage{stmaryrd}
\usepackage{accents} 
\usepackage{enumitem} 
\usepackage{subcaption}
\usepackage{microtype}
\usepackage{colonequals}
\usepackage{textcomp} 
\usepackage{gensymb}
\usepackage{color}
\usepackage{lmodern}

\usepackage[bookmarksopen,bookmarksdepth=2]{hyperref} 
\allowdisplaybreaks%

\makeatletter\let\over\@@over\makeatother

\numberwithin{equation}{section}
\theoremstyle{plain} 
\newtheorem{theorem}{Theorem}[section]

\newtheorem{lemma}[theorem]{Lemma} 
\theoremstyle{remark}
\newtheorem{remark}[theorem]{Remark}
\theoremstyle{definition}

\newtheorem{assumption}{Assumption}

\newcommand{\be}{\begin{equation}}
\newcommand{\ee}{\end{equation}}%
\newcommand{\bse}{\begin{subequations}}
\newcommand{\ese}{\end{subequations}}

\newcommand{\realpart}{\operatorname{Re}}
\newcommand{\imagpart}{\operatorname{Im}}

\newcommand{\sech}{\operatorname{sech}}

\newcommand{\id}{\operatorname{id}}

\newcommand{\p}{\partial}
\newcommand{\R}{\mathbb{R}} 
\newcommand{\F}{\mathscr{F}}

\newcommand{\placeholder}{\,\cdot\,}

\newcommand{\n}[2][]{#1\lVert #2 #1\rVert}
\newcommand{\abs}[2][]{#1\lvert #2 #1\rvert}

\newcommand{\bdd}{\mathrm{b}}       
\newcommand{\loc}{{\mathrm{loc}} } 
\newcommand{\even}{{\mathrm{e}} }

\newcommand{\C}{\mathbb{C}}
\newcommand{\tr}{\mathrm{tr\,}}
\newcommand{\scrW}{\mathscr{W}}
\newcommand{\scrX}{\mathscr{X}}
\newcommand{\bfu}{\mathbf{u}}
\newcommand{\frakb}{\mathfrak{b}}
\newcommand{\fraka}{\mathfrak{a}}
\newcommand{\Vaug}{\mathcal V^{\mathrm{aug}}}
\newcommand{\zun}{\underline{\zeta_1}}
\newcommand{\bfxi}{\boldsymbol{\xi}}
\newcommand{\bfz}{\mathbf{z}}
\newcommand{\bfv}{\mathbf{v}}

 \newcommand\per{\textup{per}}

\newcommand{\nbhdO}{\mathcal{O}}
\newcommand{\tube}{\mathcal{U}}

\newcommand\fluidD{\mathscr{D}}
\newcommand\fluidS{\mathscr{S}}
\newcommand{\fluidB}{\mathscr{B}}
\newcommand{\fluidV}{\mathscr{V}}

\newcommand{\Dom}[1]{\operatorname{Dom}{#1}}
\newcommand{\Rng}[1]{\operatorname{Rng}{#1}}

\newcommand{\cinterval}{\mathscr{I}}

\newcommand{\Lin}{\mathrm{Lin}}
\newcommand{\eng}{E}
\newcommand{\Hc}{H_c}
\newcommand{\augHam}{E_c}
\newcommand{\mom}{P}

\newcommand{\spectrum}[1]{\operatorname{spec}{#1}}

\newcommand{\xiofx}{\underline{\xi}}

\newcommand{\Xspace}{\mathbb{X}}

\newcommand{\Wspace}{\mathbb{W}}
\newcommand{\Vspace}{\mathbb{V}}

\newcommand{\jbracket}[1]{\left\langle{#1}\right\rangle}

\newcommand{\sigmas}{\sigma_{\mathrm{s}}}
\newcommand{\sigmav}{\sigma_{\mathrm{v}}}

\newcommand{\Bs}{B_{\mathrm{s}}}
\newcommand{\Bv}{B_{\mathrm{v}}}
\newcommand{\Gammav}{\Gamma_{\mathrm{v}}}
\newcommand{\Gammas}{\Gamma_{\mathrm{s}}}
\newcommand{\fs}{f^{\mathrm{s}}}
\newcommand{\ws}{w^{\mathrm{s}}}
\newcommand{\Ws}{W_{\mathrm{s}}}
\newcommand{\Ks}{K^{\mathrm{s}}}
\newcommand{\Kv}{K^{\mathrm{v}}}
\newcommand{\fv}{f^{\mathrm{v}}}
\newcommand{\imagimag}{\mathrm{ii}}
\newcommand{\realimag}{\mathrm{ri}}

\newcommand{\cm}{\mathscr{M}}
\newcommand{\cmconf}{\mathfrak{M}}
\newcommand{\km}{\mathscr{K}}

\usepackage{tikz}
\usepackage{pgfplots}
\usetikzlibrary{decorations.markings}
\usepgfplotslibrary{colormaps,fillbetween}
\pgfplotsset{compat=1.10}
\definecolor{owlgray}{RGB}{220,225,220}

\begin{document}

\title[Capillary-gravity wave-borne vortices]{Existence and stability of capillary-gravity wave-borne vortices}

\date{\today}

\author[G. Slease]{Gregory Slease}
\address{Department of Sciences and Mathematics, University of Washington Tacoma, Tacoma, WA 98402}
\email{gslease@uw.edu}

\author[S. Walsh]{Samuel Walsh}
\address{Department of Mathematics, University of Missouri, Columbia, MO 65211} 
\email{walshsa@missouri.edu} 

\author[R. Zhong]{Runzhang Zhong}
\address{Department of Mathematics, University of Missouri, Columbia, MO 65211} 
\email{rz47r@missouri.edu}

\begin{abstract}
In this paper, we consider the existence and stability of waves progressing through a finite-depth two-dimensional body of water that carry a vortex in their bulk. The waves are acted upon by gravity, and sit below a region of air at constant pressure, with the air--water interface being a free boundary along which capillary effects are incorporated. We consider two classes of vortices: point vortices, where formally the vorticity is a Dirac $\delta$, and hollow vortices, where the vortex core is a region of constant pressure outside the fluid domain and about which there is a nonzero circulation.  

First, we prove that steady small-amplitude wave-borne point vortices exist for any subcritical Froude number and supercritical Bond number, and then show they are conditionally orbitally stable. Second, through a vortex desingularization argument, we construct capillary-gravity wave-borne hollow vortices. Notably, the wave speed is $O(1)$ for both these families.
\end{abstract}

\maketitle

\setcounter{tocdepth}{1}
\tableofcontents

\section{Introduction}

A \emph{vortex} is a region of highly concentrated vorticity within a fluid. At the most singular end are  \emph{point vortices}, where formally the vorticity consists of a Dirac $\delta$ measure supported on a discrete set (the vortex centers). This is not a weak solution of the Euler equations, but rather of a model equation that dates back to the work of Helmholtz and Kirchhoff. More regular examples include \emph{vortex patches}, which are solutions to the Euler equations for which the vorticity is supported on a compact domain. \emph{Hollow vortices} are a less common variety that have recently enjoyed renewed interest. They are (classical) solutions of the free boundary Euler equations having bubbles of constant pressure suspended in the bulk of an irrotational fluid. The vorticity is supported on the free boundary in the sense that there is a nonzero circulation about each bubble, making hollow vortices more regular than point vortices but less regular than patches.

Our interest is in vortices that are carried along with traveling water waves. The existence and stability of steady water waves is a classical topic with over two centuries of mathematical literature devoted to it. It was only approximately $20$ years ago that significant progress was made in constructing rotational water waves. Initially, these works only allowed for non-localized vorticity that does not even decay in the far field. More recently, though, various types of localized vorticity waves have been studied, including wave-borne point vortices~\cite{chen2025vortex,C1,C2,CDI,CK,le2019existence,shatah2013travelling,varholm2016solitary}, hollow vortices~\cite{chen2025vortex,crowdy2026exact}, vortex patches~\cite{shatah2013travelling}, and vortex spikes~\cite{ehrnstrom2023smooth}. For a survey of these results, see~\cite[Section 6]{HHSTWWW}.  Broadly speaking, the problem has been considered either in regimes in which the influence of capillarity far outweighs that of gravity, capillarity is absent, or both capillarity and gravity are absent. Moreover, these treatments deal primarily with slow moving waves for which the wave speed is perturbative. It is also noteworthy that the stability of these waves, even in the linear sense, is largely unknown. To the best of our knowledge, the only exceptions to this are the works of Varholm, Wahlén, and Walsh~\cite{varholm2020stability}, which establishes the (conditional) orbital stability of sufficiently small-amplitude, slow moving deep water capillary-gravity wave-borne vortices, and Le~\cite{le2019existence}, which showed pairs of vortices in the same regime are orbitally unstable. 

In this paper, we give the first constructions of wave-borne point vortices and hollow vortices with strong surface tension but $O(1)$ gravity and wave speed.  Moreover, we prove that sufficiently small-amplitude wave-borne point vortices are conditionally orbitally stable even when the wave speed is not perturbative.

\subsection{Wave-borne vortex problem}

\begin{figure}

\centering
\includegraphics[width=1.0\linewidth]{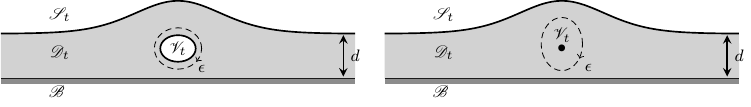}
\caption{The two species of wave-borne vortices we consider. On the left is a wave-borne hollow vortex: the thick shaded curves are the free boundary, which consists of $\partial \fluidV_t \cup \fluidS_t$, the core $\fluidV_t$ and air region are unshaded while the fluid domain $\fluidD_t$ is light gray. The rigid bed is shaded a darker gray.  On the right is a wave-borne point vortex: in this case $\fluidV_t$ is a single point. In both diagrams, a simple contour circling the vortex is depicted in dashed lines, with $\epsilon$ being the circulation of the velocity field about it. These waves are spatially localized in that $\fluidS_t$ is asymptotically flat, with $d$ being the undisturbed depth of the water.}
\label{time-dependent setup figure}
\end{figure}

Let us now formulate the system. Fix Cartesian coordinates $x = (x_1,x_2)$. We consider a body of water modeled as an incompressible, homogeneous, and inviscid fluid occupying a time-dependent domain $\fluidD_t\subset\R^2$ for time $t\geq 0$. The water is bounded below by a rigid horizontal bed $\fluidB \colonequals \{ x_2 = -d \}$, with $d > 0$ being the undisturbed depth. Above the water is air (vacuum) at constant pressure normalized to be $0$. The interface separating the air and water is a curve $\fluidS_t$ that evolves in time. In some cases, we will assume that $\fluidS_t$ can be described globally as the graph of a free surface profile $\eta = \eta(t,x_1)$, but we do not make that restriction yet. We will always consider spatially localized waves for which $\fluidS_t$ limits to the $x_1$-axis as $|x_1| \to \pm\infty$.

Our interest is in water waves carrying vortices. Let $\fluidV_t \subset\subset \fluidD_t^c$ denote a vortex core suspended in the fluid domain and lying below $\fluidS_t$. We study two models. A wave-borne \emph{point vortex} corresponds to the core being a (time-dependent) singleton $\fluidV_t = \{ \mathbf{z}(t)\}$, where $\mathbf{z}(t)$ is called the \emph{vortex center}. A \emph{hollow vortex} corresponds to the core $\fluidV_t$ being the interior of a Jordan curve and the pressure there being spatially constant. For both models, we assume that there is a nonzero circulation $\epsilon$ about $\mathscr{V}_t$. Formally, one can think of the vorticity as a Dirac measure with constant strength supported on $\partial\fluidV_t$.  It is important to note, however, that this set is not actually in the bulk of the fluid. Indeed, $\partial \fluidD_t$ has three disjoint components: $\fluidS_t$, $\partial \fluidV_t$, and $\fluidB$. The \emph{free boundary} is defined to be $\partial \fluidD_t \setminus (\fluidV_t\cup \fluidB)$, which for the point vortex case consists of just the air-water interface $\fluidS_t$ but for hollow vortices also includes the core boundary $\partial \fluidV_t$. This configuration is illustrated in Figure~\ref{time-dependent setup figure}.

In the bulk of the fluid, we impose the 2D incompressible Euler equations. The conservation of momentum and incompressibility take the form 
\bse
\label{intro general}
\begin{equation}\label{eq:2D-Incompressible-Euler}
\left\{
\begin{aligned}
  \p_t \mathbf{v} + \nabla \cdot \left( \mathbf{v} \otimes \mathbf{v} \right)  + \nabla p + g \mathbf{e}_2&= 0 \\
  \nabla\cdot \mathbf{v} &= 0
\end{aligned}
\right.
	\qquad \text{in } \fluidD_t,
\end{equation}
where $\mathbf{v}=\mathbf{v}(t,\placeholder) \colon \overline{\fluidD_t} \setminus \fluidV_t  \to\R^2$ is the velocity, $p=p(t,\placeholder) \colon \overline{\fluidD_t} \setminus \fluidV_t \to\R$ is the fluid pressure, and $g > 0$ is the constant gravitational acceleration. As discussed above, we assume the flow is irrotational, but there is a nonzero circulation $\epsilon$ about the vortex core:
\be
\label{intro vorticity assumptions}
	\nabla^\perp \cdot \mathbf{v} = 0 \quad \textrm{in } \fluidD_t, \qquad \int_{\mathcal{C}_t} \mathbf{v} \cdot \mathbf{t} \, dS = \epsilon 
\ee
where $\mathcal{C}_t \subset\subset \fluidD_t$ is any smooth Jordan curve (oriented counter-clockwise) whose interior contains $\fluidV_t$, and $\mathbf{t}$ is the unit tangent vector.

On the boundary, we impose the two conditions. The \emph{kinematic condition} is that the free boundary and bed are material lines, meaning that the vector field
\be
\label{intro general kinematic}
	\partial_t + \mathbf{v} \cdot \nabla \quad \textrm{is tangent to } \bigcup_{t \geq 0} \{t\} \times \left(  \partial \fluidD_t \setminus \fluidV_t \right).
\ee
Note that on the bed, this requirement enforces impermeability. Second, on each component of the free boundary, we impose the \emph{dynamic condition}. Since the air is at constant pressure $0$, the Young--Laplace law implies that the trace of $p$ is proportional to the signed curvature $\kappa$ of the air-water interface:  
\be
\label{intro general dynamic}
	p = \sigmas \kappa \qquad \textrm{on } \fluidS_t.
\ee
Here, $\sigmas \geq 0$ is the coefficient of surface tension on $\fluidS_t$. Likewise, in the hollow vortex case, the dynamic condition also applies on the core boundary: 
\be
\label{intro general dynamic core}
	p - p_{\mathrm{v}}(t) = \sigmav \kappa \qquad \textrm{on } \partial \fluidV_t,
\ee
where $\sigmav \geq 0$ is the coefficient of surface tension on $\partial \fluidV_t$, and $p_{\mathrm{v}}(t)$ is the spatially constant pressure inside the core. There are multiple choices for prescribing $p_{\mathrm{v}}(t)$. For example, if we imagine the core as a second compressible phase obeying the ideal gas law, then $p_{\mathrm{v}}(t)$ is inversely proportional to $|\fluidV_t|$. This is related to the so-called adiabatic approximation; see~\cite{benjamin1987hamiltonian,prosperetti1991bubbles,chen2026finite}. Alternatively, one can treat $p_{\mathrm{v}}(t)$ as external forcing. For the steady wave solutions we construct, the pressure will be constant in both space and time, which is consistent with either model.

The above conditions give a closed system for the hollow vortex model, but for a wave-borne point vortex, one must supplement them with an equation governing the motion of the vortex center. The physical intuition, which goes back to Kirchhoff and Helmholtz, is that the center should not self-advect. Thus, we must remove its singular contribution to the velocity field, and use the resulting (smooth) vector field. In other words, the vector field that drives the point vortex should be only the irrotational part of the velocity, that is
\begin{equation}\label{eq:Helmholtz-Kirchhoff}
	\tag{\ref{intro general dynamic core}'}
    \frac{d \mathbf{z}}{dt} = \left( \mathbf{v}-\frac{\epsilon}{2\pi}\nabla^{\perp} (\log{\abs{\placeholder-\mathbf{z}(t)})} \right) \Big|_{\mathbf{z}(t)}.
\end{equation}
\ese
In what follows, the time-dependent wave-borne hollow vortex problem refers to the equations~\eqref{eq:2D-Incompressible-Euler}--\eqref{intro general dynamic core}. On the other hand, the time-dependent wave-borne point vortex problem corresponds to~\eqref{eq:2D-Incompressible-Euler}--\eqref{intro general dynamic} along with the Helmholtz--Kirchhoff equation for the vortex dynamics~\eqref{eq:Helmholtz-Kirchhoff}. For simplicity, we call~\eqref{intro general} the (time-dependent) wave-borne vortex problem.

\begin{remark}
\label{well-posedness remark}
The Cauchy problem for gravity and capillary-gravity water waves has been studied extensively with the assumption that $\fluidV_t = \emptyset$,  $\fluidD_t$ is simply connected, and $\nabla^\perp \cdot \mathbf{v}$ is sufficiently smooth (or just vanishes identically); see, for example~\cite[Chapter 1]{lannes2013book} for an overview. Much of this theory carries to the wave-borne vortex problem~\eqref{intro general}, but regrettably the literature treating it directly is quite sparse. 

For a wave-borne hollow vortex, though the topology of the fluid domain is more complicated, the a priori estimates of Shatah and Zeng~\cite{shatah2008geometry,shatah2011interface} imply the local well-posedness of the Cauchy problem in sufficiently high regularity Sobolev spaces provided the \emph{Rayleigh--Taylor sign condition} is satisfied. That is, one must have that
\be
\label{rayleigh-taylor sign condition}
	-\mathbf{n} \cdot \nabla p \geq m > 0 \qquad \textrm{on } \partial\fluidD_t \setminus \fluidV_t,
\ee  
for some constant $m > 0$ where $\mathbf{n}$ is the outward unit normal vector to $\fluidD_t$ along the free boundary. See also the discussion of the planar hollow vortex problem in~\cite[Appendix B]{chen2026desingularization}. Local well-posedness for the wave-borne point vortex problem likewise follows under the same condition, as the vortex dynamics equation~\eqref{eq:Helmholtz-Kirchhoff} is simply an ODE and the influence of the vortex on $\fluidS_t$ can be treated as external forcing. 

Long-time well-posedness is of course far more subtle even for the irrotational case. Su~\cite{su2020longtime} obtained long-time well-posedness for gravity wave-borne point vortices in infinite-depth water again under the assumption~\eqref{rayleigh-taylor sign condition}, and Wan~\cite{wan2025gravitycapillary} recently showed long-time well-posedness for the capillary-gravity case. On the other hand, Su~\cite{su2023transition} found initial data such that, in finite time, the point vortex comes close enough to $\fluidS_t$ that the Rayleigh--Taylor sign condition is violated, at which point the system becomes ill-posed. 
\end{remark}

\subsection{Traveling wave-borne point vortices}\label{sec:intro-point-vortex}

 A \textit{steady} or \textit{traveling} wave-borne vortex  is a solution to~\eqref{intro general} that evolves in time by translating at constant velocity $c \in \mathbb{R}$ without changing shape. Concretely, this means that the fluid domain, air-water interface, and vortex core at time $t$ have the form
 \[
 	\fluidD_t = \left\{  (x_1+c t, x_2) : x \in  \fluidD_0 \right\}, \quad 
	\fluidS_t = \left\{  (x_1 + ct, x_2) : x  \in \fluidS_0 \right\}, \quad 
	\fluidV_t = \left\{  (x_1 + ct, x_2) : x  \in \fluidV_0 \right\}.
\]
Making the change of independent variables $x \to (x_1-ct, x_2)$, the wave-borne vortex problem~\eqref{intro general} can be recast as time-independent PDE for the relative velocity field $\mathbf{u} = (u_1, u_2) \colon \overline{\fluidD_0} \setminus \fluidV_0 \to \mathbb{R}^2$, which is related to the time-dependent velocity field in the lab frame via
\[
	\mathbf{v}(t,x) \equalscolon \mathbf{u}(x_1 -ct, x_2) + c\mathbf{e}_1.
\]

It is also convenient at this stage to nondimensionalize the system, taking $d$ as the characteristic length scale and $|c|$ as the characteristic velocity. We will write $\fluidD$, $\fluidS$, and $\fluidV$ for the fluid domain, air-water interface, and vortex core boundary in the nondimensional variables (in the moving frame), and we denote by $\gamma$ the nondimensionalized vortex strength. But, otherwise, we will use the same notation for the dimensional and dimensionless quantities unless explicitly stated.  In particular, $\fluidD$ is the region lying below $\fluidS$, above the bed $\fluidB = \{ x_2 = -1\}$, and in the exterior of $\fluidV$. Note that $\mathbf{u} \to (-1,0)$ as $|x_1| \to \infty$, meaning that the far-field relative velocity is $O(1)$. 

This process introduces several nondimensional parameters that strongly influence the qualitative property of the system: the \textit{Froude number} $F$ and the \textit{Bond numbers} $\Bs$ and $\Bv$ associated to air-water interface and core boundary (if it is nonempty), respectively. These are defined by 
\[
	F^2 = \frac{c^2}{gd},  \qquad \Bs \colonequals  \frac{\sigmas}{dc^2}, \qquad \Bv \colonequals \frac{\sigmav}{dc^2}.
\]
One can think of $F$ as a nondimensionalized wave speed, while the Bond numbers are nondimensionalized measures of the capillarity on the corresponding free boundary components. Based on the dispersion relation for the water wave problem (without a submerged vortex), we call $F = 1$ and $\Bs = 1/3$ the critical Froude number and (surface) Bond numbers.  For our constructions, the size of $\Bv$ is unimportant except for the case $\Bv = 0$, which would require a different scaling argument as we discuss further below.

Moving to the nondimensionalized variables in the translating frame, the resulting system is the steady incompressible Euler equations
\bse
\label{intro steady Euler}
\begin{equation}\label{eq:2D-steady-Incompressible-Euler}
\left\{
\begin{aligned}
  \nabla \cdot \left( \mathbf{u} \otimes \mathbf{u} \right)  + \nabla p + \frac{1}{F^2} \mathbf{e}_2&= 0 \\
  \nabla\cdot \mathbf{u} &= 0
\end{aligned}
\right.
	\qquad \text{in } \fluidD,
\end{equation}
together with the localized vorticity assumption
\be
\label{intro steady vorticity assumptions}
	\nabla^\perp \cdot \mathbf{u} = 0 \quad \textrm{in } \fluidD, \qquad \int_{\mathcal{C}} \mathbf{u} \cdot \mathbf{t} \, dS = \gamma.
\ee
The kinematic boundary condition now takes the form
\be
\label{intro steady kinematic condition}
	\mathbf{u} \quad \textrm{is tangent to} \quad \partial \fluidD \setminus \fluidV,
\ee
while the \emph{dynamic} or \emph{Bernoulli boundary condition} on the air-water interface is
\be
\label{intro steady dynamic condition surface}
	\Bs \kappa + \frac{1}{2}|\mathbf{u}|^2 + \frac{1}{F^2} x_2 = \frac{1}{2} \quad\text{on}\quad \fluidS.\\
\ee
Note that here we have used Bernoulli's law to eliminate the pressure in~\eqref{intro steady dynamic condition surface}. The constant $1/2$ on the right hand-side is found by considering the far-field limit. For the wave-borne hollow vortex problem, the dynamic condition on the vortex core similarly can be written
\be
\label{intro steady dynamic condition core}
	\Bv \kappa + \frac{1}{2}|\mathbf{u}|^2 + \frac{1}{F^2} x_2 = q \quad\text{on}\quad \partial\fluidV,
\ee
where $q$ is the so-called \emph{Bernoulli constant}. On the other hand, in the case of a wave-borne point vortex, the Helmholtz--Kirchhoff model states~\eqref{eq:Helmholtz-Kirchhoff} is equivalent to
\be
\label{intro steady KH equation}
\tag{\ref{intro steady dynamic condition core}'}
	0 = \left( \mathbf{u} - \frac{\gamma}{2\pi} \nabla^\perp \log{|\placeholder - \mathbf{z}|} \right)\big|_{\mathbf{z}}.
\ee
\ese
This expresses the fact that a point vortex carried by the wave will be stationary in the moving frame.

 We will refer to~\eqref{intro steady Euler} as the \emph{steady wave-borne vortex problem}. Again, the understanding is that in the context of hollow vortices, this corresponds to \eqref{eq:2D-steady-Incompressible-Euler}--\eqref{intro steady dynamic condition core}, while for wave-borne point vortices, we mean \eqref{eq:2D-steady-Incompressible-Euler}--\eqref{intro steady dynamic condition surface} together with the vortex dynamic equation~\eqref{intro steady KH equation}. When $\Bs > 0$, we say these are capillary-gravity wave-borne vortices, whereas if $\Bs = 0$, they are called gravity wave-borne vortices.

\subsection{Statement of results}
\label{intro statement of results section}

In this paper, we consider both the existence of steady wave-borne vortices and their (nonlinear) stability. We give a preliminary statement of results here, as the rigorous versions are best phrased after making a further change of variables. 

First, we show that there exists a family of small-amplitude but fast moving capillary-gravity wave-borne point vortices. 

{\begin{theorem}[Wave-borne point vortices]
\label{intro existence point vortex theorem}
Fix $k \geq 2$ and $\alpha\in(0,1)$. For any supercritical surface Bond number $\Bs^0 > \frac13$ and any subcritical Froude number $F_0 \in (0,1)$, there exists a three-parameter family $\cm_\loc$ of solitary capillary-gravity water waves with a submerged point vortex solving~\eqref{intro steady Euler} and satisfying the following. 
\begin{enumerate}[label=\rm(\alph*)]
	\item $\cm_\loc$ admits the real-analytic parameterization 
	\[
		\cm_\loc = \left\{ (\mathbf{u}^{\beta,F,\Bs}, \, \mathbf{z}^{\beta,F,\Bs}, \, \gamma^{\beta,F,\Bs}, \, \fluidS^{\beta,F,\Bs}) : \quad  |\beta| < \beta_0, ~ |F-F_0| + |\Bs - \Bs^0| < \delta \right\}
	\]
	for some $\beta_0, \delta > 0$. Here $\beta$ is, to leading order, the height of the vortex center above the bed.
	
	\item Each $\fluidS^{\beta,F,\Bs}$ is globally the graph of a function $\eta^{\beta,F,\Bs} \in C^{k+\alpha}(\mathbb{R}) \cap H^{k}(\mathbb{R})$, and we have the asymptotics
	\be
	\label{intro pv asymptotics}
	\begin{aligned}
		\eta^{\beta,F,\Bs} & = O(\beta^2) \quad \textrm{in } C^{k+\alpha}(\mathbb{R}) \cap H^{k}(\mathbb{R}), \\
		\mathbf{z}^{\beta,F,\Bs} & = \left( 0, ~  -1 + \beta + O(\beta^3) \right),  \\
		\gamma^{\beta,F,\Bs} &= 4\pi\beta + \frac43\pi^3\beta^3 + O(\beta^5).
	\end{aligned}
	\ee
	\item $\cm_\loc$ includes every solution in a neighborhood of the trivial solution $(-\mathbf{e}_1, -\mathbf{e}_2, 0, \{x_2 = 0\})$ that is symmetric with respect to the $x_2$-axis.
\end{enumerate}
\end{theorem}}

We prove this theorem via an implicit function theorem argument. As noted above, the main novelty is that the wave speed is $O(1)$ --- or, equivalently given the nondimensionalization, the gravity is $O(1)$. Previously, Varholm~\cite{varholm2016solitary} constructed families of solitary and periodic capillary-gravity wave-borne vortices for which both the circulation and wave speed were $O(\beta)$, essentially. An analogous ``fast moving'' gravity wave-borne point vortex existence theory was given by Chen, Varholm, Walsh, and Wheeler~\cite{chen2025vortex}. Theorem~\ref{intro existence point vortex theorem} adapts the approach of that paper to incorporate surface tension. It is possible to use (analytic) global bifurcation techniques to extend $\cm_\loc$ to a much larger family. However, as usual, surface tension on $\fluidS$ means that in general one cannot hope that these solutions will exhibit any monotonicity properties, which are essential to obtaining a satisfying characterization of the limiting behavior. For instance, it is not possible to rule out the possibility that a global solution curve is a closed loop under these assumptions. For that reason, we do not pursue a global continuation argument here.

The next theorem establishes the existence of steady wave-borne hollow vortices. This is a considerably more delicate construction than wave-borne point vortices. We allow for either strong surface tension or no surface tension on the air-sea interface; on the vortex boundary, we are able to treat any $\Bv > 0$.
\begin{theorem}[Wave-borne hollow vortices]
\label{intro existence hollow vortex theorem}
Fix $k \geq 2$, $\alpha \in (0,1)$, and $\lambda > 1$. Assume that either 
\[
	\Bs > \tfrac{1}{3} \textrm{ and } F^2 < 1  \qquad \textrm{or} \qquad \Bs = 0 \textrm{ and } F^2 > 1.
\]
For any $\Bv > 0$, there exists a curve $\km_\loc$ of solitary wave-borne hollow vortices solving~\eqref{intro steady Euler} and satisfying the following. 
\begin{enumerate}[label=\rm(\alph*\rm)]
	\item \label{intro hollow curve part} The curve $\km_\loc$ admits a real-analytic parameterization
\[
	\km_\loc = \left\{ \left( \mathbf{u}^\rho, \, \gamma^\rho, \, q^\rho, \fluidS^\rho, \, \fluidV^\rho \right) : \quad 0 < \lambda \rho \ll 1 \right\},
\]
where $\Bs$, $\Bv$, and $F$ are fixed along the family. 
\item \label{intro hollow asymptotics part} Each $\fluidS^\rho$ is the graph of a $C^{k+\alpha}$ function $\eta^\rho = \eta^\rho(x_1)$, and each vortex core boundary $\partial\fluidV^\rho$ admits the polar graph parameterization
\[
	\partial\fluidV^\rho = \left\{ -\mathbf{e}_2 + \rho \left(  R^\rho(\theta) \cos{\theta}, ~ \lambda +  R^\rho(\theta) \sin{\theta} \right)   : \quad \theta \in [0,2\pi)  \right\},
\]
where $R^\rho = R^\rho(\theta) \in C_{\per}^{k+\alpha}([0,2\pi])$. Moreover, we have the asymptotics
\[
		R^\rho = 1 + O(\rho^2) \quad  \textrm{in } C_\per^{k+\alpha}([0,2\pi]) \qquad 
		\eta^\rho = O(\rho^2) \quad \textrm{in } C^{k+\alpha}(\mathbb{R}),
\]
and
\be
\label{intro hollow gamma q asymptotics}
		\gamma^\rho = O(\rho^2), \quad q^\rho = O(\rho^{-1}).
\ee
\end{enumerate}
\end{theorem}

Like Theorem~\ref{intro existence point vortex theorem}, this result is proved via the implicit function theorem: we imagine injecting a point vortex into the fluid domain through the bed, while simultaneously desingularizing it into a hollow vortex. The parameter $\rho$ is the radius of the vortex core to leading order. Gravity wave-borne hollow vortices ($\Bv = \Bs = 0$ and $F^2 > 1$) were constructed with a thematically similar argument in~\cite{chen2025vortex}, where the desingularization was actually carried out for a (generic) gravity wave-borne point vortex. The same approach, combined with Theorem~\ref{intro existence point vortex theorem}, can be easily adapted to prove the existence of wave-borne hollow vortices with $\Bv =0$ but strong surface tension on the air-sea interface ($\Bs > 1/3$, $F^2 < 1$). However, when surface tension is present on the vortex boundary ($\Bv > 0$), the process requires more finesse as the curvature of $\partial\fluidV^\rho$ is an $O(1/\rho)$ singular term in the dynamic condition~\eqref{intro steady dynamic condition core}. The need to balance this against the (singular) kinetic energy density term $|\mathbf{u}^\rho|^2/2$ necessitates taking either $\gamma$ or $\Bv$ to be small; here, we elect to do the former.  We also mention the recent work of Crowdy~\cite{crowdy2026exact}, where an entirely different approach based on complex function theory is used to construct wave-borne hollow vortices in the absence of gravity.

Once the existence of a traveling wave is known, the natural next question is whether it is stable. Our main result in this direction states roughly that sufficiently small-amplitude capillary-gravity wave-borne point vortices are \emph{conditionally orbitally stable}. Orbital here means roughly that stability is modulo translation in the $x_1$-direction, which is natural given that the system is invariant with respect to this symmetry group. As mentioned above, the wave-borne point vortex problem is not known to be globally well-posed in time, and hence any stability result must be conditional on the existence of a (bounded) solution. 

To state this more precisely, it is best to use a reformulation of the problem~\eqref{intro general} in the style of Zakharov--Craig--Sulem, which results in a nonlocal nonlinear system on a fixed domain. That is, we suppose that $\fluidS_t$ is the graph of a free surface profile $\eta = \eta(t, x_1)$, which is of course true for the solutions given by Theorem~\ref{intro existence point vortex theorem} for $\beta$ sufficiently small. The velocity field in a neighborhood of $\fluidS_t$ can be (uniquely) decomposed as
\be
\label{intro hodge helmholtz}
	\mathbf{v} = \nabla \Phi(t) + \epsilon \nabla \Theta(\placeholder; \, \mathbf{z}(t)),
\ee
where $\Phi = \Phi(t, \placeholder) \colon \overline{\fluidD_t} \to \mathbb{R}$ is harmonic in the interior of $\overline{\fluidD_t}$, and $\Theta$ is an explicit function corresponding to the contribution of the point vortex. As a new unknown, we consider the restriction of the potential $\Phi$ to the free surface:
\[
	\varphi = \varphi(t,x_1) \colonequals \Phi(t, x_1, \eta(t, x_1)).
\]
It is known from essentially the work of Rouhi and Wright~\cite{rouhi1993hamiltonian}, Varholm, Wahlén, and Walsh~\cite{varholm2020stability}, and Varholm~\cite{varholm2016solitary} that the capillary gravity wave-borne point vortex problem~\eqref{intro general} can be written as a Hamiltonian system for the state variables $u\colonequals(\eta, \varphi, \mathbf{z})$:
\begin{equation}\label{eq:general-Hamiltonian-formulation}
	\frac{du}{dt}=J(u)DE(u),
\end{equation} 
with respect to a state-dependent Poisson map $J(u)$, where $E=E(u)$ is the energy functional. There is also a conserved quantity, the horizontal linear momentum $\mom = \mom(u)$, which is generated in an appropriate sense by the translation invariance of the system in the horizontal direction. 

While we postpone many of the details, in order to understand the functional analytic setting, it is important to discuss one feature of the energy. The contribution to the kinetic energy from the irrotational part of the velocity field is given by
\be
\label{irrotational kinetic energy term}
	\frac{1}{2} \int_{\fluidD_t} |\nabla \Phi|^2 \, dx = \frac{1}{2} \int_{\mathbb{R}} \varphi G(\eta) \varphi \, dx_1, 
\ee
where $G(\eta)$ is the Dirichlet--Neumann operator associated to $\fluidD_t$; see Section~\ref{hamiltonian formulation ww problem section}. In particular, $G(0)$ is the Fourier multiplier $|\partial_{x_1} | \tanh{|\partial_{x_1}|}$. Note that in the infinite-depth case, $G(0) = |\partial_{x_1}|$, and so purely irrotational kinetic energy~\eqref{irrotational kinetic energy term} is finite so long as $\varphi$ is in the homogeneous Sobolev space $\dot H^{1/2}(\mathbb{R})$. Here, however, we will need to work with $\varphi$ in the more exotic space $\Xspace_2$ which can be realized as the closure of Schwartz functions with respect to the norm
\begin{equation*}
	\n{\placeholder}_{\Xspace_2} \colonequals \| \sqrt{|\partial_{x_1} | \tanh{|\partial_{x_1}|}} \placeholder \|_{L^2(\mathbb{R})}.
\end{equation*}
Adapting the analysis to this space is one of the main technical challenges for the finite-depth case relative to the infinite-depth case considered in~\cite{varholm2020stability}.

That said, our main result on stability is the following.

\begin{theorem}[Stability]
\label{intro stability theorem}
Let $\eta^{\beta,F,\Bs},\varphi^{\beta,F,\Bs},\mathbf{z}^{\beta,F,\Bs})\in\cm_\loc$ be a solitary wave carrying a point vortex constructed in Theorem~\ref{intro existence point vortex theorem}. Then, $(\eta^{\beta,F,\Bs},\varphi^{\beta,F,\Bs},\mathbf{z}^{\beta,F,\Bs})\in\cm_\loc$ is conditionally orbitally stable in the following sense: for any $R,r>0$, there exists $r_0>0$ such that if $u = (\eta,\varphi,\mathbf{z})$ is a solution of~\eqref{eq:general-Hamiltonian-formulation} on $[0,t_0)$ satisfying
\be
\label{intro tubular bound}
	\sup_{t \in [0,t_0)}  \left( \| (\eta(t) \|_{H^{3+}} + \| \varphi(t) \|_{\dot H^{\frac{5}{2}+} \cap \Xspace_2} + |\bfz_2(t)| \right)  < R, 
\ee
with initial data satisfying
\be
\label{intro initial data bound}
	\| \eta(0) - \eta^{\beta,F,\Bs} \|_{H^1} + \| \varphi(0) - \varphi^{\beta,F,\Bs} \|_{\Xspace_2} + |\mathbf{z}(0) - \mathbf{z}^{\beta,F,\Bs} | < r_0,
\ee
then 
\be
\label{intro orbital stability}
\begin{aligned}
	\sup_{t \in [0,t_0)} \inf_{s \in \mathbb{R}} & \Big( \| \eta(t, \placeholder -s) - \eta^{\beta,F,\Bs} \|_{H^1} + \| \varphi(t, \placeholder -s) - \varphi^{\beta,F,\Bs} \|_{\dot H^{\frac{1}{2}} \cap \Xspace_2} \\
	& \qquad\qquad + |\mathbf{z}(t) + s \mathbf{e}_1 - \mathbf{z}^{\beta,F,\Bs} | \Big)< r.
\end{aligned}
\ee
\end{theorem}

The classical approach to proving orbital stability or instability for Hamiltonian systems in the presence of symmetry is the GSS method of Grillakis, Shatah and Strauss~\cite{grillakis1987stability1}. That machinery, unfortunately, does not apply to the wave-borne point vortex problem for a number of reasons: the Poisson map $J=J(u)$ is neither surjective nor independent of the state $u$, and the energy, smoothness, and well-posedness spaces for capillary-gravity water waves do not coincide because the problem is quasilinear. We therefore use a relaxed version of GSS developed in~\cite{varholm2020stability} to treat the deep water wave-borne point vortex problem. 

One feature of that theory is that it is formulated on a scale of spaces $\Wspace\hookrightarrow\Vspace\hookrightarrow\Xspace$ rather than a single space, with stability made conditional on an a priori bound in $\Wspace$. This scale is the reason behind the varying norms in the statement of Theorem~\ref{intro stability theorem}. In particular, the assumption in~\eqref{intro tubular bound} amounts to asking that $\sup_{t \in [0,t_0)} \| u(t) \|_{\Wspace} < R$, whereas the norms occurring in~\eqref{intro initial data bound} and \eqref{intro orbital stability} are at the level of the energy space $\mathbb{X}$.

\subsection{Plan of the article}

Section~\ref{steady point vortex section} is dedicated to the proof of Theorem~\ref{intro existence point vortex theorem}. It is well-known that in the supercritical Bond number and subcritical Froude number regime ($\Bs > 1/3$, $F^2 < 1$), the linearization of the irrotational water wave problem at a trivial laminar flow has trivial kernel. For the wave-borne point vortex system, we prove that the corresponding linearized operator is a triangular operator matrix. One diagonal entry is identical to that for the irrotational system, which we prove is invertible. The other diagonal entry arises from linearizing the Helmholtz--Kirchhoff model~\eqref{intro steady KH equation}, which is likewise easily seen to be finite-dimensional and invertible. An application of the implicit function theorem completes the argument. Note that this closely parallels the gravity wave case in~\cite{chen2025vortex}, except that with surface tension, we have a Wentzell-type boundary condition.

Next, the proof of Theorem~\ref{intro existence hollow vortex theorem} is carried out in Section~\ref{steady hollow vortex section}. Here, in contrast to~\cite{chen2025vortex}, we do not take a wave-borne point vortex and desingularize it, but rather inject a hollow vortex into the fluid domain through the bed. That is, formally, we take the trivial irrotational wave and look for solutions with a vortex core of approximate radius $0 < \rho \ll 1$ that is $O(\rho)$ above the bed and about which the circulation is $O(\rho)$.  As mentioned above, this scaling is motivated by the need to balance the $O(1/\rho)$ curvature term in the Bernoulli condition~\eqref{intro steady dynamic condition core} with the kinetic energy density.  A quite delicate analysis is required to understand the asymptotics of this system as $\rho \searrow 0$. Ultimately, though, we show that the linearized problem at the trivial solution is indeed invertible, and conclude the existence of nontrivial wave-borne hollow vortices once again by the implicit function theorem.

Finally, Section~\ref{stability section} contains the proof of Theorem~\ref{intro stability theorem}. After recalling the version of the GSS method from~\cite{varholm2020stability}, we confirm that the wave-borne point vortex system has a (noncanonical) Hamiltonian formulation in the spirit of Zakharov--Craig--Sulem, and that sections of the families of traveling waves constructed in Theorem~\ref{thm:existence} can be characterized as constrained minimizers of the energy on level sets of the (linear horizontal) momentum. Following the basic idea of Mielke~\cite{mielke2002energetic}, as well as the earlier work~\cite{varholm2020stability}, we confirm that the Hessian of the augmented Hamiltonian has Morse index $1$ when the vortex is sufficiently close to the bed. In fact, we find that the unique negative eigenvalue of the Hessian is $O(c^2)$, which, in contrast to~\cite{varholm2020stability}, is now $O(1)$ rather than perturbative. Likewise, the moment of instability is confirmed to be strictly convex in this range, which implies the orbital stability of the corresponding bound state.

\subsection*{Notation}

Before beginning, we lay out some notational conventions used throughout the paper.

For $D\subseteq\mathbb{R}^n$ or $\mathbb{C}$, any non-negative integer $\ell$, and any $\alpha\in(0,1)$, let $C^{\ell + \alpha}(D)$ denote the space of H\"older continuous functions of order $\ell$, exponent $\alpha$, and domain $D$. If $D$ is unbounded, let $C_0^{\ell + \alpha}(D)$ denote the subspace of all $u\in C^{\ell + \alpha}(D)$ whose partial derivatives of order at most $\ell$ vanish uniformly at infinity, and $C_\bdd^{\ell+\alpha}(D)$ the subspace of those $u$ whose partial derivatives of order at most $\ell$ are bounded. If $D$ is instead a subset of $\mathbb{C}$, then $\partial_z$ stands for the standard Wirtinger derivative, and $C_{\text{e}}^{\ell + \alpha}(D)$ denotes those functions which are even across the imaginary axis. Likewise, define $C_{0,\text{e}}^{\ell + \alpha}(D) \colonequals C_0^{\ell + \alpha}(D)\cap C_{\text{e}}^{\ell + \alpha}(D)$. We denote by $H^s(D)$ the standard $L^2$-based Sobolev space over the domain $D$ of order $s \in \mathbb{R}$, while the homogeneous version is written $\dot H^s(D)$. In Sections~\ref{steady point vortex section} and~\ref{steady hollow vortex section}, the regularity index $k \geq 2$ and the exponent $\alpha \in (0,1)$ are fixed; Section~\ref{stability section} requires $k$ to be sufficiently large.

We write $\jbracket{\placeholder} \colonequals (1 + |\placeholder|^2)^{1/2}$ for the Japanese bracket and $\nabla^\perp \colonequals (-\partial_{x_2}, \partial_{x_1})$. The notation $A \lesssim B$ means that $A \leq CB$ for a constant $C > 0$ independent of the relevant quantities, and $A \eqsim B$ means that $A \lesssim B \lesssim A$. For Banach spaces $X$ and $Y$, $\Lin(X; Y)$ is the space of bounded linear operators from $X$ to $Y$, and $\Lin(X) \colonequals \Lin(X;X)$; for a (possibly unbounded) linear operator $L$, $\operatorname{Dom}{L}$ and $\operatorname{Rng}{L}$ denote its domain and range. 

We write $\mathbb{T}$ for the complex unit circle. For a complex- or real-valued function $\phi = \phi(\tau)$, with $\tau = e^{i\theta} \in \mathbb{T}$, we define $\partial_\tau \phi = -i e^{-i\theta} \partial_\theta \phi$. This allows one to define the corresponding spaces $C^{k+\alpha}(\mathbb{T})$ and $C^{k+\alpha}(\mathbb{T};\mathbb{R})$ of Hölder continuous functions of order $k \geq 0$ and exponent $\alpha \in (0,1)$ with domain $\mathbb{T}$. Any $\phi \in C^{\alpha}(\mathbb{T})$ admits the Fourier series representation
	$$\phi(\tau) = \sum_{m\in\mathbb{Z}}\hat{\phi}_m\tau^m,$$
where $\tau = e^{i\theta}$, and
	$$\hat{\phi}_m = \frac{1}{2\pi}\int_{\mathbb{T}}\phi(\tau)\tau^{-m}d\theta.$$
When $\phi$ is real-valued, the Fourier modes must satisfy $\hat{\phi}_{-m} = \overline{\hat{\phi}_{m}}$. Denote by $\mathring{C}^{\ell + \alpha}(\mathbb{T})$ the space of all $\phi\in C^{\ell + \alpha}(\mathbb{T})$ with $\hat{\phi}_0 = 0$. For any $m \geq 0$, define the projection $P_m \colon C^{k+\alpha}(\mathbb{T}) \to C^{k+\alpha}(\mathbb{T})$ by
\[
	P_m \phi (\tau) \colonequals \left\{
	\begin{aligned}
		\widehat\phi_m \tau^m + \widehat \phi_{-m} \tau^{-m} & \qquad \textrm{if } m > 0 \\
		\widehat \phi_0 & \qquad \textrm{if } m = 0,
	\end{aligned}
	\right.
\]
and set $P_{\leq m} \colonequals P_0 + \cdots + P_m$, $P_{<m} \colonequals P_{\leq m} - P_m$, and $P_{> m} \colonequals 1- P_{\leq m}$.  

\section{Wave-borne point vortices}
\label{steady point vortex section}

In this section, we prove the existence of a family of steady capillary-gravity wave-borne point vortices. Fix an integer $k \geq 2$ and Hölder exponent $\alpha \in (0,1)$. These will be constant throughout the section; the air-sea interface we construct will ultimately be of class $C^{k+\alpha}$. 

\subsection{Conformal reformulation} 
Identify the Cartesian coordinate $x=(x_1,x_2)\in\R^2$ with the complex number $z=x_1+x_2 i\in\C$. The objective is to find traveling waves carrying a single submerged point vortex located at $\bfz=ib$, for some $-1<b<0$, with nondimensionalized vortex strength $\gamma$. The governing equations~\eqref{intro steady Euler} can be rewritten in the $z$-plane as follows. First, the fact that $\bfu$ is irrotational and incompressible is equivalent to the complex relative velocity field satisfying the Cauchy--Riemann equations:
\begin{equation}\label{eq:complex-irrotational-incompressible}
	\bfu_1 - i\bfu_2 \text{ is holomorphic in } \fluidD \setminus \{\bfz\}.
\end{equation}
The kinematic condition~\eqref{intro steady kinematic condition} on the bed and surface likewise become
\begin{equation}\label{eq:complex-kinematic}
	\bfu_1 + i\bfu_2 \text{ is tangent to } \fluidS \cup \fluidB,
\end{equation}
while the dynamic condition~\eqref{intro steady dynamic condition surface} on the air-sea interface is unchanged.  It is well-known that together~\eqref{intro steady dynamic condition surface} and~\eqref{eq:complex-irrotational-incompressible} imply that the momentum equations in~\eqref{eq:2D-steady-Incompressible-Euler} are satisfied in $\fluidD\setminus\{ \bfz\}$. On the other hand, the Helmholtz--Kirchhoff model for the point vortex motion can be written quite concisely as 
\begin{equation}\label{eq:complex-HK}
	\bfu_1 - i\bfu_2 = \frac{\gamma}{2\pi i}\,\frac{1}{z-ib}+ O(\abs{z - ib}) \qquad \text{as } z \to ib.
\end{equation}
Note that the absence of an $O(1)$ term above is equivalent to the vortex being stationary in the moving frame, meaning that it is carried by the wave.

\begin{figure}
\includegraphics{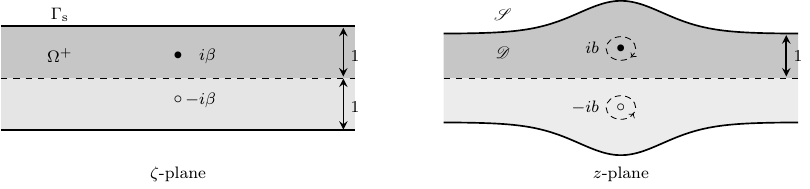}
\caption{Conformal and (doubled) physical domain for the wave-borne point vortex problem.}
\end{figure}

As always, a major challenge in studying water waves of any kind is that determining the fluid domain $\fluidD$ is itself part of the problem. For two-dimensional waves, a classical strategy for  addressing this issue is to view $\fluidD$ as the image of an infinite strip under an a priori unknown conformal mapping $f$, effectively pulling the system back to a fixed canonical domain at the price of making the governing equations there more complicated. Our convention will be that the conformal domain is in the complex $\zeta$-plane.

It will be useful to build in a vertical shift in coordinates so that the bed $\fluidB$ is the image of the real axis in the $\zeta$-plane. This choice permits us to ``double'' the domain via the Schwarz reflection principle, with the kinematic condition on the bed enforced via a symmetry assumption. With that in mind, denote by $\zeta = \zeta_1 + \zeta_2i$ the coordinates in the conformal domain, and set
\begin{equation*}
	\Omega \colonequals \{\zeta \in \C: -1< \zeta_2 < 1\}, \qquad \Omega_+ \colonequals \{\zeta \in \C : 0 < \zeta_2 < 1\}, \qquad \Gamma \colonequals \{\zeta \in \C : \zeta_2 = 1\}.
\end{equation*}
If the air-sea interface $\fluidS$ is $C^{k+\alpha}$ for some $k \geq 2$, then by Kellogg's theorem, there exists a mapping $f\in C^{k+\alpha}(\bar\Omega)$ injective and conformal on $\bar\Omega$ such that $f(\Omega_+)=\fluidD$. We are free to further require that
\begin{equation*}
	\partial_\zeta f(\zeta) \to 1 \qquad \text{as } |\zeta_1| \to \infty.
\end{equation*}
For small-amplitude waves, we expect to have that $f+i$ is a perturbation of the identity. We assume that the fluid domain is even with respect to the imaginary axis, and, as discussed above, ask for even symmetry with respect to the line $z = -i$. These requirements translate to two symmetry conditions on the conformal mapping
\begin{equation}\label{eq:symmetry-condition}
	f(-\zeta)+i=-f(\zeta)-i,\qquad f(\bar\zeta)+i = \overline{f(\zeta)} -i
\end{equation}
for all $\zeta\in\Omega$. In particular, $f+i$ is real-on-real, which forces $f(\R)=\fluidB$. 

Denote by $i\beta$ the pre-image of the point vortex under $f$:
\begin{equation}\label{eq:b-beta-relation}
	ib = f(i\beta)
\end{equation}
The complex relative potential in $\Omega$ then takes the explicit form
\begin{equation}\label{eq:def-W}
	W = W(\zeta;\gamma,\beta) \colonequals \frac{\gamma}{2\pi i}\log\left(\frac{\sinh\left(\frac{\pi}{2}(\zeta - i\beta)\right)}{\sinh\left(\frac{\pi}{2}(\zeta + i\beta)\right)}\right) - \zeta,
\end{equation}
such that any solution of \eqref{eq:complex-irrotational-incompressible}, \eqref{eq:complex-kinematic}, and \eqref{eq:complex-HK} must have complex velocity
\begin{equation}\label{eq:velocity-representation}
	\bfu_1 - i\bfu_2 = \left(\frac{\partial_\zeta W}{\partial_\zeta f}\right)\circ f^{-1};
\end{equation}
see, for example,~\cite{chen2025vortex}. One readily verifies that $\imagpart W$ is constant on $\{\zeta_2 = 0\}$ and $\{\zeta_2 = \pm 1\}$, that $\p_\zeta W$ is meromorphic on $\Omega$ with simple poles at $\zeta = \pm i\beta$ of residues $\pm\frac{\gamma}{2\pi i}$, representing both the physical and phantom vortices, and that $\p_\zeta W\to -1$ as $\abs{\zeta_1}\to\infty$, consistent with the far-field condition $\mathbf{u} \to -\mathbf{e}_1$. In particular, for any conformal mapping $f$ as above, the velocity field given by \eqref{eq:velocity-representation} automatically satisfies the irrotationality and incompressibility condition \eqref{eq:complex-irrotational-incompressible} and the kinematic condition \eqref{eq:complex-kinematic}. It remains to ensure the Bernoulli condition and the Helmholtz--Kirchhoff condition \eqref{eq:complex-HK}. Intuitively, we understand the first term on the right-hand side of~\eqref{eq:def-W} as representing the contribution of the vortex at $i\beta$, a phantom vortex with opposite strength at $-i\beta$, and then periodically extended in the vertical direction to ensure that top and bottom boundaries of $\Omega$ are streamlines. Note that because of this symmetry, we may allow $\beta\in(-1,1)$, which is useful as we are bifurcating from $\beta=0$. For $\beta<0$, the roles of the resulting vortex pair are simply exchanged: the physical vortex is the one in $\Omega_+$, located at $-i\beta$ with the vortex strength $-\gamma$, whose image $-i(2+b)= f(-i\beta)$ is in $\fluidD$.

Since the conformal mapping $f$ can be recovered from its imaginary part, it is convenient to recast the problem in terms of the (real-valued) unknown
\begin{equation*}
	w(\zeta) \colonequals \imagpart(f(\zeta)+i-\zeta) = \imagpart f(\zeta) +1 - \zeta_2.
\end{equation*}
Thus, $w\in C^{k+\alpha}(\bar\Omega)$ is a harmonic function that vanishes at infinity and inherits the symmetry of the conformal mapping $f+i$ from \eqref{eq:symmetry-condition}. We therefore introduce the natural function space for $w$:
\begin{equation}\label{eq:function-space-w}
	\mathscr{W} \colonequals \left\{ w \in C_{0,\even}^{k+\alpha}(\overline\Omega)  : \Delta w = 0 \text{ in } \Omega,~w\textrm{ is odd in } \zeta_2 \right\}.
\end{equation}

Over the conformal domain $\Omega$, the Helmholtz-Kirchhoff condition \eqref{eq:complex-HK} reads
\begin{equation}\label{eq:HK-conformal}
	\frac{\partial_\zeta W}{\partial_\zeta f} - \frac{\gamma}{2\pi i}\,\frac{1}{f - ib} = O(\abs{\zeta - i\beta}) \qquad\text{as } \zeta \to i\beta.
\end{equation}
Using the explicit formula for $W$ in~\eqref{eq:def-W}, and performing a Laurent expansion of $\partial_\zeta W$ at $\zeta = i\beta$, one finds
\begin{equation}\label{eq:W-Laurent-expansion}
	\frac{\partial_\zeta W(\zeta)}{\partial_\zeta f(\zeta)} - \frac{\gamma}{2\pi i}\,\frac{1}{f(\zeta)-ib}
	= \frac{1}{\partial_\zeta f(i\beta)}\left(-1 + \frac{\gamma}{4}\cot(\pi\beta) - \frac{\gamma}{4\pi i}\,\frac{\partial_\zeta^2 f(i\beta)}{\partial_\zeta f(i\beta)}\right) + O(\abs{\zeta - i\beta}),
\end{equation}
Thus, \eqref{eq:HK-conformal} holds if and only if the constant term in \eqref{eq:W-Laurent-expansion} vanishes. By the symmetries of $w$, we have $\p_{\zeta_1}w, \partial_{\zeta_1}\partial_{\zeta_2}w=0$ on the imaginary axis, so the Cauchy--Riemann equations give
\begin{equation*}
	\partial_\zeta f(i\beta) = 1 + \partial_{\zeta_2} w(i\beta) \in \R, \qquad \partial_\zeta^2 f(i\beta) = i\partial_{\zeta_1}^2 w(i\beta).
\end{equation*}
Therefore, the constant term in the Laurent expansion~\eqref{eq:W-Laurent-expansion} vanishes precisely when
\begin{equation*}
	-1 + \frac{\gamma}{4}\cot(\pi\beta) - \frac{\gamma}{4\pi}\frac{\partial_{\zeta_1}^2 w}{1+\partial_{\zeta_2}w}\bigg|_{\zeta=i\beta} = 0.
\end{equation*}
Introducing the auxiliary unknown $\varkappa \in \R$ through the following equation
\begin{equation}\label{eq:point-vortex-advection}
	\frac{\partial_{\zeta_1}^2 w}{1+\partial_{\zeta_2} w}\bigg|_{\zeta=i\beta} = \pi\varkappa,
\end{equation}
we can solve for $\gamma$ and arrive at
\begin{equation}\label{eq:gamma-relation}
	\gamma = \gamma(\varkappa,\beta) \colonequals \frac{4\sin(\pi\beta)}{\cos(\pi\beta)-\varkappa\sin(\pi\beta)}.
\end{equation}
Note that $\varkappa$ being real-valued is necessary for there to exist a real-valued $\gamma$ satisfying the Helmholtz--Kirchhoff equation.

Under the relation \eqref{eq:gamma-relation}, the Helmholtz--Kirchhoff condition is therefore equivalent to \eqref{eq:point-vortex-advection}. Note that the denominator in \eqref{eq:point-vortex-advection} is nonzero due to the fact that the conformal mapping $f$ has non-vanishing derivatives.

Next, we rewrite the Bernoulli condition \eqref{intro steady dynamic condition surface} in conformal variables as
\begin{equation}\label{eq:conformal-Bernoulli}
	-\Bs\,\frac{(1+\partial_{\zeta_2}w)\partial_{\zeta_1}^2 w-(\partial_{\zeta_1} \partial_{\zeta_2}w) \partial_{\zeta_1} w}{\left((\partial_{\zeta_1} w)^2 + (1+\partial_{\zeta_2} w)^2\right)^{3/2}} + \frac12\,\frac{a^2}{(\partial_{\zeta_1}w)^2 + (1+\partial_{\zeta_2} w)^2} + \frac{1}{F^2}w = \frac12 \qquad\text{on } \Gamma,
\end{equation}
where the function $a=a(\zeta_1;\varkappa,\beta)$ is given by
\begin{equation}\label{eq:def-a}
	a = a(\zeta_1;\varkappa,\beta) \colonequals \partial_\zeta W(\zeta_1 + i;\gamma(\varkappa,\beta),\beta) = -1 - \frac{\gamma}{2}\,\frac{\sin(\pi\beta)}{\cosh(\pi\zeta_1)+\cos(\pi\beta)}.
\end{equation}

The argument above shows that the steady wave-borne point vortex problem~\eqref{intro steady Euler} can be recast in terms of the unknowns $(w,\varkappa)$, which must satisfy \eqref{eq:point-vortex-advection} and \eqref{eq:conformal-Bernoulli}, the pair $(w,\varkappa)$ serves as the unknown, while the conformal altitude $\beta$, the Froude number $F$, and air-sea interface Bond number $\Bs$ serve as parameters. In order for solutions to be conformal problem to give rise to physical solutions, we require in addition that the resulting mapping $f$ is indeed injective and conformal on $\overline\Omega$.

\subsection{Existence of steady wave-borne point vortices}

We now proceed to give the proof of Theorem~\ref{intro existence point vortex theorem}, which is done through the implicit function theorem applied at the trivial solution $(w,\varkappa,\beta) = (0,0,0)$. 
 Towards that end, we reformulate the system one last time as an abstract operator equation. The ambient space for the domain of this operator we take to be 
\[
	\scrX \colonequals \left( \scrW \cap H^{k+\frac12}(\Omega) \right) \times \R
\]
Note that additional Sobolev regularity of $w$ is built into the definition, as this will be necessary when studying orbital stability in Section~\ref{stability section}. We work in the open set
\begin{equation}\label{eq:def-O}
	\mathscr{O} \colonequals \left\{(w,\varkappa,\beta,F,\Bs) \in \scrX\times\R^3 : \qquad 
	\begin{aligned}
	&(\partial_{\zeta_1} w)^2+(1+\partial_{\zeta_2}w)^2 > 0 \text{ in } \overline{\Omega},\\
	&\abs{\beta} < 1, \quad 0<F<1, \quad \Bs > \tfrac{1}{3}, \\
	&\cos(\pi\beta) > \varkappa\sin(\pi\beta)
	\end{aligned}\right\},
\end{equation}
membership in which ensures that the mapping $f$ is conformal and that $\gamma(\varkappa,\beta)$ is well-defined by \eqref{eq:gamma-relation} and that we remain in the same connected component of the domain of $\gamma$ as the trivial solution. In particular, we have
\begin{equation*}
	\gamma(\varkappa,\beta)\sin(\pi\beta) \geq 0,
\end{equation*}
with equality if and only if $\beta = 0$. Consequently, the coefficient $a(\,\cdot\,;\varkappa,\beta) \leq -1$ for all $(w,\varkappa,\beta,F,\Bs) \in \mathscr{O}$. Note also that for $w$ sufficiently small, $f$ is necessarily injective on $\overline{\Omega}$, and hence the conformal reformulation is equivalent to the original physical problem.

We may now rewrite the steady wave-borne point vortex problem as
\[
	\mathscr{F}(w,\varkappa;\beta,F,\Bs) = 0,
\]
where $\mathscr{F} \colonequals (\mathscr{F}_1,\mathscr{F}_2) \colon \mathscr{O} \to \mathscr{Y}$ is given by
\[
	\begin{split}
	& \mathscr{F}_1(w,\varkappa;\beta,F,\Bs) \\
	& \qquad \colonequals \left(-\Bs\,\frac{(1+\partial_{\zeta_2}w)\partial_{\zeta_1}^2 w-(\partial_{\zeta_1} \partial_{\zeta_2}w) \partial_{\zeta_1} w}{\left((\partial_{\zeta_1} w)^2 + (1+\partial_{\zeta_2} w)^2\right)^{3/2}} + \frac12\,\frac{a(\placeholder; \varkappa,\beta)^2}{(\partial_{\zeta_1}w)^2 + (1+\partial_{\zeta_2} w)^2} + \frac{1}{F^2}w - \frac12\right)\bigg|_{\Gamma}\\
	& \mathscr{F}_2(w,\varkappa;\beta,F,\Bs) \colonequals \frac{1}{\pi}\,\frac{\partial_{\zeta_1}^2 w}{1+\partial_{\zeta_2}w}\bigg|_{\zeta=i\beta} - \varkappa,
	\end{split}
\]
with the codomain being
\begin{equation*}
	\mathscr{Y} \colonequals \left(C_{0,\even}^{k-2+\alpha}(\Gamma) \cap H^{k-2}(\Gamma)\right)\times\R.
\end{equation*}
Clearly, $\F$ is a real-analytic mapping, and $\mathscr{F}(0,0; 0,F,\Bs) = 0$ for any choice of $F$ and $\Bs$. The next theorem is equivalent to Theorem~\ref{intro existence point vortex theorem} but stated in terms of the conformal variables.

\begin{theorem}[Capillary-gravity wave-borne point vortices] \label{thm:existence}
For any subcritical Froude number $0<F_0<1$ and any supercritical Bond number $\Bs^0>\tfrac{1}{3}$, there exists a three-dimensional real-analytic manifold $\cmconf_\loc$ of solitary capillary-gravity water waves with a submerged point vortex such that the following holds.
\begin{enumerate}[label = \rm(\alph*)]

\item In a neighborhood of $(0,0;0,F_0,\Bs^0)$ in $\mathscr{X}\times\mathbb{R}^3$, $\cmconf_\loc$ comprises the entire zero-set of $\mathscr{F}$.

\item The manifold admits the real-analytic parameterization
	\begin{equation}
		\cmconf_{\mathrm{loc}} = \{(w^{\beta,F,\Bs},\varkappa^{\beta,F,\Bs};\beta,F,\Bs) : 
    |\beta|+|F-F_0|+|\Bs-\Bs^0|\ll 1\} \subset \mathscr{O}
	\end{equation}
where 
\[
	\left(w^{0,F_0,\Bs^0},\varkappa^{0,F_0,\Bs^0}\right)=(0,0), \qquad \left(w^{-\beta,F,\Bs},\varkappa^{-\beta,F,\Bs}\right)=\left(w^{\beta,F,\Bs},-\varkappa^{\beta,F,\Bs}\right).
\]
\item The solutions on $\cmconf_\loc$ have the asymptotic expansions
\begin{equation}\label{eq:w-kappa-asymptotics}
    \begin{aligned}
        w^{\beta,F,\Bs} &=  \frac{1}{2}\ddot{w}^{F,\Bs} \beta^2 + O(\beta^4) \qquad \textrm{in } \scrW \cap H^{k+\tfrac{1}{2}}(\Omega) \\
        \varkappa^{\beta,F,\Bs} &=  \frac{1}{2\pi}
        (\partial_{\zeta_1}^2 \partial_{\zeta_2}\ddot{w}^{F,\Bs})(0) \beta^3 
        + O(\beta^5),
    \end{aligned}
    \end{equation}
    where $\ddot{w}^{F,\Bs}$ is the unique solution of
	\begin{equation*}
	\left\{\begin{aligned}
	\Delta \ddot{w}^{F,\Bs}=0&\quad\text{in }\Omega\\
	\ddot{w}^{F,\Bs}=0&\quad\textrm{on }\{{\zeta_2} = 0\}\\
	\Bs \partial_{\zeta_1}^2 \ddot{w}^{F,\Bs}+\partial_{\zeta_2} \ddot{w}^{F,\Bs} - \frac{1}{F^2}\ddot{w}^{F,\Bs}=2\pi^2\sech^2{\left(\frac{\pi}{2}{\zeta_1}\right)}&\quad\textrm{on }\Gamma.\\
	\end{aligned}\right.
	\end{equation*}
\item The nondimensionalized vortex strength $\gamma$ and the vertical coordinate of the vortex center $b$ satisfy
	\begin{align}
		\gamma^{\beta,F,\Bs} &= 4\pi\beta + \frac43\pi^3\beta^3 + O(\beta^4),\label{eq:gamma-asymptotics}\\
		b^{\beta,F,\Bs} &= -1+\beta +\frac12(\p_{\zeta_2}\ddot{w}^{F,\Bs})(0)\beta^3 + O(\beta^4).\label{eq:b-asymptotics}
	\end{align}
\end{enumerate}	
	\end{theorem}

\begin{proof}
A direct calculation shows that the linearization of $\mathscr{F}$ at the trivial solution is
\begin{equation}
	D_{(w,\varkappa)}\mathscr{F}(0,0;0,F_0,\Bs^0)=\begin{pmatrix}
		\left(-\Bs^0\partial^2_{{\zeta_1}}-\partial_{{\zeta_2}}+\frac{1}{F_0^2}\right)|_\Gamma&0\\
		\frac{1}{\pi}\partial^2_{{\zeta_1}}\big|_{\zeta=0}&-1\\
	\end{pmatrix}.
\end{equation}
Note that the lower left entry is bounded, the lower right entry is invertible, and the upper left entry can be seen as $m(D)\circ\tr$, where $\tr\colon\scrW \cap H^{k+\frac12}(\Omega)\to C_{0,\even}^{k+\alpha}(\Gamma) \cap H^{k}(\Gamma)$ is the trace operator on the boundary $\Gamma$ and $m(D)$ is the Fourier multiplier with symbol
	\begin{equation}
		m(\xi)\colonequals \Bs^0 |\xi|^2-|\xi|\coth{|\xi|}+\frac{1}{F_0^2}.
	\end{equation}
The second term comes from the Dirichlet--Neumann operator on the strip $\Omega_+$, keeping in mind that odd symmetry in $\zeta_2$ effectively imposes a homogeneous Dirichlet condition on the real axis.

It is routine to check that $m$ is in the symbol class $S^2$, that is 
\[
	\abs{\partial_\xi^\ell m(\xi)} \lesssim_\ell (1+\abs{\xi})^{2-\ell} \qquad \textrm{for all } \ell \geq 1, ~\xi \in \mathbb{R}.
\]
From the elementary inequality $|\xi|\coth{|\xi|} \leq 1 + |\xi|^2/3$, it follows that
\begin{equation}\label{eq:m-lower-bound}
	m(\xi) \;\geq\; \left(\Bs^0-\tfrac13\right)\xi^2 + \frac{1}{F_0^2}-1 \gtrsim 1+\xi^2,
\end{equation}
where we have used the assumptions on $\Bs^0$ and $F_0$. Thus, $m$ is uniformly positive, and, therefore, the Fourier multiplier $m(D)^{-1}$ with symbol $m(\xi)^{-1}$ is well-defined.  From the lower bound \eqref{eq:m-lower-bound} and repeated applications of the Leibniz rule, one can readily verify that $m(D)^{-1} \in S^{-2}$. By the boundedness of $S^{-2}$ multipliers on Besov spaces --- and hence Sobolev and H\"older spaces --- we have that
\[
	m(D) \colon C^{k+\alpha}_{0,\even}(\Gamma)\cap H^{k}(\Gamma)\to C^{k-2+\alpha}_{0,\even}(\Gamma)\cap H^{k-2}(\Gamma) \quad \textrm{is an isomorphism.}
\]
Also note that $m(\xi)$ is even and hence $m(D)$ and $m(D)^{-1}$ preserve evenness in $\zeta_1$.

The trace operator $\tr\colon \scrW\cap H^{k+\frac{1}{2}}(\Omega) \to C^{k+\alpha}_{0,\even}(\Gamma) \cap H^{k}(\Gamma)$ is also an isomorphism with the inverse is given by harmonic extension into $\Omega_+$ and then odd reflection over $\{\zeta_2 = 0\}$. 

Hence, an application of the (real-analytic) implicit function at the trivial solution $(0,0;0,F_0,\Bs^0)$ shows the existence of $\cmconf_\loc$. Finally, the asymptotic expansions in \eqref{eq:w-kappa-asymptotics} can be founded by implicit differentiation, while the asymptotic expansions in~\eqref{eq:gamma-asymptotics} and \eqref{eq:b-asymptotics} are obtained by differentiating \eqref{eq:gamma-relation} and \eqref{eq:b-beta-relation}.
\end{proof}

\section{Steady wave-borne hollow vortices}
\label{steady hollow vortex section}

In this section, we turn to the proof of Theorem~\ref{intro existence hollow vortex theorem} on the existence of wave-borne vortex patches. The basic strategy is to repeat the vortex injection construction from the previous section, but simultaneously desingularize the point vortex into a hollow vortex using the general approach of~\cite{chen2025vortex,chen2026desingularization}. As discussed above, a new scaling and considerable care is needed to address the singularity arising from the curvature term in the dynamic condition when $\Bv > 0$. We therefore focus on this case, and outline how $\Bv = 0$ is treated only at the end of the section.

\subsection{Reformulations}
We begin by reformulating the problem, first as a local equation in terms of conformal variables, then as a nonlocal problem that encapsulates both the wave-borne point vortex and hollow vortex problems. Throughout the remainder of the section, we fix a regularity index $k \geq 2$ and Hölder exponent $\alpha \in (0,1)$.

\subsubsection*{Complex variable formulation}
Suppose there is a single hollow vortex transported by the wave, meaning that the vortex core $\fluidV \subset\subset \fluidD$ is a nonempty, simply connected open set. We again identify the real variable $x \in \mathbb{R}^2$ with $z = z_1 + i z_2 \in \mathbb{C}$ in the obvious way. As we saw for the wave-borne point vortex problem, the incompressibility and irrotationality of $\mathbf{u}$ in $\fluidD$ can be expressed in terms of the complex velocity field $u_1 - i u_2$ as the requirement that 
\begin{equation}\label{hollow vortex conformal holomorphic}
	u_1 - i u_2 \quad\text{is holomorphic in}\quad\mathscr{D}\setminus \fluidV,\\
\end{equation}
and the kinematic boundary condition~\eqref{intro steady kinematic condition} is
\begin{equation}\label{hollow vortex conformal kinematic}
	u_1 + i u_2 \quad\text{is purely tangential on}\quad\mathscr{S}\cup\mathscr{B}\cup\partial\mathscr{V}.\\
\end{equation}
Recall also that the Bernoulli boundary condition is imposed on both the air-water interface \eqref{intro steady dynamic condition surface} and vortex boundary \eqref{intro steady dynamic condition core}. In the conformal $z$-plane, these take the form
\begin{equation}\label{conformal Bernoulli condition local}
\left\{
\begin{aligned}
	\Bs \kappa + \frac{1}{2}\left(u_1^2+u_2^2 \right) + \frac{1}{F^2}y & = \frac{1}{2} & \qquad &  \textrm{on } \fluidS\\
	\Bv \kappa + \frac{1}{2}\left(u_1^2+u_2^2\right) + \frac{1}{F^2}y & = q & \qquad & \textrm{on } \partial\fluidV,
\end{aligned}
\right.
\end{equation}
Finally, the requirement in~\eqref{intro steady vorticity assumptions} that the circulation about $\fluidV$ is $\gamma$ can be written as
\begin{equation}
\label{hollow complex circulation condition}
	\gamma = \int_{\partial\fluidV}(u_1 - iu_2 ) \, dz.\\
\end{equation}

\begin{figure}
	\centering
	\includegraphics[width=0.9\linewidth]{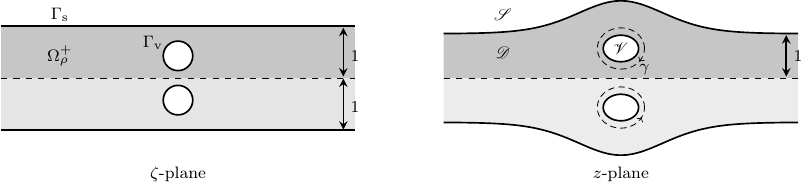}
	\caption{The conformal domain in the $\zeta$-plane is depicted on the right, with the corresponding (doubled) physical domain in the complex $z$-pane depicted on the right..}
\end{figure} 

\subsubsection*{Nonlocal formulation}

Following the same general strategy as for wave-borne point vortices, the next step is to pull the system back to a canonical domain in the complex $\zeta$-plane using an a priori unknown conformal mapping $f = f(\zeta)$. For wave-borne hollow vortices, however, the domain is necessarily multiply connected. We expect that for small hollow vortices, the vortex boundaries will be approximately circular, as this agrees with the streamline pattern near the starting point vortex. It is therefore natural to take 
	\begin{equation*}
	 \Omega _{\rho,\beta} \colonequals 
	\begin{cases}
		\Omega\setminus\{i\beta,-i\beta\},&\rho = 0\\
		\Omega\setminus\overline{B_{|\rho|}(i\beta)\cup B_{|\rho|}(-i\beta)},& 0 < |\rho| \ll 1,
	\end{cases}
	\end{equation*}	
for $0 < |\rho| \ll 1$ and $\beta \in [-1,1]$, where $\Omega$ is, as before, the doubled strip 
\[
	\Omega \colonequals \{ \zeta \in \mathbb{C} : |\realpart{\zeta} | < 1 \}.
\]
We call $\rho$ the conformal radius of the vortex core and $\beta$ its conformal altitude.  The physical domain will then be $\fluidD = f(\Omega_{\rho,\beta}^+)$, with $\Omega_{\rho,\beta}^+ = \Omega_{\rho,\beta} \cap \{ \realpart{\zeta} > 0 \}$. Requiring that $f+i$ is real-on-real ensures that the $\zeta_1$ axis is mapped to the bed $\fluidB$.  The surface $\fluidS$ is likewise the image under $f$ of $\Gammas \colonequals  \{\eta = 1\}$, while the boundary of the upper vortex will be given by $\partial\mathscr{V} = f(\Gammav)$, where $\Gammav \colonequals  \partial B_\rho(i\beta)$. 

In fact, the conformal radius $\rho$ will serve as the bifurcation parameter in the desingularization process, which is why we allow it to be negative. We must have that $|\beta| > |\rho|$ in order to have a physical solution for which the vortex core and phantom vortex core are not in contact. We will therefore mandate that $\beta$ is itself dependent on $\rho$ via
\[
	\beta = \lambda \rho,
\]
where $\lambda > 1$ is fixed. For brevity, let $\Omega_\rho$ denote the corresponding conformal domain $\Omega_{\rho, \lambda\rho}$.

The canonical domain $\Omega_{\rho}$ is still $\rho$-dependent, and hence not an appropriate setting for the bifurcation argument. Following the general approach in~\cite{chen2025vortex}, we therefore reformulate once more in terms of (normalized) boundary traces, which will ultimately yield a nonlocal problem defined on a fixed domain. 
 
The conformal mapping $f$ is constructed using the ansatz
\be
\label{hollow f ansatz}
	f = -i + \id + \fs + \fv,
\ee
where $\fs \in C^{k + \alpha}(\overline{\Omega};\mathbb{C})$ is holomorphic in the entire strip and accounts for the shape of the air-water interface, and $\fv \in C^{k+\alpha}(\overline{\Omega_\rho}; \mathbb{C})$ is a conformal mapping that gives the geometry of the vortex boundaries. Just like we saw in Section~\ref{steady point vortex section}, the mapping $\fs$ can be reconstructed from the trace of its imaginary part on the upper boundary $\Gammas$. Similarly, one can express $\fv$ in terms of its trace on $\Gammav$ using a layer-potential operator.  However, because we are varying $\rho$, and hence $\Gammav$, it is more convenient to work with objects defined on the topologically equivalent set $\mathbb{T}$. With that in mind, we set
	$$f^{\text{v}}(\zeta) \colonequals  \rho^2\mathcal{Z}^{\rho}[\mu](\zeta),$$
where $\mu \in C^{k+\alpha}(\mathbb{T})$ is real-valued density and $\mathcal{Z}^{\rho}$ is the layer potential
\be
\label{definition layer-potential operator}
	\mathcal{Z}^{\rho}[\mu](\zeta) \colonequals  \frac{1}{2\pi i}\int_{\mathbb{T}}\left(\frac{\mu(\sigma)}{\rho\sigma + i\beta - \zeta} + \frac{\mu(\overline{\sigma})}{\rho\sigma - i\beta - \zeta}\right)\rho d\sigma.
\ee
Note that 
\[
	\rho \mapsto \mathcal{Z}^\rho \in \Lin{(C^{\ell+\alpha}(\mathbb{T};\mathbb{C})}, \, C^{k+\alpha}(\overline{\Omega_\rho};\mathbb{C}))
\]
is real-analytic in a neighborhood of $0$.  

As in the point vortex problem, requiring that $i+f$ is real-on-real and imaginary-on-imaginary ensures that the real axis is mapped to the bed and the domain has the desired even symmetry with respect to the $z_2$-axis. We therefore as that both $\fs$ and $\fv$ are likewise real-on-real and imaginary-on-imaginary. For $\fs$, this can be accomplished by working with the new unknown $\ws \colonequals \imagpart{\fs}$, which must then satisfy $\ws \in C_{0,\even}^{k+\alpha}(\overline{\Omega})$. To enforce the desired symmetry properties for $\fv$ requires restricting the densities to the a subspace. With that in mind, we define for any $\ell \geq 0$, the spaces
\be
\label{definition Cii Cri}
\begin{aligned}
	C^{\ell+\alpha}_{\imagimag}(\mathbb{T}) & \colonequals  \{\varphi \in C^{\ell + \alpha}(\mathbb{T};\mathbb{R}): i^{m}\hat{\varphi}_m\in i\mathbb{R}\text{ for all }m\in\mathbb{Z}\} \\
	C^{\ell+\alpha}_{\realimag}(\mathbb{T}) & \colonequals  \{\varphi \in C^{\ell + \alpha}(\mathbb{T};\mathbb{R}): i^{m}\hat{\varphi}_m\in \mathbb{R}\text{ for all }m\in\mathbb{Z}\}. 
\end{aligned}
\ee
When the density $\mu$ belongs to $C_{\imagimag}^{k+\alpha}$, then the corresponding mapping $\fv$ is real-on-real and imaginary-on-imaginary; see~\cite[ Section 5]{chen2025vortex}.  We additionally make the normalizing assumption that $\mu \in \mathring{C}^{k+\alpha}(\mathbb{T})$, that is $\hat\mu_0 = 0$. As we verify later, the nonlinear operator representing the dynamic condition on $\Gammav$ will eventually take values in $C_{\realimag}^\alpha$.

The pullback of the relative velocity field to $\Omega_\rho$ can be expressed quite readily through a complex potential. The leading-order part is given by the corresponding relative potential for the point vortex problem. For a cotranslating pair of point vortices in a channel with centers at $\pm i \beta$ and with corresponding strengths $\pm \gamma$, this takes the form
\be
\label{hollow W0 formula}
	W_0 = W_0(\zeta;\gamma,\beta) \colonequals  \frac{\gamma}{2\pi i}\log\left(\frac{\sinh\left(\frac{\pi}{2}(\zeta - i\beta)\right)}{\sinh\left(\frac{\pi}{2}(\zeta + i\beta)\right)}\right) -\zeta.
\ee

The lemma below details the next higher-order terms in the expansion of the complex relative potential on $\Omega_\rho$, which will be needed later in the desingularization argument. It is reprised from \cite[Lemma 5.1]{chen2025vortex}. Note that the relative potential is agnostic to the effects of surface tension, and hence the result follows exactly as in the gravity wave case. 

\begin{lemma}[Complex relative potential]
\label{hollow complex potential lemma}
There exists real-analytic mappings 
\[
	(\gamma_1, \rho) \mapsto \nu(\placeholder; \gamma_1, \rho) \in \mathring{C}_{\imagimag}^{k+\alpha}(\mathbb{T}; \mathbb{R}) \quad \textrm{and} \quad
	(\gamma_1, \rho) \mapsto \Ws(\placeholder; \gamma_1, \rho) \in C_0^{k+\alpha}(\overline{\Omega}; \mathbb{C})
\]
defined in a neighborhood of $(0,0) \in \mathbb{R}^2$ such that the complex relative potential
 \be
 \label{hollow vortex complex potential form}
 	W^\rho(\placeholder; \gamma_1) \colonequals W_0(\placeholder; \rho \gamma_1,\rho\lambda) + \rho\mathcal{Z}^{\rho}[\nu] + \rho^2 \Ws
\ee
satisfies for all $|\rho| > 0$, 
\begin{enumerate}[label=\rm(\alph*)]
	\item $\partial_\zeta W^\rho$ is holomorphic on $\Omega_\rho$ with $\partial_\zeta W^\rho \in C_\bdd^{k-1 + \alpha}(\overline{\Omega_{\rho}})$.
	\item $W^\rho - W_0$ is single-valued and $\imagpart{W^\rho}$ is constant on each component of $\partial\Omega_{\rho}$.
	\item $\partial_\zeta W^\rho(\zeta) \to -1$ as $\zeta\to \pm\infty$ in $\Omega_{\rho}$.
	\item $\partial_\zeta W^\rho$ is real on real and real on imaginary.
\end{enumerate}
\end{lemma}

One should understand the ansatz in~\eqref{hollow vortex complex potential form} as saying that $W^\rho$ is given by $W_0$ with a higher-order correction (represented by the layer-potential term) for the shape of the vortex, and another higher-order correction (represented by $\Ws$) for the deformation of the air-water interface.

Furthermore, in order avoid the singularity that would occur at $\rho = 0$ in $W^0$ for arbitrary vortex strength, we insist that $\gamma$ must be of the form $\gamma = \gamma_1\rho$, and we will henceforth treat $\gamma_1$ as the unknown.

Thanks to Lemma~\ref{hollow complex potential lemma}, we have that for any conformal mapping $f$ of the form~\eqref{hollow f ansatz} with the above stated symmetry properties, using the complex potential~\eqref{hollow vortex complex potential form} ensures that the corresponding relative velocity field is irrotational and incompressible~\eqref{hollow vortex conformal holomorphic}, and it satisfies the kinematic condition~\eqref{hollow vortex conformal kinematic} as well as the circulation condition~\eqref{hollow complex circulation condition}. Thus, all that remains in order to have a solution of the wave-borne hollow vortex problem is the dynamic condition~\eqref{conformal Bernoulli condition local}. 

With that in mind, we introduce notation for the  restriction to $\Gammas$ and $\Gammav$ of $\partial_\zeta W^\rho$: for each $\rho \ll 1$, define
\begin{equation*}
	\begin{aligned}
	&a_{\text{s}} = a_{\text{s}}(\xi_1;\gamma_1,\rho) \colonequals  |\partial_\zeta W^\rho(\xi_1 + i;\gamma_1)| \in C_\bdd^{k-1+\alpha}(\mathbb{R}), \\
	&a_{\text{v}} = a_{\text{s}}(\tau;\gamma_1,\rho)  \colonequals  |\partial_\zeta W^\rho(\rho\tau + i\beta;\gamma_1)| \in C^{k-1+\alpha}(\mathbb{T}). \\
	\end{aligned}
\end{equation*}
Likewise, to evaluate $\partial_\zeta f$ on the free boundary, we make use of the layer-potential operators
	\begin{equation*}
	\begin{split}
	\mathcal{Z}^{\rho}_{\text{s}}\mu&\colonequals  \mathcal{Z}^{\rho}[\mu](\placeholder + i) \in C_\bdd^{1+\alpha}(\mathbb{R}; \mathbb{C}) \\
	\mathcal{Z}^{\rho}_{\text{v}}\mu &\colonequals  \mathcal{Z}^{\rho}[\mu](\rho\placeholder + i\rho\lambda) \in C^{1+\alpha}(\mathbb{T}; \mathbb{C}).
	\end{split}
	\end{equation*}
It is important to note that the layer-potential operator $\mathcal{Z}^\rho$ and the restricted layer-potential operators commute differently with derivatives. In particular,
\[
	\partial_\zeta \mathcal{Z}^\rho\mu = \frac{1}{\rho} \mathcal{Z}^\rho \mu^\prime, \qquad \partial_{\zeta_1} \mathcal{Z}_{\mathrm{s}}^\rho\mu = \mathcal{Z}_{\mathrm{s}}^\rho\mu^\prime , \qquad \partial_\tau \mathcal{Z}_{\mathrm{v}}^\rho\mu = \mathcal{Z}_{\mathrm{v}}^\rho\mu^\prime,
\]	
as can be easily verified from~\eqref{definition layer-potential operator}.

Then, the dynamic condition~\eqref{conformal Bernoulli condition local} pulled back to the conformal domain takes the form
	\begin{equation}
	\label{hollow nonlocal equation}
	\left\{
	\begin{aligned}
	\frac{1}{2}\frac{\abs{a_{\text{s}}}^2}{\abs{\partial_\zeta f}^2} + \Bs\Ks(w,\mu,\rho) + \frac{1}{F^2}\imagpart{f} &= \frac12&\qquad &\text{on } \Gammas \\
	\frac{1}{2}\frac{\abs{a_{\text{v}}}^2}{\abs{\partial_\zeta f}^2} + \Bv\Kv(w, \mu, \rho) + \frac{1}{F^2}\imagpart{f} & = q& \qquad & \text{on } \Gammav,\\
	\end{aligned} \right.
	\end{equation}
where for brevity we are writing $f$, which is now viewed as a function of $(\mu, w, \rho)$. Similarly, $\Ks$ and $\Kv$ correspond to the curvature along the air-water interface and vortex core boundary, respectively, which are the (nonlinear) functions of $(\mu,\rho)$ given by
\be
\label{definition Ks Kv}
\begin{aligned}
 	\Ks(w,\mu,\rho) & = \Ks(w,\mu,\rho)(\zeta_1) \colonequals \frac{\imagpart{\left( \partial_\zeta f(\zeta_1+i)  \overline{\partial_\zeta^2 f(\zeta_1+i} )\right)}}{|\partial_\zeta f(\zeta_1+i)|^3} \\
	\Kv(w,\mu,\rho) & = \Kv(w,\mu,\rho)(\tau) \colonequals \frac{1}{\rho\abs{ \partial_\zeta f(\rho \tau + i \rho \lambda)}} + \frac{\realpart{\left(\tau \partial_\zeta f(\rho \tau+ i \rho \lambda) \overline{\partial_\zeta^2 f(\rho \tau+ i \rho \lambda) } \right)}}{\abs{ \partial_\zeta f(\rho \tau + i \rho \lambda)}^3}.
\end{aligned}
\ee

\subsection{Function spaces and the abstract operator equation}

Next, we fix the functional analytic setting and rewrite the nonlocal nonlinear problem~\eqref{hollow nonlocal equation} as an abstract operator equation to which we will eventually apply the implicit function theorem. This is actually a somewhat delicate task, as we notice that the curvature term on the vortex core~\eqref{definition Ks Kv} is unbounded in the limit $\rho \to 0$. In some sense, the heart of the desingularization argument is in correctly defining a nonlinear functional that is real analytic in a neighborhood of $\rho = 0$ but whose zero-set includes all solutions to~\eqref{hollow nonlocal equation}.

As discussed so far, our unknowns are $w$ and the density $\mu$, along with the circulation $\gamma_1$ and Bernoulli constant $q$. In fact, $q$ will be unbounded for $\rho$ near $0$, and so it will also be necessary to use a ``normalized Bernoulli constant'' $Q$, as we explain shortly below. The domain of the abstract operator equation is then a subset of the Banach space
\[
	u \colonequals (w,\mu,\gamma_1,Q) \in \mathscr{X} \colonequals \mathscr{W} \times \mathring{C}_{\imagimag}^{k+\alpha}(\mathbb{T}) \times \mathbb{R}^2. 
\]
Note that here, for convenience, we are bundling the unknowns together. Let 
\begin{align*}
	\mathscr{O} \colonequals  \Big\{ \{(u, \rho)\in\mathscr{X}\times\mathbb{R}:
		&\quad  (\partial_{\zeta_1} w)^2 + (1 + \partial_{\zeta_2} w)^2, ~|\partial_\zeta f|^2 > 0~ \textrm{in } \overline{\Omega},~ \lambda \rho < 1 	\Big\},
\end{align*}
and note that for $(u,\rho) \in \mathscr{O}$, both the mapping $-i+ \id + \fs$ and the full mapping $f = -i+ \id + \fs + \fv$ are conformal on $\overline{\Omega}$. While one could only impose the latter condition, it will turn out to be more convenient to have that both hold. The main point is that $\mathscr{O}$ is an open neighborhood of the origin in $\mathscr{X} \times \mathbb{R}$.

Now, let us regroup terms in the dynamic condition on $\Gammav$ as
\be
\label{hollow vortex core Bernoulli}
	\frac{1}{2}\left(\frac{a_{\text{v}}^2}{|\partial_\zeta f|^2} - 1\right) + B_{\text{v}}\left(\frac{1}{\rho|\partial_\zeta f|} + \frac{\realpart\left(\tau \partial_\zeta f\overline{\partial_\zeta^2 f}\right)}{|\partial_\zeta f|^3}\right) + \frac{1}{F^2}\imagpart{f} - q = 0.
\ee
The basic strategy is to note that while $1/(\rho |\partial_\zeta f|)$ is expected to be $O(\rho^{-1})$ as $\rho \searrow 0$, it is constant to leading order as $f+i$ is a perturbation of the identity. Thus, the singular part can be absorbed into the Bernoulli constant $q$. This motivates considering the nonlinear operator
\[
	R = R(w,\mu,\rho)(\tau) \colonequals  \left( \frac{1}{\rho| \partial_\zeta f |} - \frac{1}{\rho|1 + \partial_\zeta \fs(0)|} \right) \Big|_{\zeta = i\lambda \rho + \rho \tau} \qquad \textrm{for } \tau \in \mathbb{T}, 
\]
and defining the normalized Bernoulli constant $Q$ by
\be
\label{relationship Q and q}
	Q \colonequals  q - \frac{\Bv}{\rho|1 + \partial_\zeta \fs(0)|}.
\ee
Then the singular terms in~\eqref{hollow vortex core Bernoulli} can be written as
\[
	\Bv R(u;\rho) - Q = \frac{\Bv}{\rho|\partial_\zeta f|} - q.
\]
Notice that for any $0 < |\rho| \ll 1$, $R(\placeholder; \rho)$ is real-analytic as a mapping from a neighborhood of the origin in $\mathscr{X}$ to $C^{k-1+\alpha}(\mathbb{T})$. We claim that $R$ can be extended to a real-analytic mapping defined on $\mathscr{O}$, which notably includes the section $\rho = 0$.

First observe that because $\fs$ is real-on-real and imaginary-on-imaginary by construction, its power series has the special form
\[
	\partial_\zeta \fs(\zeta) = (\partial_\zeta \fs)(0) + \zeta^2 r[w](\zeta) \colonequals (\partial_\zeta \fs)(0) + \zeta^2 \sum_{j = 1}^{\infty}\frac{(\partial_{\zeta}^{2j+1}\fs)(0)}{j!}\zeta^{2j-2}.
\]
The symmetry properties of $\fs$ imply in particular that $\partial_\zeta f(0) \in \mathbb{R}$. Thus, for $(u,\rho) \in \mathscr{O}$ sufficiently small, we have
\[
	1+\partial_\zeta \fs(0) = 1+ \partial_{\zeta_2} w(0) > 0,
\]
and hence
\begin{align*}
	 |1 + \partial_\zeta \fs(0)|^2 - |\partial_\zeta f|^2 &= |1 + \partial_\zeta \fs(0)|^2 - |1+\partial_\zeta \fs + \rho \mathcal{Z}^\rho \mu^\prime|^2 \\
	& = -2 \left| 1+ \partial_\zeta \fs(0) \right|  \realpart{ \left( \zeta^2 r + \rho \mathcal{Z}^\rho \mu^\prime\right)} -  \left| \zeta^2 r + \rho \mathcal{Z}^\rho \mu^\prime \right|^2.
\end{align*}
Using this fact, we compute that at $\zeta = \rho(\tau + i \lambda) \in \Gammav$, 
\be
\label{hollow expansion of R}
	\begin{aligned}
	R(w,\mu,\rho)(\tau) &= \frac{|1 + \partial_\zeta \fs(0)|^2 - |\partial_\zeta f|^2}{\rho|1 + \partial_\zeta \fs(0)||\partial_{\zeta}f | \left( |1+\partial_\zeta \fs(0)| + |\partial_\zeta f| \right)} \\
	&= \frac{-2 \left| 1 + \partial_\zeta \fs(0) \right| \realpart\left(\rho(\tau + i\lambda)^2r + \mathcal{Z}_{\text{v}}^{\rho}\mu^\prime \right) - \rho \abs{\rho(\tau + i\lambda)^2 r + \mathcal{Z}_{\text{v}}^{\rho}\mu^\prime}^2}{|1 + \partial_\zeta \fs(0)||\partial_{\zeta} f |(|1 + \partial_\zeta \fs(0)| + |\partial_\zeta f|)}\\
	&\equalscolon -\frac{2\realpart\left(\mathcal{Z}_{\text{v}}^{\rho}\mu^\prime + \rho(\tau + i\lambda)^2r \right)}{| \partial_{\zeta}f |(|1 + \partial_\zeta \fs(0))| + |\partial_{\zeta}f|)} + \mathcal{A}\\
	\end{aligned}
\ee
where $\mathcal{A} = \mathcal{A}(u; \rho) \colon \mathscr{O} \to C^{k-1+\alpha}(\mathbb{T})$ is the real-analytic operator defined by
\begin{align*}
	\mathcal{A}(u;\rho) & = 
	 \left\{ 
	 \begin{aligned} 
	 	- \frac{\rho\abs{\rho(\tau + i\lambda)^2r + \mathcal{Z}_{\text{v}}^{\rho}\mu^\prime}^2}{|1 + \partial_\zeta f^{\text{s}}(0)|| \partial_\zeta f|(|1 + \partial_\zeta \fs(0)| + |\partial_\zeta f|)} & \qquad |\rho| > 0 \\
		0 & \qquad \rho = 0
	\end{aligned} \right. \\
	& = 
	O\left( \rho^2 \| w \|_{C^{k-1+\alpha}} + \rho \| \mu \|_{C^{k-1+\alpha}} \right)\qquad \textrm{in } C^{k-1+\alpha}(\mathbb{T}).
\end{align*}
It is clear, then, that we can extend $R$ as a real-analytic operator defined at $\rho =0$. In particular, since $\mathcal{A}$ vanishes identically at $\rho = 0$, we see from the first term on the last line of~\eqref{hollow expansion of R} that 
	\begin{equation}
	\label{asymptotics of R}
	\begin{split}
	R(w,\mu,0) 
	&= -\frac{\realpart{\mathcal{C}\mu^\prime}}{|1 + \partial_{\zeta_2} w(0) |^2}, 
	\end{split}
	\end{equation}
where $\mathcal{C}$ is the Cauchy integral operator 
\[
	\mathcal{C}[\mu](\tau) \colonequals  \frac{1}{2\pi i}\int_{\mathbb{T}}\frac{\mu(\sigma)-\mu(\tau)}{\sigma - \tau}  d\sigma.
\]
It will be important for our later computations to note that $\mathcal{C}$ acts as a Fourier multiplier: for any integer $m$, $\mathcal{C} \tau^m = -\tau^m$ if $m < 0$ and $\mathcal{C} \tau^m = 0$ if $m \geq 0$.

Now, combining these observations, we can reframe the wave-borne hollow vortex problem as the operator equation
\be
\label{hollow vortex abstract operator equation}
	\mathscr{G}(u; \rho) = 0,
\ee
for the (real-analytic) operator $\mathscr{G} \colon \mathscr{O} \to \mathscr{Y}$ given by
\begin{equation*}
	\begin{split}
	\mathscr{G}_1(u;\rho) &\colonequals \left( \frac{1}{2}\frac{a_{\text{s}}^2}{| \partial_\zeta f|^2} + \Bs\Ks(u;\rho) + \frac{1}{F^2}\imagpart{f} - \frac12 \right) \Big|_{\zeta=\zeta_1 + i} \\
	\mathscr{G}_2(u;\rho) &\colonequals \left( \frac{1}{2}\left(\frac{a_{\text{v}}^2}{| \partial_\zeta f|^2} - 1\right) + \Bv R(u; \rho) + \Bv \frac{\realpart\left(\tau \partial_\zeta f \overline{\partial_\zeta^2 f}\right)}{|\partial_\zeta f|^3} + \frac{1}{F^2}\imagpart{f} -Q \right)\Big|_{\zeta = \rho (i\lambda + \tau)}, \\
	\end{split}
\end{equation*}
and where the codomain is $\mathscr{Y} \colonequals C_{0,\even}^{2-\ell_{\mathrm{s}}+\alpha}(\mathbb{R}) \times C_{\realimag}^{\alpha}(\mathbb{T})$, with $\ell_{\mathrm{s}} = 2$ if $\Bs > 0$ and $\ell_{\mathrm{s}} = 1$ if $\Bs = 0$.

\begin{lemma}
\label{abstract operator lemma}
The mapping $\mathscr{G}\colon\mathscr{O}\to\mathscr{Y}$ is well-defined, real analytic, and it satisfies $\mathscr{G}(0;0) = 0$. 
\end{lemma}
\begin{proof}
That $\mathscr{G}_1$ is well-defined and real analytic is obvious, as there are no singular terms, so it suffices to consider only $\mathscr{G}_2$. We have already seen that $R$ is real-analytic on $\mathscr{O}$. The only term that requires any care is the kinetic energy density. Recalling the asymptotics of $W^\rho$ in~\eqref{hollow vortex complex potential form}, we first expand $a_{\text{v}}^2$ at $\rho = 0$:
	\begin{equation*}
	\begin{split}
	a_{\text{v}}^2 &= \abs{\partial_\zeta W_{0} + \mathcal{Z}_{\mathrm{v}}^{\rho}\nu^\prime + \rho^2 \partial_\zeta \Ws}^2 + O(\rho^2)\quad\text{in }C^{k-1 + \alpha}(\mathbb{T}).\\
	\end{split}
	\end{equation*}
Note that from the definition of the complex relative potential~\eqref{hollow W0 formula}, it follows that
	$$\partial_\zeta W_0(\rho\tau + i\lambda\rho) = -1+ \gamma_1\theta(\tau) + O(\rho^2)\quad\text{in }C^{k-1 + \alpha}(\mathbb{T}),$$
where
	$$\theta = \theta(\tau) \colonequals  \frac{\lambda}{\pi\tau(\tau+2i\lambda)}.$$
Thus,
	\begin{equation}
	\label{asymptotics of av}
	\begin{split}
	a_{\text{v}}^2 &= \abs{W_{\zeta}^0}^2 + 2\realpart\left(\partial_\zeta W_0 \overline{\mathcal{Z}_{\mathrm{v}}^{\rho}\nu^\prime}\right) + \abs{\mathcal{Z}_{\mathrm{v}}^{\rho}\nu^\prime}^2 + O(\rho^2)\\
	&= 1 - 2\gamma_1 \realpart\left(\theta\overline{\mathcal{Z}_{\mathrm{v}}^{\rho}\nu^\prime} +\theta\right) + \gamma_1^2 |\theta|^2 + O(\rho^2)\quad\text{in }C^{\ell + 1 + \alpha}(\mathbb{T})\\
	\end{split}
	\end{equation}
near $\rho = 0$. Thus, $\mathscr{G}_2$ is analytic as well. That its range lies in the subspace $C_{\realimag}^\alpha(\mathbb{R})$ follows as in \cite[Lemma 6.4(b)]{chen2026desingularization}. Finally, a simple computation shows that $\mathscr{G}(0;0) = 0$.
\end{proof}

\subsection{Existence of wave-borne hollow vortices}

We are now ready to prove Theorem~\ref{intro existence hollow vortex theorem}, which we first do in terms of the nonlocal formulation.

\begin{theorem}[Wave-borne hollow vortices]
\label{hollow vortex theorem}
Let the air-sea Bond number $\Bs$ and Froude number $F$ be given so that either 
\[
	\Bs > \tfrac{1}{3} \textrm{ and } 0 < F^2 < 1  \qquad \textrm{or} \qquad \Bs = 0 \textrm{ and } F^2 > 1.
\]
For any vortex core Bond number $\Bv > 0$, there exists a curve $\mathfrak{K}_\loc$ of solutions to the solitary capillary wave-borne hollow vortex problem~\eqref{hollow vortex abstract operator equation} admitting the real-analytic parameterization
	$$\mathfrak{K}_\loc \colonequals  \{(w^{\rho},\mu^{\rho},\gamma^{\rho},Q^{\rho},\rho):|\rho| \ll 1\} \subset \mathscr{G}^{-1}(0),$$
with $(w^0,\mu^0,\gamma^0,Q^0,0) = (0, 0,0,0,0).$
\end{theorem}

\begin{proof}
We have already confirmed that $\mathscr{G}$ is real analytic and $(0,0)$ is in its zero-set. The idea is simply to apply the implicit function theorem there. To understand the linearization of $\mathscr{G}_1$, we begin by expanding $a_{\text{s}}^2$ at $\rho = 0$. This gives
	\begin{equation*}
	\begin{aligned}
	a_{\text{s}}^2 &= \abs{\partial_\zeta W_0}^2 + 2\realpart\left(\overline{\partial_\zeta W_{0}}\mathcal{Z}_{\text{s}}^{\rho}\nu'\right) + O\left(\rho^2\right) & \qquad\text{in }C^{k-1+ \alpha}(\mathbb{R}) \\
	&= 1 - 2\realpart\left(\mathcal{Z}_{\text{s}}^{\rho}\nu'\right) + O(\rho^2) & \qquad\text{in }C^{k-1 + \alpha}(\mathbb{R})\\
	\end{aligned}
	\end{equation*}
as $\rho \to 0$. Hence,
	\begin{equation*}
	\begin{split}
	\frac{1}{2}\frac{a_{\text{s}}^2}{|\partial_\zeta f |^2} &= \frac{1}{2| \partial_\zeta f|^2} - \frac{\realpart{\mathcal{Z}_{\text{s}}^{\rho}\nu'}}{| \partial_\zeta f|^2} + \mathscr{R},\\
	\end{split}
	\end{equation*}
where $\mathscr{R}:\mathscr{O}\to\mathscr{Y}_1$ is real analytic and satisfies
	$$\mathscr{R} = O\left(\rho^2 + \|w\|_{C^{k+\alpha}}^2\right)\quad\text{in}\quad C^{k-\ell_{\mathrm{s}} + \alpha}(\mathbb{R}).$$
From this, it follows that the Frechét derivative at $0$ is
	$$D_u\mathscr{G}_1(0,0)\begin{pmatrix}
	\dot{w}\\
	\dot{\mu}\\
	\dot{\gamma}_1\\
	\dot{Q}\\
	\end{pmatrix} = \left(-\Bs \partial_{\zeta_1}^2 \dot w  - \partial_{\zeta_1} \dot w+ \frac{1}{F^2} \dot w \right)\Big|_{\Gammas}.$$
Note that coincides the linearization of the Bernoulli condition on the air-sea interface for the wave-borne point vortex problem, indicating that the geometry of the hollow vortex only enters as a high-order effect. 

Next, we compute the linearization of the Bernoulli condition on the vortex boundary. Recall that we have already determined the leading-order form of $R$ in~\eqref{asymptotics of R}, and the leading-order forms of $a_{\mathrm{v}}$ in~\eqref{asymptotics of av}. 

Based on these earlier calculations and the symmetry properties of $f$, it is straightforward to verify that
	\begin{equation}
	\label{hollow DG2 formula}
	\begin{split}
	D_u\mathscr{G}_2(0,0)\begin{pmatrix}
	\dot{w}\\
	\dot{\mu}\\
	\dot{\gamma}_1\\
	\dot{Q}\\
	\end{pmatrix} &= -\dot{\gamma}_1 \realpart{\left(\frac{\lambda}{\pi\tau(\tau+2i\lambda)}\right)}  - \partial_{\zeta_2} \dot w(0) + \Bv \realpart\left(\tau\overline{\mathcal{C}\dot{\mu}''} - \mathcal{C}\dot{\mu}'\right) - \dot{Q}.\\
	\end{split}
	\end{equation}

The full Frechét derivative of $\mathscr{G}$ therefore is lower block triangular, and takes the form
	$$D_u\mathscr{G}(0; 0)\begin{pmatrix}
	\dot{w}\\
	\dot{\mu}\\
	\dot{\gamma}_1\\
	\dot{Q}\\
	\end{pmatrix} = \begin{pmatrix}
	 D_w\mathscr{G}_1(0;0)&0\\
	D_{w}\mathscr{G}_2(0;0)&D_{(\mu,\gamma_1,Q)}\mathscr{G}_2(0;0)\\
	\end{pmatrix}\begin{pmatrix}
	\dot{w} \\
	\dot{\mu}\\
	\dot{\gamma}_1 \\
	\dot{Q}\\
	\end{pmatrix}.$$
For the case $\Bs > 1/3$ and $0 < F^2 < 1$, we have already shown that the upper left block is invertible $\mathscr{W} \to \mathscr{Y}_1$ in the proof of Theorem~\ref{intro existence point vortex theorem}. When $\Bs = 0$ and $F^2 > 1$, its invertibility was proved in \cite[Theorem 2.2]{chen2025vortex}. The lower left block in $D_u \mathscr{G}(0;0)$ is clearly bounded, and thus it remains only to verify that the lower right block is invertible $\mathring{C}_{\imagimag}^{k+\alpha}(\mathbb{T}) \times \mathbb{R}^2 \to C_{\realimag}^{k-2+\alpha}(\mathbb{T})$.

Towards, that end, we first observe that from~\eqref{hollow DG2 formula} it follows that
\begin{align*}
	P_0 D_{(\gamma_1, Q)} \mathscr{G}_2(0;0) 
	\begin{pmatrix} 
		\dot \gamma_1 \\ \dot Q 
	\end{pmatrix} & = -\frac{\dot \gamma_1}{2\pi \lambda}  - \dot Q \\
	P_1 D_{(\gamma_1, Q)} \mathscr{G}_2(0;0) 
	\begin{pmatrix} 
		\dot \gamma_1 \\ \dot Q 
	\end{pmatrix} & = -\dot{\gamma}_1 \realpart{\left( \frac{1+4\lambda^2}{4 \pi \lambda^2} i \tau \right)}.
\end{align*}
As it is simply a bounded upper triangular operator with invertible diagonal entries, we therefore conclude that
\be
\label{isomorphism onto P<=1}
	P_{\leq 1} D_{(\gamma_1, Q)} \mathscr{G}_2(0;0) \colon \mathbb{R}^2 \to P_{\leq 1} C_{\realimag}^{k-2+\alpha}(\mathbb{T}) \quad \textrm{is an isomorphism}.
\ee
On the other hand, writing $\dot \mu$ as a power series in $\tau$, keeping in mind that the definition of $\mathring{C}_{\imagimag}^{k+\alpha}(\mathbb{T})$, gives
\[
	\dot \mu = \sum_{m \in \mathbb{Z}\setminus\{0\}} \hat{\dot\mu}_m \tau^m, \qquad i^m\hat{\dot\mu}_m \in i \mathbb{R},~\hat{\dot\mu}_{-m} = \overline{\hat{\dot\mu}_m},
\]
and so we find that

\be
\label{hollow DG2Dmu formula}
	\begin{split}
	\realpart{(\tau\overline{\mathcal{C}\dot{\mu}^{\prime\prime}} - \mathcal{C}\dot{\mu}^\prime)} &= \realpart\left(-\sum_{m = 1}^{\infty}m(m + 1)\hat{\dot{\mu}}_{m}\tau^{m + 3} + 2\sum_{m = -\infty}^{-1}m\hat{\dot{\mu}}_m\tau^{m-1}\right)\\
	&= \frac{1}{2}\left(\hat{\dot{\mu}}_{-2}\tau^{-3} + \hat{\dot{\mu}}_{-1}\tau^{-2} - \hat{\dot{\mu}}_{1}\tau^{2} - \hat{\dot{\mu}}_{2}\tau^{3}\right)\\			&\qquad - \sum_{m=1}^{\infty} \left(m(m + 1)\hat{\dot{\mu}}_{m} + (m + 2)\hat{\dot{\mu}}_{m+2}\right)\tau^{m+3}\\
	&\qquad + \sum_{m=-\infty}^{-1} \left(m(m - 1)\hat{\dot{\mu}}_{m} + (m - 2)\hat{\dot{\mu}}_{m-2}\right)\tau^{m - 3}.\\
	\end{split}
\ee
Clearly, then, the range of $D_\mu \mathscr{G}_2(0;0)$ is a subset of $P_{>1} {C}_{\realimag}^{k-2+\alpha}(\mathbb{T})$. We claim that $D_\mu \mathscr{G}_2(0;0)$ is injective, and so because it a Fourier multiplier, it is invertible $\mathring{C}_{\imagimag}^{k+\alpha} \to P_{>1} C_{\realimag}^{k-2+\alpha}$.

Suppose $\dot\mu\in\ker D_{\mu}\mathscr{G}_2(0;0)$. Then by applying the projections $P_2$ and $P_3$ to the formula~\eqref{hollow DG2Dmu formula}, we conclude that $\hat{\dot\mu}_1 = \hat{\dot\mu}_2 = 0$. Recall that $\hat{\dot\mu}_0 = 0$ by the definition of the space $\mathring{C}_{\imagimag}^{k+\alpha}$. Arguing by induction on $m$, for $m \geq 4$, we see that $P_m D_\mu \mathscr{G}_2(0;0)\dot \mu = 0$ forces $\hat{\dot\mu}_{m-1} = 0$.  Thus, 
\[
	D_\mu \mathscr{G}_2(0;0) \colon \mathring{C}_{\imagimag}^{k+\alpha}(\mathbb{T}) \to P_{>1}C_{\realimag}^{k-2+\alpha}(\mathbb{T}) \quad \textrm{is an isomorphism}.
\]
Combining this with our previous observation~\eqref{isomorphism onto P<=1} completes the proof that $D_u \mathscr{G}_2(0;0)$ is an isomorphism, and hence so is $D_u \mathscr{G}(0;0)$. Finally, applying the real-analytic implicit function theorem furnishes the existence of the real-analytic curve $\mathfrak{K}_\loc$. 
\end{proof}

Theorem~\ref{intro existence hollow vortex theorem} is now a quick corollary.

\begin{proof}[Proof of Theorem~\ref{intro existence hollow vortex theorem}]
Translating back the solutions along $\mathfrak{K}_\loc$ from Theorem~\ref{hollow vortex theorem} to the original formulation yields the family $\mathscr{K}_\loc$ from part~\ref{intro hollow curve part}. At each parameter value $\rho > 0$, the vortex boundary $\partial\fluidV^\rho$ is the image of $\partial B_\rho(i\lambda \rho)$ under the conformal mapping $f^\rho$ constructed from $(w^\rho,\mu^\rho)$ via the ansatz~\eqref{hollow f ansatz}. Since $f+i$ is an $O(\rho^2)$ perturbation of identity in $C^{k+\alpha}$, the physical vortex boundary is necessarily a polar graph, as claimed in part~\ref{intro hollow asymptotics part}. Likewise, because $\gamma^\rho = \gamma_1^\rho \rho$, by construction, we see that $\gamma^\rho = O(\rho^2)$. Finally, the fact that $q^\rho$ diverges like $1/\rho$ as $\rho \searrow 0$ is a consequence of its relation to the normalized Bernouli constant~\eqref{relationship Q and q}:
\[
	q^\rho \colonequals Q^\rho + \frac{\Bv}{\rho | 1+ \partial_{\zeta_2} w^\rho(0)|} = O(1/\rho). \qedhere
\]
\end{proof}

\section{Orbital stability}
\label{stability section}

\subsection{General stability theory}
\label{general stability theory section}
	 For the convenience of the reader, this section recalls the relaxed GSS method introduced in~\cite{varholm2020stability} that will be the main abstract tool for the proof of Theorem~\ref{intro stability theorem}. Out of necessity, we recycle some notation. 

	Because of the quasilinear structure of our problem, we work with a scale of spaces
        \[
            \Wspace \hookrightarrow \Vspace \hookrightarrow \Xspace, 
        \]
        where $\Xspace$ is a real Hilbert space, while $\Vspace$ and $\Wspace$ are reflexive real Banach spaces. The inner product on $\Xspace$ will be denoted by $(\placeholder, \placeholder)_{\Xspace}$, and the corresponding norm by $\n{\placeholder}_{\Xspace}$.  Likewise, let $\n{\placeholder}_{\Vspace}$ and $\n{\placeholder}_{\Wspace}$ be the norms for $\Vspace$ and $\Wspace$, respectively.  We write $\Xspace^*$ for the (continuous) dual of $\Xspace$, which is naturally isomorphic to $\Xspace$ via the mapping $I \colon u \in \Xspace \mapsto (u, \placeholder)_{\Xspace} \in \Xspace^*$. As usual, we identify $\Xspace^{**}$ with $\Xspace$, and likewise for $\Vspace$ and $\Wspace$.  The pairing of $\Xspace$ and $\Xspace^*$ we denote by $\jbracket{\placeholder, \placeholder}_{\Xspace^* \times \Xspace}$, while $\jbracket{ \placeholder, \placeholder }_{\Wspace^* \times \Wspace}$ is the pairing between $\Wspace^*$ and $\Wspace$; when there is no risk of confusion, we will omit the subscript.   
        
       At the lowest level of regularity is the energy space $\Xspace$, which is where the Hamiltonian system will be formulated and is the natural setting for analyzing the spectrum. In general, though, the energy functionals may not be smooth on $\Xspace$, and so we introduce the smaller space $\Vspace$. Finally, we think of $\Wspace$ as a ``well-posedness space'', with the norm coming from higher-order energy estimates used to prove that the Cauchy problem is at least locally well-posed in time. The norm on $\Wspace$ also plays the secondary role of allowing us to get control over $\Vspace$ via interpolation.  More precisely, we require the following:
        
        \begin{assumption}[Spaces]
            \label{abstract interpolation assumption}
            Let $\Xspace$, $\Vspace$, and $\Wspace$ be given as above.  Assume that there exist constants $\theta \in (0,1]$ and $C> 0$ such that
            \begin{equation}
                \label{abstract interpolation inequality}
                \n{u}_{\Vspace}^3 \leq C \n{u}_{\Xspace}^{2+\theta} \n{u}_{\Wspace}^{1-\theta}
            \end{equation}
            for all $u \in \Wspace$.
        \end{assumption}

As usual, we often wish to work with functionals defined not on the full space, but on some open subset. For traveling waves with a point vortex, there must be a positive separation between the vortex center and the air--water free boundary as well as the rigid bed.  Abstractly, we will handle these types of situations by introducing an open set $\nbhdO \subset \Xspace$, where solutions must live.
        
        Suppose that $\hat{J} \colon D(J) \subset \Xspace^* \to \Xspace$ is a closed linear operator, and that we for each $u \in \nbhdO \cap \Vspace$ have a bounded linear operator $B(u) \in \Lin(\Xspace)$. We endow $\Xspace$ with symplectic structure in the form of the state-dependent Poisson map
        \begin{equation}
            \label{abstract state dependent poisson map}
            J(u) \colonequals B(u)\hat{J},
        \end{equation}
        which is required to satisfy a number of hypotheses. Notably, we do \emph{not} require $J(u)$ to be surjective, as is done in GSS, but ask instead for something slightly stronger than that the range of $J(u)$ is dense in $\Xspace$.  We are also allowing for a state-dependent Poisson map, albeit one which can be factored into a ``nice'' state-dependent part $B(u)$ and a state-independent part $\widehat{J}$. This generalization of GSS is necessary because the position of the vortex center enters into the Poisson map.

        \begin{assumption}[Poisson map]
            \label{abstract symplectic assumption}
            \leavevmode
            \begin{enumerate}[label=\rm(\roman*)]
                \item The domain $\Dom(\hat{J})$ is dense in $\Xspace^*$.     \label{J densely defined assumption}
                \item \label{J injectivity assumption} $\hat{J}$ is injective.
                \item For each $u \in \nbhdO \cap \Vspace$, $J(u)$ is skew-adjoint in the sense that
                \[
                \jbracket{J(u)v,w} = -\jbracket{v,J(u)w}
                \]
                for all $v,w \in \Dom(\hat{J})$.
            \end{enumerate}
        \end{assumption}
        
        The main object of interest for this work is the abstract Hamiltonian system 
        \begin{equation}
            \label{abstract Hamiltonian system}
            \frac{du}{dt} = J(u) D \eng(u), \qquad u|_{t=0} = u_0,
        \end{equation}
        where $\eng \in C^3(\nbhdO \cap \Vspace; \R)$ is the \emph{energy functional}.   In addition to the energy, we suppose that there is a second conserved quantity $\mom \in C^3(\nbhdO \cap \Vspace; \R)$, which we call the \emph{momentum}. In order to state what it means to be a solution of \eqref{abstract Hamiltonian system}, and to work with it in a meaningful way, we need to be able to view $D\eng(u)$ and $D\mom(u)$ as elements of $\Xspace^*$.
        
        \begin{assumption}[Derivative extension] \label{extend DP and DE assumption} There exist mappings $\nabla \eng, \nabla \mom \in C^0(\nbhdO \cap \Vspace;\Xspace^*)$ such that $\nabla \eng (u)$ and $\nabla \mom (u)$ are extensions of $D\eng(u)$ and $D\mom(u)$, respectively, for every $u \in \nbhdO \cap \Vspace$.
        \end{assumption}
        
        We say that $u  \in C^0([0,t_0);  \nbhdO \cap \Wspace)$ is a solution of \eqref{abstract Hamiltonian system} on the interval $[0, t_0)$ if
        \begin{equation}
            \label{weak abstract Hamiltonian system}
            \frac{d}{dt}\left\langle u(t),   w \right\rangle = -\left\langle \nabla\eng(u(t)),    J(u(t)) w \right\rangle \qquad \text{ for all $w \in \Dom(\hat{J})$,}
        \end{equation}
        is satisfied in the distributional sense on $(0,t_0)$, the initial condition $u(0) = u_0$ is satisfied, and both $\eng$ and $\mom$ are conserved.
        
        We assume further that there exists a one-parameter family of affine maps $T(s) \colon \Xspace \to \Xspace$, with linear part $dT(s)u \colonequals T(s)u - T(s)0$, having the properties described below. That the group is affine rather than linear in the classical GSS framework is due to the way translation invariance manifests in the wave-borne point vortex problem; see~\eqref{eq:symmetry-group}.
        
        \begin{assumption}[Symmetry group]
            \label{abstract symmetry assumption}
            The symmetry group $T(\placeholder)$ satisfies the following.  
            \begin{enumerate}[label=\rm(\roman*)]
                \item \label{invariances} \textup{(Invariance)}
                    The neighborhood $\nbhdO$, and the subspaces $\Vspace$ and $\Wspace$, are all invariant under the symmetry group. Moreover, $I^{-1}\Dom(\hat{J})$ is invariant under the \emph{linear} symmetry group.
                \item \label{group flow property} \textup{(Flow property)}
                    We have $T(0) = dT(0) = \mathrm{Id}_\Xspace$, and for all $s, r \in \R$,
                    \[
                        T(s+r) = T(s)T(r), \qquad \text{and hence} \qquad dT(s+r) = dT(s)dT(r).
                    \] 
                \item \label{unitary assumption} \textup{(Unitarity)}
                    The linear part $dT(s)$ is a unitary operator on $\Xspace$, and an isometry on $\Vspace$ and $\Wspace$, for each $s \in \R$.
                \item \label{strong continuity} \textup{(Strong continuity)}
                    The symmetry group is strongly continuous on $\Xspace$, $\Vspace$, and $\Wspace$.
                \item \label{affine bound assumption} \textup{(Affine part)} 
                    The function $T(\placeholder)0$ belongs to $C^3(\R; \Wspace)$, and there exists an increasing function $\omega \colon [0,\infty) \to [0,\infty)$ such that
                    \[
                    \n{T(s)0}_{\Wspace} \leq \omega(\n{T(s)0}_{\Xspace}), \quad \text{for all } s \in \R.
                    \]
                \item \label{commutativity assumption} \textup{(Commutativity with $J$)}
                    For all $s \in \R$,
                    \begin{equation}
                        \label{abstract commutation identity}
                        \begin{aligned}
                            \hat{J}I dT(s) &= dT(s)\hat{J}I,\\
                            dT(s)B(u) &= B(T(s)u)dT(s), \quad \text{for all } u \in \nbhdO\cap \Vspace.
                        \end{aligned}
                    \end{equation}
                \item \label{T'(0) assumption} \textup{(Infinitesimal generator)}
                    The infinitesimal generator of $T$ is the affine mapping 
                    \[
                        T'(0)u=\lim_{s \to 0} \left( s^{-1}(T(s)u - u) \right)= dT'(0)u + T'(0)0,
                    \]
                    with dense domain $\Dom(T^\prime(0)) \subset \Xspace$ consisting of all $u\in \Xspace$ such that the limit exists in $\Xspace$ (note that $\Dom(T^\prime(0)) =\Dom(dT^\prime(0))$ by the first part of assumption \ref{affine bound assumption}). Similarly, we may speak of the dense subspaces $\Dom(T'(0)|_\Vspace) \subset \Vspace$ and $\Dom(T'(0)|_\Wspace) \subset \Wspace$ on which the limit exists in $\Vspace$ and $\Wspace$, respectively.
                    
                    We assume that $\nabla P (u) \in \Dom(\hat{J})$ for every $u \in \Dom(T'(0)|_\Vspace) \cap \nbhdO$, and that
                    \begin{equation}
                        \label{abstract T'(0) and P' identity}
                        T'(0)u = J(u)\nabla P (u)
                    \end{equation}
                    for all such $u$. Moreover, we assume that
                    \begin{equation}
                        \label{abstract derivative commutation identity}
                        \hat{J}IdT'(0) = dT'(0)\hat{J}I.
                    \end{equation}
                \item \label{range density} \textup{(Density)}
                    The subspace
                    \[
                        \Dom(T'(0)|_\Wspace) \cap \Rng{\hat{J}}
                    \]
                    is dense in $\Xspace$.
                \item \label{T conserves energy} \textup{(Conservation)}
                    For all $u \in \nbhdO \cap \Vspace$, the energy is conserved by flow of the symmetry group:
                    \begin{equation}
                        \label{abstract energy invariance under T}
                        \eng(u) = \eng(T(s)u), \qquad \text{for all } s \in \R.
                    \end{equation}
            \end{enumerate}
        \end{assumption}
        
        We say that $u \in C^1(\R; \nbhdO \cap \Wspace)$ is a \emph{bound state} of the Hamiltonian system \eqref{abstract Hamiltonian system} provided that it is a solution of the form
        \[
            u(t) = T(c t) U_c,
        \]
        for some $c \in \R$ and $U_c \in \nbhdO \cap \Wspace$. We will also refer to $U_c$ itself as a bound state.

        \begin{assumption}[Bound states]
            \label{bound state assumption}
            There exists a one-parameter family of bound state solutions $\{ U_c :  c \in \cinterval \}$ to the Hamiltonian system \eqref{abstract Hamiltonian system}, where $\cinterval \subset \R$ is a non-empty open interval, and exhibiting the following properties.
            \begin{enumerate}[label=\rm(\roman*)]
                \item \label{bound states improved regularity}
                    The mapping $c \in \cinterval    \mapsto U_c \in \nbhdO \cap \Wspace$ is $C^1$.
                \item \label{bound state domain assumption}
                    For all $c \in \cinterval$, 
                    \begin{equation}
                        \label{technical bound state assumption}
                        U_c \in \Dom(T'''(0)) \cap \Dom(\hat{J}I T'(0)),
                    \end{equation}
                    and
                    \begin{equation}
                        \label{second technical bound state assumption}
                        U_c,~\hat{J}IT'(0)U_c \in \Dom(T'(0)|_\Wspace).
                    \end{equation}
                \item \label{bound state non-degeneracy}
                    The non-degeneracy condition $T'(0) U_c \neq 0$ holds for every $c \in \cinterval$. 
                \item \label{non-periodic bound state assumption}
                    Either $s\mapsto T(s) U_c$ is periodic, or $\liminf_{|s| \to \infty} \n{ T(s) U_c - U_c }_{\Xspace} > 0$.
            \end{enumerate}
        \end{assumption}

For a fixed parameter $c$, we define the \emph{augmented Hamiltonian} to be the functional $\augHam \in C^3(\Vspace \cap \nbhdO; \R)$ given by 
        \[ \augHam(u) \colonequals \eng(u) - c \mom(u).\]
        From the above assumptions it follows that $U_c$ is necessarily a critical point of $\augHam$, and thus each bound state $U_c$ is a critical point of the energy under the constraint of fixed momentum; the wave speed $c$ serves as the Lagrange multiplier.

       With sufficient regularity, we expect that the Hessian of the augmented Hamiltonian has a nontrivial null space containing  $T^\prime(0) U_c$. In many cases --- including our application to wave-borne point vortices --- one finds that the bound states are a saddle point of the energy with Morse index $1$, meaning that there is a unique negative eigenvalue, and the essential spectrum is positive and bounded away from $0$.  This motivates the next assumption on the configuration of the spectrum for the general case.

        \begin{assumption}[Spectrum]
            \label{spectral assumptions}
            The operator $D^2 \augHam(U_c) \in \Lin(\Vspace,\Vspace^*) $ extends uniquely to a bounded linear operator $\Hc \colon \Xspace \to \Xspace^*$ such that:
            \begin{enumerate}[label=\rm(\roman*)]
                \item \label{extensibility assumption}
                    $I^{-1} \Hc$ is self-adjoint on $\Xspace$.
                \item \label{spectrum config assumption}
                    The spectrum of $I^{-1} \Hc$ satisfies 
                    \begin{equation}
                        \label{spectrum of Hc}
                        \spectrum{(I^{-1} \Hc)} = \{ -\mu_c^2,  0\} \cup \Sigma_c,
                    \end{equation}
                    where $-\mu_c^2 < 0$ is a simple eigenvalue corresponding to a unit eigenvector $\chi_c$, $0$ is a simple eigenvalue generated by $T$, and $\Sigma_c \subset (0,\infty)$ is bounded away from $0$.  
            \end{enumerate}
        \end{assumption}

Assuming that all of the hypotheses from the previous subsection hold, we now state the main stability theorem from \cite{varholm2020stability}.   For a fixed bound state $U_c$ and radius $r > 0$, we define the tubular neighborhoods
\begin{align*}
\tube_r^\Xspace & \colonequals \{ u \in \mathcal{O} : \inf_{s \in \R} \| u - T(s) U_c \|_{\Xspace} < r \}, \\
\tube_r^\Wspace & \colonequals \{ u \in \mathcal{O} \cap \Wspace : \inf_{s \in \R} \| u - T(s) U_c \|_{\Wspace} < r \}.
\end{align*}
Similarly, for any $R > 0$, let $\mathcal{B}_R^\Wspace$ denote the intersection of $\mathcal{O}$ with the ball of radius $R$ centered at the origin in $\Wspace$.  Then $U_c$ is said to be \emph{conditionally orbitally stable} provided that for all $r > 0$ and $R > 0$, there exists $r_0 > 0$ such that if $u \colon [0,t_0) \to \mathcal{B}_R^\Wspace$ is a solution to \eqref{abstract Hamiltonian system} with $u(0) \in \tube_{r_0}^\Xspace$, then $u(t) \in \tube_r^\Xspace$ for all $t \in [0,t_0)$.  Here conditional refers to the fact that stability only holds provided we know the solution exists, and that its growth in $\Wspace$ is controllable.
                
        The \emph{moment of instability} is the scalar-valued function $d =d(c)$ that results from evaluating the augmented Hamiltonian along the family of bound states:
        \begin{equation}
            \label{abstract d definition}
            d(c) \colonequals \augHam(U_c) = \eng(U_c) - c \mom(U_c).
        \end{equation}      
        The main appeal of the GSS method lies in its ability to characterize the orbital stability of bound states in terms of the sign of $d^{\prime\prime}$.  In the more general setting of interest here, we have by \cite[Theorem 2.4]{varholm2020stability} the following stability criterion.  
        
        \begin{theorem}[Stability]
            \label{abstract stability theorem}
            Suppose that the above assumptions hold.  If $d^{\prime\prime}(c) > 0$, then the bound state $U_c$ is conditionally orbitally stable. 
        \end{theorem}

        \begin{remark}\label{weakened hypotheses remark}
        In fact, this result is slightly sharper than \cite[Theorem 2.4]{varholm2020stability}, which imposes two additional assumptions on the Poisson map:
        \begin{enumerate}[label=\rm(\roman*)]
        		 \item \label{B bijective assumption} For each $u \in \nbhdO \cap \Vspace$, the operator $B(u)$ is bijective. 
		  \item \label{regularity of J} The map $u \mapsto B(u)$ is of class $C^1(\nbhdO \cap \Vspace;\Lin(\Xspace)) \cap C^1(\nbhdO \cap \Wspace; \Lin(\Wspace))$.
        \end{enumerate}
        Neither of these holds for the wave-borne point vortex problem in finite-depth; see Remark~\ref{problem with B remark}. Thankfully, the proof of \cite[Theorem 2.4]{varholm2020stability} does not actually require them. On the other hand, they are needed for the proof of the instability result~\cite[Theorem 2.6]{varholm2020stability}. 
        \end{remark}
        
\subsection{Hamiltonian formulation of the wave-borne point vortex problem}
\label{hamiltonian formulation ww problem section} 

We now work to reframe the wave-borne point vortex problem in the form of a Hamiltonian system, then confirm it satisfies the plethora of assumptions enumerated above. Throughout this section, we work in dimensional variables, with the depth $d$ normalized to $d=1$. 

Recalling from Section~\ref{intro statement of results section} the decomposition \eqref{intro hodge helmholtz} of the velocity into an irrotational part $\nabla\Phi$ and the vortical contribution $\epsilon\nabla\Theta$ near the free surface $\fluidS_t$, as well as the velocity trace $\varphi\colonequals\Phi(t,x_1,\eta(t,x_1))$ of $\Phi$ on $\fluidS$, we now define $\Theta$ in the finite-depth setting. The requirements on $\Theta$ are that it must satisfy the impermeable condition on $\fluidB$, carry the correct singularity at the vortex center $\bfz$ and decay as  $\abs{x_1}\to\infty$. This is achieved by taking $\Theta=\realpart w$, where $w$ is the complex potential
\begin{equation}\label{eq:complex_potential}
    w(z) \colonequals \frac{1}{2\pi i}\log\left( \frac{\sinh(\frac{\pi}{2}(z - \bfz))}{\sinh(\frac{\pi}{2}(z - \bar{\bfz}))}\right),
\end{equation}
with the complex coordinate $z=x_1 + ix_2$ identified with the Cartesian coordinate $x=(x_1,x_2)$ and $\bfz=\bfz_1 + \bfz_2 i$ being the vortex position, and $\bar\bfz$ being its reflection through the undisturbed air-sea interface.
It is convenient to split $\Theta =\Theta_1-\Theta_2$, where
\begin{equation*}
	\Theta_1 \colonequals \realpart{\left(\frac{1}{2\pi i}\log\left(\sinh\left(\frac\pi2(z-\bfz)\right)\right)\right)}, \qquad 
    \Theta_2 \colonequals \realpart{\left(\frac{1}{2\pi i}\log\left(\sinh\left(\frac\pi2(z-\bar \bfz)\right)\right)\right)}.
\end{equation*}
Moreover, the function $\Theta$ is a well-defined harmonic function on the open set 
\[
    \{(x_1,x_2)\in\R^2: x_1\neq \bfz_1\text{ or }\abs{x_2}<-\bfz_2\},
\]
and $\nabla\Theta$ extends to a smooth velocity field on $\fluidD_t \setminus \{\bfz\}$ and is $L^2$ on the complement of any neighborhood of $\bfz$ in $\fluidD_t$. For later use, we also introduce the harmonic conjugate $\Gamma$ of $\Theta$, which we similarly split as $\Gamma =\Gamma_1-\Gamma_2$ for
\begin{equation*}
    \Gamma_1 \colonequals -\imagpart\left(\frac{1}{2\pi i}\log\left(\sinh\left(\frac\pi2(z-\bfz)\right)\right)\right), \qquad \Gamma_2 \colonequals -\imagpart\left(\frac{1}{2\pi i}\log\left(\sinh\left(\frac\pi2(z-\bar \bfz)\right)\right)\right).
\end{equation*}

Since $\Theta$ is explicitly given by the vortex center $\bfz(t)$, and $\Phi$ is harmonic and hence entirely determined by $\eta$ and its trace $\varphi$, it suffices to enforce the kinematic~\eqref{intro general kinematic} and dynamic conditions~\eqref{intro general dynamic} on the free surface $\fluidS_t$, with the vortical part $\nabla\Theta|_{\fluidS_t}$ entering as a forcing term. As these both involve taking tangential and normal derivatives of $\Phi$ and $\Theta$, we introduce the shorthand
\[
    \nabla_\perp \colonequals  (-\eta'\p_{x_1} + \p_{x_2})\big|_{\fluidS_t},
    \qquad
    \nabla_\top \colonequals  (\p_{x_1} + \eta'\p_{x_2})\big|_{\fluidS_t},
\]
which arise when parametrizing the free surface using $\eta$. Note also that we are using the convention that spatial derivatives of quantities restricted to the boundary are denoted by a prime, while $\p_{x_1}$ is reserved for functions of two or more spatial variables.

The tangential derivative $\nabla_{\top}\Phi$ is simply $\varphi'$, but to express the normal derivative $\nabla_\perp\Phi$ requires the nonlocal Dirichlet--Neumann operator $G(\eta)$, defined by
\begin{equation}\label{eq:DN_operator}
    G(\eta)\varphi \colonequals  \nabla_\perp(\mathscr{H}(\eta)\varphi),
\end{equation}
where $\mathscr{H}(\eta)\varphi$ is the unique bounded harmonic extension of $\varphi$ to the fluid domain $\fluidD_t$ satisfying the impermeable boundary condition on the bed $\fluidB$. Note, then, that we have
\[
	\nabla \Phi|_{\fluidS_t}
		=
	\begin{pmatrix}
		\frac{1}{\jbracket{\eta^\prime}^2} & -\frac{\eta^\prime}{\jbracket{\eta^\prime}^2} \\
		\frac{\eta^\prime}{\jbracket{\eta^\prime}^2} & \frac{1}{\jbracket{\eta^\prime}^2}
	\end{pmatrix}
	\begin{pmatrix} 
		\varphi^\prime\\
		G(\eta)\varphi
	\end{pmatrix},
	\qquad
	\partial_t \varphi = \partial_t \Phi|_{\fluidS_t} + \partial_t \eta \partial_{x_2} \Phi|_{\fluidS_t}.
\]

Using the above relations, the wave-borne point vortex problem can be recast in the style of Zakharov--Craig--Sulem as a nonlocal equation for $(\eta, \varphi, \bfz)$:
\begin{equation}\label{eq:system}
	\left\{
		\begin{aligned}
    		\p_t \eta &= G(\eta)\varphi + \epsilon\nabla_\perp\Theta, \\
    		\p_t\varphi &= -\frac{(\varphi')^2 - 2\eta'\varphi'G(\eta)\varphi
    		- (G(\eta)\varphi)^2}{2\langle\eta'\rangle^2} - g\eta
    		+ \sigma_s\left(\frac{\eta'}{\langle\eta'\rangle}\right)' \\
    		&\qquad - \epsilon\varphi'\Theta_{x_1}|_{\fluidS_t}
    		- \frac{\epsilon^2}{2}\big(|\nabla\Theta|^2\big)\big|_{\fluidS_t}
    		+ \epsilon\boldsymbol{\xi}|_{\fluidS_t} \cdot \p_t\bfz, \\
    		\p_t\bfz &= \nabla\Phi(\bfz)
    		- \epsilon\,\p_{x_1}\Theta_2(\bfz)\, \mathbf{e}_1
		\end{aligned}
	\right.
\end{equation}
where we have introduced 
\[
	\Xi \colonequals \Theta_1 + \Theta_2, \qquad 
	\bfxi \colonequals (\Theta_{x_1}, \Xi_{x_2}).
\]
These are motivated by the observation that $\nabla_{\bfz}\Theta = -\bfxi$: Indeed, $\p_{\bfz_1}\Theta = -\p_{x_1}\Theta$, while since $\bar \bfz$ is the reflection of $\bfz$, we have $\p_{\bfz_2}\Theta_1 = -\p_{x_2}\Theta_1$ but $\p_{\bfz_2}\Theta_2 = \p_{x_2}\Theta_2$, so that $\p_{\bfz_2}\Theta = -\p_{x_2}(\Theta_1+\Theta_2) = -\p_{x_2}\Xi$.

The first equation in \eqref{eq:system} is the air-sea interface kinematic condition in~\eqref{intro general kinematic}. To derive the second, we start from the unsteady Bernoulli law, 
\[
	\partial_t \left(  \Phi + \epsilon  \Theta \right) + \frac{1}{2} |\nabla \Phi + \epsilon \nabla \Theta|^2 + g x_2 + P = 0,
\]
which is just the integrated momentum equations~\eqref{eq:2D-Incompressible-Euler} using the Hodge--Helmholtz decomposition~\eqref{intro hodge helmholtz}. This is valid in any simply connected neighborhood of $\fluidS_t$ that avoids the point vortex. Evaluating this identity on $\fluidS_t$, the dynamic condition~\eqref{intro general dynamic} is equivalent to replacing the trace of the pressure by the signed curvature. Stated in terms of surface variables, this yields the second equation in~\eqref{eq:system}. The third equation is the Helmholtz--Kirchhoff equation~\eqref{eq:Helmholtz-Kirchhoff}, again taking into account the Hodge--Helmholtz decomposition~\eqref{intro hodge helmholtz}.

We now rewrite \eqref{eq:system} as a Hamiltonian system for the state variables $u = (\eta, \varphi, \bfz)$. The first step is to fix a functional analytic framework. For that, we introduce the continuous scale of spaces
\begin{equation}\label{eq:scale}
    \begin{aligned}
        \Xspace^\frac12&\colonequals H^1(\R)\times\Xspace_2\times \R^2\\
        \Xspace^s &= \Xspace_1^s \times \Xspace_2^s \times \Xspace_3
        \colonequals  H^{s+1/2}(\R) \times (\dot{H}^s(\R) \cap
        \Xspace_2) \times \R^2, \quad s > 1/2,
    \end{aligned}
\end{equation}
where $\Xspace_2$ is the completion of the space $\mathcal{S}_0$ of
Schwartz functions whose Fourier transforms vanish in a neighborhood of
the origin, with respect to the norm
\begin{equation*}
    \n{f}_{\Xspace_2}^2\colonequals \int_\R \abs{\xi}\tanh{\abs{\xi}}\abs{\hat f(\xi)}^2\,d\xi.
\end{equation*}
It is easy to see that for each $s \geq 1/2$, $\Xspace^s$ is a Hilbert space, and the embedding $\Xspace^s \hookrightarrow \Xspace^{s'}$ is dense for all $1/2 \leq s' \leq s$.

For the energy space, we take
\begin{equation}\label{eq:X}
    \Xspace \colonequals  \Xspace^{1/2} = H^1(\R) \times \Xspace_2
    \times \R^2,
\end{equation}
which has the dual space
\[
    \Xspace^* = H^{-1}(\R) \times \Xspace_2^* \times \R^2,
\]
and for which the isomorphism $I\colon \Xspace \to \Xspace^*$
takes the explicit form
\[
    I = (1 - \p^2_{x_1},\ \abs{\p_{x_1}}\tanh{\abs{\p_{x_1}}},\
    \mathrm{Id}_{\R^2}).
\]
As explained in Section~\ref{intro statement of results section}, the definition of $\Xspace_2$ is made so that $\n{\varphi}_{\Xspace_2}^2/2$ is the kinetic energy contribution of the irrotational part of the velocity field at the undisturbed state, and the second component of the Riesz map $I$ is the finite-depth Dirichlet--Neumann operator $G(0)$. 
 
The energy functional will involve Dirichlet--Neumann operator $G(\eta)$, whose smooth dependence on $\eta$ requires more regularity than $\Xspace$. We therefore take $s>1$ in \eqref{eq:scale} and set
\begin{equation}\label{eq:V}
    \Vspace \colonequals  \Xspace^{1+} = H^{3/2+}(\R) \times (\dot{H}^{1+}(\R) \cap \Xspace_2) \times \R^2,
\end{equation}
where by $\Xspace^{1+}$ we mean $\Xspace^{1+\sigma}$ for $0 < \sigma \ll 1$. For the well-posedness space we use
\begin{equation}\label{eq:W}
    \Wspace \colonequals  \Xspace^{5/2+} = H^{3+}(\R) \times (\dot{H}^{5/2+}(\R) \cap \Xspace_2) \times
    \R^2.
\end{equation}

\begin{lemma}[Function spaces]\label{lem:spaces}
Let $\Xspace$, $\Vspace$, and $\Wspace$ be defined by \eqref{eq:X}, \eqref{eq:V}, and \eqref{eq:W}, respectively. Then there exist constants $C > 0$ and $\theta \in (0,1]$ such that the interpolation inequality \eqref{abstract interpolation inequality} holds.
\end{lemma}
\begin{proof}
We establish the interpolation inequality for each component separately, then combine. Since $\tanh\abs{\xi} \eqsim \abs{\xi}$ for $\abs{\xi} \leq 1$ and $\tanh\abs{\xi} \eqsim 1$ for $\abs{\xi} > 1$, the $\Xspace_2$-norm satisfies
\begin{equation}\label{eq:X2-high-low}
    \n{\varphi}_{\Xspace_2}^2
    \eqsim \int_{\abs{\xi} \leq 1} \abs{\xi}^2 \abs{\hat\varphi}^2\,d\xi
    + \int_{\abs{\xi} > 1} \abs{\xi} \abs{\hat\varphi}^2\,d\xi.
\end{equation}
We claim that
\begin{equation}\label{eq:phi-interp}
    \n{\varphi}_{\dot{H}^{1+\sigma}}
    \lesssim \n{\varphi}_{\Xspace_2}^{1-\lambda}
    \n{\varphi}_{\mathbb{W}_2}^{\lambda}
\end{equation}
for $\lambda = \frac{1+2\sigma}{2(2+\tau)} \approx 1/4$, where $\sigma, \tau > 0$ are the small parameters defining $\Vspace$ and $\Wspace$.

For low frequencies $\abs{\xi} \leq 1$, since
$\abs{\xi}^{2+2\sigma} \leq \abs{\xi}^2$ for $\abs{\xi} \leq 1$, we have
\[
    \int_{\abs{\xi} \leq 1} \abs{\xi}^{2+2\sigma} \abs{\hat\varphi}^2\,d\xi \leq \int_{\abs{\xi} \leq 1} \abs{\xi}^2 \abs{\hat\varphi}^2\,d\xi \lesssim \n{\varphi}_{\Xspace_2}^2 \lesssim \n{\varphi}_{\Xspace_2}^{2(1-\lambda)} \n{\varphi}_{\Wspace_2}^{2\lambda},
\]
where the last step uses $\n{\varphi}_{\Xspace_2} \lesssim \n{\varphi}_{\Wspace_2}$.

For high frequencies $\abs{\xi} > 1$, the exponent $\lambda$ chosen above satisfies $(1-\lambda) \cdot \frac{1}{2} + \lambda \cdot (\frac{5}{2}+\tau) = 1 + \sigma$. By H\"{o}lder's inequality,
\begin{align*}
    \int_{\abs{\xi}>1} \abs{\xi}^{2+2\sigma} \abs{\hat\varphi}^2\,d\xi &= \int_{\abs{\xi}>1} \big(\abs{\xi}\abs{\hat\varphi}^2\big)^{1-\lambda} \big(\abs{\xi}^{5+2\tau}\abs{\hat\varphi}^2\big)^{\lambda}\,d\xi \\
    &\leq \left(\int_{\abs{\xi}>1} \abs{\xi}\abs{\hat\varphi}^2\,d\xi\right)^{1-\lambda}\left(\int_{\abs{\xi}>1} \abs{\xi}^{5+2\tau}\abs{\hat\varphi}^2\,d\xi\right)^{\lambda} \lesssim \n{\varphi}_{\Xspace_2}^{2(1-\lambda)}\n{\varphi}_{\Wspace_2}^{2\lambda}.
\end{align*}
Combining low and high frequencies gives \eqref{eq:phi-interp}, and
therefore
\begin{equation}\label{eq:V2-interp}
    \n{\varphi}_{\mathbb{V}_2}^2
    = \n{\varphi}_{\dot{H}^{1+\sigma}}^2 + \n{\varphi}_{\Xspace_2}^2
    \lesssim \n{\varphi}_{\Xspace_2}^{2(1-\lambda)}
    \n{\varphi}_{\mathbb{W}_2}^{2\lambda}.
\end{equation}
The $\eta$-component satisfies $\n{\eta}_{H^{3/2+\sigma}} \lesssim \n{\eta}_{H^1}^{1-\lambda} \n{\eta}_{H^{3+\tau}}^{\lambda}$ by the standard Gagliardo--Nirenberg inequality, with the same exponent $\lambda$, as one checks from $\frac{3}{2}+\sigma = (1-\lambda) \cdot 1 + \lambda \cdot (3+\tau)$; the $\bfz$-component is finite-dimensional, so it contributes trivially. Denoting $a_i = \n{u_i}_{\Xspace_i}$ and $b_i = \n{u_i}_{\mathbb{W}_i}$ for $i = 1, 2, 3$, we have $\n{u_i}_{\mathbb{V}_i}^2 \lesssim a_i^{2(1-\lambda)} b_i^{2\lambda}$ for each $i$. Hence, by an application of Young's inequality
\[
    \n{u}_{\Vspace}^2
    \lesssim \sum_{i=1}^{3} a_i^{2(1-\lambda)} b_i^{2\lambda}
    \leq \left(\sum_{i=1}^{3} a_i^2\right)^{1-\lambda}
         \left(\sum_{i=1}^{3} b_i^2\right)^{\lambda}
    = \n{u}_{\Xspace}^{2(1-\lambda)} \n{u}_{\Wspace}^{2\lambda}.
\]
Raising to the power $3/2$, we conclude
\[
    \n{u}_{\Vspace}^3 \lesssim \n{u}_{\Xspace}^{3(1-\lambda)}
    \n{u}_{\Wspace}^{3\lambda}
    = \n{u}_{\Xspace}^{2+\theta} \n{u}_{\Wspace}^{1-\theta},
\]
with $\theta = 1 - 3\lambda$, which is positive provided $\sigma$ and
$\tau$ are sufficiently small (indeed $\lambda \to \frac{1}{4}$ as
$\sigma, \tau \to 0$).
\end{proof}

For the problem to be well-defined, the fluid domain must be non-degenerate, and the surface $\fluidS_t$ must lie strictly between the point vortex at $\bfz \in \fluidD_t$ and the mirror vortex $\bar \bfz$. We therefore let
\[
    \mathcal{O} \colonequals  \{u \in \Xspace : \eta > -1 \text{ on } \R, \
    -1 < \bfz_2 < \eta(\bfz_1) < -\bfz_2\}
\]
and seek solutions taking values in $\mathcal{O} \cap \mathbb{W}$ at each time. Note that $\mathcal{O}$ is indeed open in $\Xspace$, since $H^1(\R) \hookrightarrow C_0(\R)$. This differs from \cite{varholm2020stability} in two respects: the constraint $\eta > -1$ ensures that $\fluidD_t$ is non-degenerate, while the constraint $\bfz_2 > -1$ keeps the vortex inside.

We endow $\Xspace$ with symplectic structure by prescribing a Poisson map. First, consider the linear operator $\hat{J}: \Dom(\hat{J}) \subset \Xspace^* \to \Xspace$
defined by
\begin{equation}\label{eq:Jhat}
    \hat J \colonequals  \begin{pmatrix} 0 & 1 & 0 & 0 \\
    -1 & 0 & 0 & 0 \\
    0 & 0 & 0 & \epsilon^{-1} \\ 0 & 0 & -\epsilon^{-1} & 0 \end{pmatrix},
\end{equation}
with the natural domain
\[
    \Dom(\hat{J}) \colonequals  (H^{-1}(\R) \cap \Xspace_2)
    \times (H^1(\R) \cap \Xspace_2^*) \times \R^2.
\]
One can understand $\hat J$ as encoding the Hamiltonian structure for the point vortex and the water wave in isolation. To get the full system, we must incorporate wave--vortex interaction terms. For each $u \in \mathcal{O} \cap \Vspace$, define
\begin{equation}\label{eq:B}
    B(u) \colonequals  \id_\Xspace + \mathscr{K}(u),
\end{equation}
where $\mathscr{K}(u)$ is the linear operator given by
\[
    \mathscr{K}(u)\dot{w} \colonequals
    \begin{pmatrix} 0 & 0 & 0 & 0 \\
    -\epsilon\Xi_{x_2}|_{\fluidS_t} & \epsilon\Theta_{x_1}|_{\fluidS_t} &
    \epsilon\Theta_{x_1}|_{\fluidS_t} & \epsilon\Xi_{x_2}|_{\fluidS_t} \\
    0 & 1 & 0 & 0 \\ -1 & 0 & 0 & 0 \end{pmatrix}
    \begin{bmatrix} \langle\Theta_{x_1}|_{\fluidS_t}, \dot{\eta}\rangle \\
    \langle\Xi_{x_2}|_{\fluidS_t}, \dot{\eta}\rangle \\
    \dot{\bfz}_1 \\ \dot{\bfz}_2 \end{bmatrix}
\]
for all $\dot{w}=(\dot\eta,\dot\varphi,\dot \bfz) \in \Rng{\hat{J}}$. The full Poisson map is formed, as in
\eqref{abstract state dependent poisson map}, by composing $\hat{J}$ with $B(u)$.

\begin{remark}
\label{problem with B remark}
Unlike in \cite{varholm2020stability}, the pairing $\langle\Xi_{x_2}|_{\fluidS_t},\dot\eta\rangle$ appearing in
$\mathscr{K}(u)$ is not obviously finite for $\dot\eta \in H^1(\R)$: this is because $\Xi_{x_2}|_{\fluidS_t}(x_1) \to \pm \frac12$ as $x_1\to\pm\infty$, so $\Xi_{x_2}|_{\fluidS_t} \notin H^{-1}(\R)$. However, we can show that $\left(\Xi_{x_2}|_{\fluidS_t}\right)'\in H^{3/2+}(\R) \subset H^{-1/2}(\R)$. Consequently, if $\dot{w}\in\Rng{\hat{J}}$, so that $\dot\eta\in H^1(\R)\cap\Xspace_2^*$, the pairing $\langle\Xi_{x_2}|_{\fluidS_t},\dot\eta\rangle$ is exactly the duality pairing and is therefore finite.
\end{remark}

\begin{lemma}[Properties of $J$]\label{lem:properties-of-J}
    For each $u\in \mathcal{O}\cap\mathbb{V}$, the operator $J(u):\Dom(\hat J)\to \Xspace$ is given by
    \begin{equation}\label{eq:J}
    J(u) \colonequals  B(u)\hat{J} = \begin{pmatrix} 0 & 1 & 0 & 0 \\
    -1 & J_{22} & J_{23} & J_{24} \\
    0 & J_{32} & 0 & \epsilon^{-1} \\
    0 & J_{42} & -\epsilon^{-1} & 0 \end{pmatrix},
\end{equation}
where the entries are given by
\begin{align*}
    J_{22} &= -\epsilon\Xi_{x_2}|_{\fluidS_t}\langle\placeholder, \Theta_{x_1}|_{\fluidS_t}\rangle
    + \epsilon\Theta_{x_1}|_{\fluidS_t}\langle\placeholder, \Xi_{x_2}|_{\fluidS_t}\rangle, \\
    J_{23} &= -\Xi_{x_2}|_{\fluidS_t}, \quad J_{24} = \Theta_{x_1}|_{\fluidS_t}, \\
    J_{32} &= \langle\placeholder, \Xi_{x_2}|_{\fluidS_t}\rangle, \quad
    J_{42} = -\langle \placeholder, \Theta_{x_1}|_{\fluidS_t}\rangle,
\end{align*}
and the following hold:
\begin{enumerate}[label=\rm(\roman*)]
    \item \label{pv dense domain J}  $\Dom(\hat J)\subset \Xspace^*$ is dense.
    \item \label{pv hat J injective} $\hat J$ is closed and injective.
    \item \label{pv J skew} For each $u \in \nbhdO \cap \Vspace$, $J(u)$ is skew-adjoint
    in the sense that
                \[
                	\jbracket{J(u)v,w} = -\jbracket{v,J(u)w}
                \]
    for all $v,w \in \Dom(\hat{J})$.
\end{enumerate}
\end{lemma}
\begin{proof}
    To show that $\Dom(\hat J)$ is dense in $\Xspace^*$, it suffices to show that $H^{-1}(\R)\cap \Xspace_2$ is dense in $H^{-1}(\R)$ and that $H^1(\R)\cap \Xspace_2^*$ is dense in $\Xspace_2^*$.

    Let $f\in H^{-1}(\R)$ be given, and for each $n \geq 1$, define
    \[
    	\widehat f_n \colonequals \chi_n \widehat f, \qquad \textrm{where} \quad  \chi_n \colonequals \chi_{\{\frac1n<\abs{\xi}<n\}}.
\]
Then, $f_n\in H^{-1}(\R)\cap\Xspace_2$, since the weights $\jbracket{\xi}^{-2}$ and $\abs{\xi}\tanh\abs{\xi}$ are comparable (with constants depending on $n$) on the support of $\chi_n$. Now,
    \[
        \begin{aligned}
            \n{f_n-f}_{H^{-1}}^2
            &=\int_\R \jbracket{\xi}^{-2}\abs{\hat f_n -\hat f}^2\,d\xi\\
            &=\int_\R \jbracket{\xi}^{-2}\abs{\hat f}^2
            \big(\chi_{\{\abs{\xi}\leq\frac1n\}}
            +\chi_{\{\abs{\xi}\geq n\}}\big)\,d\xi\to 0
        \end{aligned}
    \]
    as $n\to\infty$ by dominated convergence. The same truncation argument shows that $H^1(\R)\cap \Xspace_2^*$ is dense in $\Xspace_2^*$, whose norm is given by 
    \[
    	\n{f}_{\Xspace_2^*}^2 = \int_\R (\abs{\xi}\tanh\abs{\xi})^{-1}\abs{\hat f}^2\,d\xi.
\]
 Indeed, on the support of $\chi_n$ this weight is comparable to $\jbracket{\xi}^{2}$, and the truncation errors vanish as $n\to\infty$ in $\Xspace_2^*$ as before. This finishes the proof of~\ref{pv dense domain J}.

    It is clear from \eqref{eq:Jhat} that $\hat J$ is closed and injective, and hence \ref{pv hat J injective}. Finally, for
    $u\in\nbhdO\cap \Vspace$, the skew-adjointness of $J(u)$ follows from a direct computation with \eqref{eq:J}, which proves~\ref{pv J skew}.
\end{proof}

Next, we determine the energy associated to the wave-borne point vortex system. For an irrotational wave, the kinetic energy is given by $\frac{1}{2}\int_{\fluidD_t}|\mathbf{v}(t)|^2\,dx$, where we recall that $\mathbf{v} = (v_1, v_2)$ is the velocity field. To adapt this to the point vortex case, we use the velocity decomposition \eqref{intro hodge helmholtz} and formally integrate by parts. This produces traces on $\fluidS_t$, plus terms at the vortex center, while the boundary term on the bed $\fluidB$ vanishes by the impermeable condition.  The Helmholtz--Kirchhoff model~\eqref{eq:Helmholtz-Kirchhoff} is equivalent to neglecting the singular energetic contribution at the vortex. Ultimately, this leads us to define the energy functional $E = E(u)$ by
\begin{equation}\label{eq:energy}
    E(u) \colonequals  K(u) + V(u),
\end{equation}
where
\begin{equation}\label{eq:kinetic}
\begin{aligned}
    K(u) &\colonequals  K_0(u) + \epsilon K_1(u) + \epsilon^2 K_2(u) \\
    &\colonequals  \frac{1}{2}\int_{\R} \varphi G(\eta)\varphi\,dx_1 + \epsilon\int_{\R} \varphi\nabla_\perp\Theta\,dx_1 + \frac{1}{2}\epsilon^2 \int_{\R}\Theta|_{\fluidS_t}\nabla_\perp\Theta\,dx_1 + \frac12\epsilon^2 \Gamma_2(\bfz)
\end{aligned}
\end{equation}
is the kinetic energy, and
\begin{equation}\label{eq:potential}
    V(u) \colonequals  \int_{\R}\left(\frac{1}{2}g\eta^2 + \sigma_s(\langle\eta'\rangle - 1)\right)dx_1
\end{equation}
is the potential energy. Here $\Gamma_2(\bfz)$ denotes the evaluation of $\Gamma_2$ at the vortex center, which is finite since $\Gamma_2$ is non-singular there.

Notice that $V$ depends solely on the surface profile. The momentum, which is the conserved quantity generating the symmetry group characterized by \eqref{abstract T'(0) and P' identity}, can be defined as 
\begin{equation}\label{eq:momentum}
    P = P(u) \colonequals  \epsilon\bfz_2 - \int_{\R} \eta'(\varphi + \epsilon\Theta|_{\fluidS_t})\,dx_1.
\end{equation}

It is easy to see that $E, P \in C^\infty(\nbhdO \cap \Vspace;\R)$.
By inspection, we see that $DE$ and $DP$ admit the explicit extensions
\begin{align}
    \nabla E(u) &\colonequals (E'_\eta(u), E'_\varphi(u), \nabla_{\bfz}E(u)),
    \label{eq:gradE}\\
    \nabla P(u) &\colonequals  (P'_\eta(u), P'_\varphi(u), \nabla_{\bfz}P(u)),
    \label{eq:gradP}
\end{align}
in $C^\infty(\nbhdO \cap \Vspace; \Xspace^*)$, with
\begin{equation}\label{eq:gradE_components}
\begin{aligned}
    E'_\eta(u) &\colonequals  \frac{(\varphi')^2 - 2\eta'\varphi'G(\eta)\varphi
    - (G(\eta)\varphi)^2}{2\langle\eta'\rangle^2} + g\eta
    - \sigma_s\left(\frac{\eta'}{\langle\eta'\rangle}\right)' \\
    &\quad + \epsilon\varphi'\Theta_{x_1}|_{\fluidS_t}
    + \frac{\epsilon^2}{2}(|\nabla\Theta|^2)|_{\fluidS_t}, \\
    E'_\varphi(u) &\colonequals  G(\eta)\varphi + \epsilon\nabla_\perp\Theta, \\
    \nabla_{\bfz}E(u) &\colonequals  -\frac{1}{2}\epsilon^2\int_{\R}
    \nabla_\perp(\Theta\,\bfxi)\,dx_1
    - \epsilon\int_{\R}\varphi\nabla_\perp\bfxi\,dx_1
    - \epsilon^2\p_{x_1}\Theta_2(\bfz)\,e_2,
\end{aligned}
\end{equation}
and
\begin{equation}\label{eq:gradP_components}
\begin{aligned}
    P'_\eta(u) &\colonequals  \varphi' + \epsilon\Theta_{x_1}|_{\fluidS_t}, \\
    P'_\varphi(u) &\colonequals  -\eta', \\
    \nabla_{\bfz}P(u) &\colonequals  \epsilon e_2
    + \epsilon\int_{\R}\eta'\bfxi|_{\fluidS_t}\,dx_1.
\end{aligned}
\end{equation}
Thus Assumption~\ref{extend DP and DE assumption} is indeed satisfied.

The next theorem confirms that the Hamiltonian system for this choice of the energy and Poisson map corresponds to the water wave with a point vortex problem in \eqref{eq:system}.

\begin{theorem}[Hamiltonian formulation]\label{thm:hamiltonian}
A function $u \colonequals  (\eta, \varphi, \bfz) \in C^1([0,t_0); \Wspace \cap \nbhdO)$ is a solution of the capillary--gravity water wave problem with a point vortex \eqref{eq:system} if and only if it is a solution to the abstract Hamiltonian system
\begin{equation}\label{eq:hamiltonian-system}
    \frac{du}{dt} = J(u)DE(u),
\end{equation}
where $J = J(u)$ is the Poisson map \eqref{eq:J} and $E$ is the energy functional defined in \eqref{eq:energy}.
\end{theorem}

\begin{proof}
The proof follows a similar computation in \cite[Theorem~5.3]{varholm2020stability}. Written out more explicitly using \eqref{eq:J}, the Hamiltonian system \eqref{eq:hamiltonian-system} reads
\[
\left\{
\begin{aligned}
    \p_t\eta &= E'_\varphi(u)\\
    \p_t\varphi &= \epsilon\bfxi|_{\fluidS_t}\cdot \big(\jbracket{E'_\varphi(u), \partial_{x_2} \Xi|_{\fluidS_t}} + \epsilon^{-1}\p_{\bfz_2}E(u),\, -\jbracket{E'_\varphi(u), \partial_{x_1} \Theta|_{\fluidS_t}}- \epsilon^{-1}\p_{\bfz_1}E(u)\big)\\
    & \qquad  -E'_\eta(u) \\
    \p_t\bfz &=\big(\jbracket{E'_\varphi(u), \partial_{x_2} \Xi|_{\fluidS_t}} + \epsilon^{-1}\p_{\bfz_2}E(u),\, -\jbracket{E'_\varphi(u), \partial_{x_1}\Theta|_{\fluidS_t}} - \epsilon^{-1}\p_{\bfz_1}E(u)\big).
\end{aligned}\right.
\]
By \eqref{eq:gradE_components}, the first of these is exactly the kinematic equation in \eqref{eq:system}. Moreover, the equation for $\p_t\varphi$ agrees with the corresponding one in \eqref{eq:system}, as the final term is simply $\epsilon\bfxi|_{\fluidS_t}\cdot\p_t\bfz$ by the equation for $\p_t\bfz$. Only the equation for the motion of the point vortex remains.

For the vortex dynamics, we find that
\[
    \p_t\bfz_2 = \int_{\R}\big(\varphi\nabla_\perp \partial_{x_1} \Theta
    - \partial_{x_1} \Theta|_{\fluidS_t}\, G(\eta)\varphi\big)\,dx_1
    + \frac{\epsilon}{2}\int_{\R}
    \big(\Theta|_{\fluidS_t}\nabla_\perp\partial_{x_1}\Theta
    - \partial_{x_1} \Theta|_{\fluidS_t}\nabla_\perp\Theta\big)\,dx_1.
\]
To convert these into surface integrals, let $\Psi$ denote the stream function of the irrotational part $\Phi$ of the flow, which satisfies $\nabla\Phi = \nabla^\perp\Psi$. By the kinematic, $\Psi$ is constant on the bed, so we can without loss of generality take it to vanish there. In particular, we have $(\Psi|_{\fluidS_t})^\prime = G(\eta)\varphi$ and $\nabla_\perp\Psi = -\varphi'$.
Moreover, since $\partial_{x_1}\Theta - i\partial_{x_1}\Gamma$ is meromorphic with simple poles on the image lattice $(\bfz + 2i\mathbb{Z}) \cup (\bar\bfz + 2i\mathbb{Z})$, we have that
\[
    \nabla_\perp\partial_{x_1}\Theta = (\partial_{x_1}\Gamma|_{\fluidS_t})',
    \qquad
    \nabla_\perp\partial_{x_1}\Gamma = -(\partial_{x_1}\Theta|_{\fluidS_t})'.
\]
Integrating by parts in $x_1$ and using these identities, the two integrals above become
\[
    \p_t\bfz_2
    = \int_{\fluidS_t} \mathbf{n} \cdot \big(\partial_{x_1}\Gamma\nabla\Psi
    - \Psi\nabla\partial_{x_1}\Gamma\big)\,dS
    + \frac{\epsilon}{2}\int_{\fluidS_t} \mathbf{n} \cdot
    \big(\partial_{x_1}\Gamma\nabla\Gamma - \Gamma\nabla\partial_{x_1}\Gamma\big)\,dS,
\]
where $\mathbf{n}$ is the outward unit normal of $\fluidS_t\cup\fluidB$.

For the second integral, note that $\Gamma$ and $\partial_{x_1} \Gamma$ are harmonic away from the image lattice, and both vanish on the line $\{x_2 = 0\}$. By the definition of $\nbhdO$, the region bounded by $\fluidS_t$ and $\{x_2 = 0\}$ contains no point of the lattice, so applying the divergence theorem over this region yields
\[
    \int_{\fluidS_t} \mathbf{n} \cdot\big(\partial_{x_1}\Gamma\nabla\Gamma
    - \Gamma\nabla\partial_{x_1}\Gamma\big)\,dS = 0.
\]

For the first integral, $\Psi$ is harmonic in $\fluidD_t$ while $\partial_{x_1}\Gamma$ is harmonic in $\fluidD_t\setminus\{\bfz\}$, and the boundary terms on the bed vanish since $\Psi = 0$ and $\partial_{x_1}\Gamma = 0$ there. The divergence theorem therefore gives
\[
    \int_{\fluidS_t} \mathbf{n} \cdot \big(\partial_{x_1}\Gamma\nabla\Psi
    - \Psi\nabla\partial_{x_1}\Gamma\big)\,dS
    = \int_{\partial B_r(\bfz)} \mathbf{n}\cdot\big(\partial_{x_1}\Gamma\nabla\Psi
    - \Psi\nabla\partial_{x_1}\Gamma\big)\,dS
\]
for all $0 < r \ll 1$. Since $\Gamma_2$ is harmonic in the ball
$B_r(\bfz)$, only $\Gamma_1$ contributes in the limit $r\to0$. We then compute that 
\[
    \begin{aligned}
        \int_{\partial B_r(\bfz)} \partial_{x_1}\Gamma_1\, \mathbf{n} \cdot\nabla\Psi\,dS
        &= \frac{1}{4}\int_0^{2\pi}
        \frac{r\sinh(\pi r\cos\theta)}
        {\cosh(\pi r\cos\theta)-\cos(\pi r\sin\theta)}
        \,(\cos\theta,\sin\theta)\cdot(\nabla\Psi)(\bfz + re^{i\theta})
        \,d\theta\\
        &\to\frac{1}{2\pi}\int_0^{2\pi}
        (\cos^2\theta,\sin\theta\cos\theta)\,d\theta
        \cdot(\nabla\Psi)(\bfz)
        = \frac{1}{2}\partial_{x_1}\Psi(\bfz) = \frac{1}{2}\partial_{x_2}\Phi(\bfz)
    \end{aligned}
\]
and
\[
    \begin{aligned}
        \int_{\partial B_r(\bfz)}
        &\Psi\, \mathbf{n} \cdot\nabla\partial_{x_1}\Gamma_{1}\,dS
        = r\int_0^{2\pi}\Psi(\bfz+re^{i\theta})\,
        (\cos\theta,\sin\theta)\\
        &\qquad\cdot\left(
        \frac{\pi}{4}\frac{1-\cosh(\pi r\cos\theta)\cos(\pi r\sin\theta)}
        {(\cosh(\pi r\cos\theta)-\cos(\pi r\sin\theta))^2},\
        -\frac{\pi}{4}\frac{\sinh(\pi r\cos\theta)\sin(\pi r\sin\theta)}
        {(\cosh(\pi r\cos\theta)-\cos(\pi r\sin\theta))^2}
        \right)d\theta\\
        &= -\frac{1}{2\pi}\int_0^{2\pi}
        \left(\int_0^1 \nabla\Psi(\bfz + tre^{i\theta})\,dt\right)
        \cdot (\cos\theta,\sin\theta)\cos\theta\,d\theta + O(r)\\
        &\to -\frac{1}{2}\partial_{x_1}\Psi(\bfz) = -\frac{1}{2}\partial_{x_2}\Phi(\bfz)
    \end{aligned}
\]
as $r\searrow 0$. In total, this gives
\[
    \p_t\bfz_2 = \partial_{x_2}\Phi(\bfz),
\]
which is the second component of the Helmholtz--Kirchhoff model equation. By an essentially identical argument for the third row of \eqref{eq:J}, with $\partial_{x_2}\Xi$ in place of $\partial_{x_1}\Theta$ and with the explicit term $-\epsilon\,\p_{x_1}\Theta_2(\bfz)$ arising from $\nabla_{\bfz}E$ in \eqref{eq:gradE_components}, we obtain
\[
    \p_t \bfz_1 = \partial_{x_1}\Phi(\bfz) - \epsilon\,(\p_{x_1}\Theta_2)(\bfz),
\]
which completes the verification of \eqref{eq:system}.
\end{proof}

\subsection{Symmetry}
Let $T = T(s): \Xspace \to \Xspace$ be the one-parameter family of affine mappings given by
\begin{equation}\label{eq:symmetry-group}
    T(s)u \colonequals  (\eta(\placeholder - s),\ \varphi(\placeholder - s),\
    \bfz + se_1),
    \quad s \in \R,
\end{equation}
representing the invariance of the underlying system with respect to horizontal translations. The linear part of the family is
\[
    dT(s)u = (\eta(\placeholder - s),\ \varphi(\placeholder - s),\ \bfz),
    \quad s \in \R,
\]
and the infinitesimal generator of $T$ is the affine operator
\begin{equation}\label{eq:generator}
    T^\prime(0)u = dT^\prime(0)u + T^\prime(0)0 = (-\p_{x_1}\eta,\ -\p_{x_1}\varphi,\ 0)
    + (0,\ 0,\ e_1),
\end{equation}
with domain $\Dom(T^\prime(0)) = \Xspace^{3/2}$.

\begin{lemma}[Properties of $T$]\label{lem:symmetry}
The group $T(\placeholder)$ satisfies Assumption~\ref{abstract symmetry assumption}.
\end{lemma}
\begin{proof}
Parts \ref{invariances}, \ref{group flow property}, and \ref{unitary assumption} are obvious from the definition of $T$, and the strong continuity \ref{strong continuity} in the respective spaces is likewise straightforward. Observe also that $T(s)0 = s(0, 0, e_1)$, which has norm $|s|$ in both $\Xspace$ and $\Wspace$. Thus part \ref{affine bound assumption} holds with $\omega(t) = t$.

For part \ref{commutativity assumption}, note that $I^{-1}\Dom(\hat{J})$ is invariant under $dT(s)$, which is therefore the common domain of definition for both sides of the first equation in \eqref{abstract commutation identity}. Verifying equality in both equations for all $s \in \R$ is then just a matter of inserting the definitions.

For part \ref{T'(0) assumption}, observe that $\Dom(T^\prime(0)|_\Vspace) = \Xspace^{2+}$. The momentum gradient $\nabla P(u) \in \Dom(\hat{J})$ for any $u \in \nbhdO \cap  \Dom(T^\prime(0)|_\Vspace)$ as a consequence of its formula in \eqref{eq:gradP} and  \eqref{eq:gradP_components}. Moreover, \eqref{abstract T'(0) and P' identity} and \eqref{abstract derivative commutation identity} can be obtained by direct computation.

To verify part \ref{range density}, note that $\hat J$ interchanges the first two slots of $\Xspace^*$, so that
\begin{align*}
    \Rng{\hat{J}} &= (H^{1}(\R) \cap\Xspace_2^*)
    \times (H^{-1}(\R) \cap \Xspace_2) \times \R^2,\\
    \Dom(T^\prime(0)|_\Wspace) &= \Xspace^{7/2+},
\end{align*}
and hence
\begin{equation*}
    \Dom(T^\prime(0)|_\Wspace) \cap \Rng{\hat{J}} =
    (H^{4+}(\R) \cap \Xspace_2^*) \times
    (H^{-1}(\R) \cap \dot{H}^{7/2+}(\R) \cap \Xspace_2)
    \times \R^2,
\end{equation*}
which is dense in $\Xspace$ by the truncation argument used in Lemma~\ref{lem:properties-of-J}.

Finally, the conservation of energy \ref{T conserves energy} under the group is immediate given the translation invariant nature of $E$ in \eqref{eq:energy}, \eqref{eq:kinetic}, and \eqref{eq:potential}.
\end{proof}

\subsection{Bound states: small-amplitude steady wave-borne point vortices}
We now reexpress the solutions along the curve $\cmconf_\loc$ constructed in Theorem~\ref{thm:existence} in the original (dimensional) physical variables. Since we have from the theorem that
\begin{equation*}
	\p_{\zeta_1}\realpart f^{\beta,F,\Bs}(\placeholder + i) = 1 + \p_{\zeta_2} w^{\beta,F,\Bs}(\placeholder + i) = 1 + O(\beta^2),
\end{equation*}
the restriction $\realpart f^{\beta, F,\Bs}|_\Gamma$ is invertible for $|\beta|\ll 1$. Let $\zun^{\beta,F,\Bs}\colonequals\zun^{\beta,F,\Bs}(x_1)$ be the inverse.

In particular, since $w^{\beta,F,\Bs}\in C^{k+\alpha}(\overline\Omega)\cap H^{k+\frac{1}{2}}(\Omega)$ by the theorem, we can show that $\zun^{\beta,F,\Bs}\in \dot H^{k}(\R)$ and $(\zun^{\beta,F,\Bs})'\in C^{k-1+\alpha}(\R)\cap H^{k-1}(\R)$. The free surface displacement in physical variables is then
\[
    \eta^{\beta,F,\Bs}(x_1) \colonequals
    w^{\beta,F,\Bs}(\xiofx^{\beta,F,\Bs}(x_1)+i)
    \in C^{k+\alpha}(\R)\cap H^{k}(\R).
\]

The point vortex in the physical domain sits at $(0, b^{\beta,F,\Bs})$, where $b^{\beta,F,\Bs} = \imagpart f^{\beta,F,\Bs}(i\beta)$ is the vertical coordinate of the vortex center from \eqref{eq:b-beta-relation}.

The trace of the irrotational velocity potential on the free surface $\fluidS$ is obtained as follows. Here and below, $c = F\sqrt{g}$ denotes the dimensional wave speed (recall that $d = 1$). The potential $c\,\realpart\big(W\circ(f^{\beta,F,\Bs})^{-1}\big)$ generates the relative velocity field $c\mathbf{u}$, whereas $\varphi$ in \eqref{eq:system} is the trace of the potential $\Phi$ of the irrotational part of the absolute velocity $\bfv$, which vanishes at infinity. Adding $c\,x_1$ converts the former to the potential of the absolute velocity, and subtracting the vortex potential $\epsilon^{\beta,F,\Bs}\Theta(\placeholder;(0,b^{\beta,F,\Bs})) = c\gamma\,\Theta(\placeholder;(0,b^{\beta,F,\Bs}))$ isolates the irrotational part. Restricting to the free surface, this gives
\begin{align}\label{eq:phi-physical}
    \varphi^{\beta,F,\Bs}(x_1)
    &= c\bigl(x_1 - \zun^{\beta,F,\Bs}(x_1)\bigr) \notag\\
    &\quad + \frac{c\gamma^{\beta,F,\Bs}}{2\pi}\arg \left(
    \frac{\sinh \bigl(\frac{\pi}{2}(\zun^{\beta,F,\Bs}(x_1)+i(1-\beta))\bigr)}
    {\sinh \bigl(\frac{\pi}{2}(\zun^{\beta,F,\Bs}(x_1)+i(1+\beta))\bigr)}
    \right) \notag\\
    &\quad - \frac{c\gamma^{\beta,F,\Bs}}{2\pi}\arg \left(
    \frac{\sinh \bigl(\frac{\pi}{2}(x_1+i(\eta^{\beta,F,\Bs}(x_1)-b^{\beta,F,\Bs}))\bigr)}
    {\sinh \bigl(\frac{\pi}{2}(x_1+i(\eta^{\beta,F,\Bs}(x_1)+b^{\beta,F,\Bs}))\bigr)}
    \right).
\end{align}
Here, note that $(\varphi^{\beta,F,\Bs})'\in C^{k-1+\alpha}(\R)\cap H^{k-1}(\R)$, and thus $\varphi^{\beta,F,\Bs}\in \dot H^{k}(\R)\cap \Xspace_2$ by the characterization of $\Xspace_2$ given after \eqref{eq:scale}. In particular, using \eqref{eq:w-kappa-asymptotics}, \eqref{eq:gamma-asymptotics}, and \eqref{eq:b-asymptotics}, we have the following asymptotic expansions in the dimensional and physical variables:

\begin{equation}\label{eq:asymptotics-physical}
\begin{aligned}
    \eta^{\beta,F,\Bs} &= \frac12\beta^2\,\ddot{w}^{F,\Bs}|_\Gamma
    + O(\beta^4), \\
    \varphi^{\beta,F,\Bs} &= -\frac12c\beta^2\,
    i\coth(D)\,\ddot{w}^{F,\Bs}|_\Gamma 
    + O(\beta^3), \\
    b^{\beta,F,\Bs} &= -1 + \beta
    + \tfrac{1}{2}\,\p_{\zeta_2}\ddot{w}^{F,\Bs}(0)\,\beta^3
    + O(\beta^4).
\end{aligned}
\end{equation}

The idea here is to find a one-dimensional section of three-dimensional manifold $\cmconf_\loc$ that is parameterized by the wave speed, as in the formulation of Assumption~\ref{bound state assumption}.  Since the dimensional vortex strength $\epsilon = c\gamma$ appears in the energy functional and the Poisson 
map, the stability analysis must be carried out along a family of solutions with $\epsilon$ held fixed. To arrange this, while allowing $\beta$ to vary, we must allow the wave speed $c$ to vary in a $\beta$-dependent way.

Recall that from the beginning of the section we have the normalization $d = 1$. We now further normalize $g = 1$ and fix the surface tension coefficient $\sigmas > 0$ once and for all, so that
\[
    F = c, \qquad \Bs = \frac{\sigmas}{c^2}.
\]
Fix $c_0 \in (0,1)$ with $c_0^2 < 3\sigmas$, so that $(F_0, \Bs^0) \colonequals (c_0, \sigmas/c_0^2)$ satisfies the hypotheses of Theorem~\ref{thm:existence}. For $|\beta| \ll 1$ and $|c - c_0| \ll 1$, we abbreviate
\[
    (\eta^{\beta,c}, \varphi^{\beta,c}, b^{\beta,c}, \gamma^{\beta,c}, \varkappa^{\beta,c}) \colonequals (\eta^{\beta,c,\sigmas c^{-2}}, \varphi^{\beta,c,\sigmas c^{-2}}, b^{\beta,c,\sigmas c^{-2}}, \gamma^{\beta,c,\sigmas c^{-2}}, \varkappa^{\beta,c,\sigmas c^{-2}}),
\]
and likewise $\ddot w^{c} \colonequals \ddot w^{c,\sigmas c^{-2}}$. Since $c \mapsto (c, \sigmas c^{-2})$ is real-analytic, so is $(\beta,c) \mapsto (\eta^{\beta,c},\varphi^{\beta,c},b^{\beta,c})$ by Theorem~\ref{thm:existence}. Fix a reference altitude $\beta_0 \in (0, \beta_*)$, which will later be taken small, and set
\[
    \varkappa_0 \colonequals  \varkappa^{\beta_0, c_0}, \qquad
    \gamma_0 \colonequals  \gamma(\varkappa_0, \beta_0), \qquad
    \epsilon_0 \colonequals  c_0\gamma_0.
\]

We seek a curve $\beta \mapsto c(\beta)$ along which $c\,\gamma \equiv \epsilon_0$, that is, a curve of zeros of
\[
    G(c, \beta) \colonequals c - \frac{c_0\gamma_0}{\gamma^{\beta,c}},
\]
which is well defined for $(c,\beta)$ near $(c_0,\beta_0)$ since $\gamma \approx 4\pi\beta \neq 0$ there, and $G(c_0, \beta_0) = 0$ by construction. A direct computation using \eqref{eq:gamma-asymptotics} and \eqref{eq:w-kappa-asymptotics} gives
\[
    \begin{aligned}
        (\p_cG)(c_0,\beta_0) &= 1
        + \frac{c_0}{\gamma_0}\,(\p_\varkappa\gamma)(\p_c\varkappa)
        \big|_{(c_0,\beta_0)} = 1 + O(\beta_0^4),\\
        (\p_\beta G)(c_0,\beta_0)
        &= \frac{4\pi c_0}{\gamma_0}\big(1 + O(\beta_0^2)\big).
    \end{aligned}
\]
Since $(\p_cG)(c_0,\beta_0) \neq 0$ for $\beta_0$ sufficiently small, the implicit function theorem yields a unique real-analytic curve $c = c(\beta)$, defined for $\abs{\beta - \beta_0} \ll 1$, with $G(c(\beta),\beta) = 0$. Moreover,
\[
    c(\beta) = c_0 + \Lambda_0(\beta-\beta_0)
    + O(\abs{\beta-\beta_0}^2),
    \qquad
    \Lambda_0 = -\frac{4\pi c_0}{\gamma_0}\big(1 + O(\beta_0^2)\big)
    = -\frac{c_0}{\beta_0}\big(1 + O(\beta_0^2)\big).
\]
Note that $\abs{\Lambda_0}\to\infty$ as $\beta_0 \to 0$: to leading order $c \propto 1/\gamma$ along the family, so the wave speed is inversely proportional to the non-dimensionalized vortex strength. This is an artifact of the fixed-$\epsilon$ parametrization and causes no difficulty, as $\beta_0$ is fixed throughout the stability analysis.

The family of solutions for the stability analysis is defined as 
\[
    U_{c(\beta)} = \big(\eta^{\beta,c(\beta)},\
    \varphi^{\beta,c(\beta)},\ b^{\beta,c(\beta)}e_2\big).
\]

\begin{lemma}\label{lem:assumption5}
The family $\{U_{c(\beta)}\}$ satisfies
Assumption~\ref{bound state assumption}.
\end{lemma}
\begin{proof}
The curve $\beta \mapsto c(\beta)$ is real-analytic by the implicit function theorem, and $(\beta,c)\mapsto (\eta^{\beta,c},\,\varphi^{\beta,c},\,b^{\beta,c}\mathbf{e}_2)$ is real-analytic, as shown above. Thus, $\beta\mapsto U_{c(\beta)}$ is real-analytic. Since moreover $c'(\beta_0) = \Lambda_0 \neq 0$, the map $\beta \mapsto c(\beta)$ is a real-analytic diffeomorphism from a neighborhood of $\beta_0$ onto an open interval $\cinterval \ni c_0$, and we write $U_c \colonequals U_{c(\beta(c))}$ for $c \in \cinterval$. In particular, Assumption~\ref{bound state assumption}\ref{bound states improved regularity} holds.

From above, we have
\[
    \eta^{\beta,c(\beta)} \in C^{k+\alpha}(\R)\cap
    H^{k}(\R), \qquad
    \varphi^{\beta,c(\beta)} \in 
    \dot H^{k}(\R)\cap\Xspace_2,
\]
for $k \gg 1$. These embed into the spaces required in Assumption~\ref{bound state assumption}(ii). Likewise, by \eqref{eq:generator}, we have that
\[
	T^\prime(0)U_{c(\beta)} = (-\p_{x_1}\eta^{\beta,c(\beta)},\, -\p_{x_1}\varphi^{\beta,c(\beta)},\, \mathbf{e}_1).
\]
The third component is obviously non-vanishing, and hence $T'(0)U_{c(\beta)} \neq 0$. This confirms Assumption~\ref{bound state assumption}\ref{bound state non-degeneracy}. 

Finally, for $s \neq 0$, the third components of $T(s)U_{c(\beta)}$ and
$U_{c(\beta)}$ differ:
\[
    \big|\big(se_1 + b^{\beta,c(\beta)}e_2\big)
    - b^{\beta,c(\beta)}e_2\big| = |s| > 0,
\]
so $T(s)U_{c(\beta)} \neq U_{c(\beta)}$, and in fact $\n{T(s)U_{c(\beta)} - U_{c(\beta)}}_{\Xspace} \ge |s| \to \infty$ as $|s|\to\infty$; the orbit is non-periodic. 
\end{proof}

By Theorem~\ref{abstract stability theorem}, orbital stability of the family $\{U_{c(\beta)}\}$ will follow once we verify that the moment of instability $d(c)$ from \eqref{abstract d definition} has positive second derivative along the family.

\begin{theorem}[Moment of instability]
\label{thm:dpp}
For $0 < \beta_0 \ll 1$, there exists $\delta > 0$ such that if $|\beta - \beta_0| < \delta$, then $d''(c(\beta)) > 0$. 
\end{theorem}

\begin{proof}
From \eqref{abstract d definition} together with the criticality of $U_c$, we have $d'(c) = -P(U_c)$.

Differentiating in $\beta$,
\begin{equation}\label{eq:dpp-formula}
    c'(\beta)\, d''(c(\beta))
    = - \jbracket{\nabla P(U_{c(\beta)}),\, \p_\beta U_{c(\beta)}}_{\Xspace^*\!,\,\Xspace},
\end{equation}
where $\nabla P \in C^0(\nbhdO\cap \Vspace; \Xspace^*)$ is the extension of $DP$ given by \eqref{eq:gradP}. The right-hand side is continuous in $\beta$, since $\beta \mapsto U_{c(\beta)}$ is real-analytic as shown in the proof of Lemma~\ref{lem:assumption5} and $u \mapsto \nabla P(u)$ is continuous by Assumption~\ref{extend DP and DE assumption}.

Moreover,
\[
    c'(\beta_0) = \Lambda_0
    = -\frac{c_0}{\beta_0}\big(1 + O(\beta_0^2)\big) \neq 0
\]
for $0 < \beta_0$ sufficiently small, and by continuity $c'(\beta) \neq 0$ near $\beta_0$. Dividing \eqref{eq:dpp-formula} by $c'(\beta)$ shows that $d''(c(\beta))$ is continuous near $\beta_0$.

We now compute the two factors on the right of \eqref{eq:dpp-formula} at $\beta = \beta_0$. First, by \eqref{eq:asymptotics-physical},
\[
    \p_\beta\big|_{\beta_0} U_{c(\beta)}
    = (0,\, 0,\, \mathbf{e}_2) + O(\beta_0).
\]
Here, the chain-rule contributions through $c(\beta)$ are harmless, despite the fact that $c'(\beta_0) = O(\beta_0^{-1})$, as each of $\p_c\eta^{\beta,c}$, $\p_c\varphi^{\beta,c}$, $\p_c b^{\beta,c}$ is $O(\beta_0^2)$, so $\p_c U c^\prime(\beta_0) = O(\beta_0)$ and is absorbed in the error.

Second, by \eqref{eq:gradP} and \eqref{eq:gradP_components},
\[
    \nabla P(U_{c(\beta_0)})
    = (0,\, 0,\, \epsilon_0\, \mathbf{e}_2) + O(\beta_0^2),
    \qquad \epsilon_0 = c_0\gamma_0 = 4\pi c_0\beta_0
    \big(1 + O(\beta_0^2)\big).
\]
Substituting into \eqref{eq:dpp-formula} at $\beta = \beta_0$, the pairing of the leading terms is $-\epsilon_0\langle \mathbf{e}_2, \mathbf{e}_2\rangle = -\epsilon_0$, and the cross terms contribute $O(\beta_0^2)$:
\[
    c'(\beta_0)\, d''(c(\beta_0))
    = -\epsilon_0 + O(\beta_0^2)
    = -4\pi c_0\beta_0 + O(\beta_0^2).
\]
Since $c'(\beta_0) = -\tfrac{c_0}{\beta_0}(1 + O(\beta_0^2))$, dividing by it gives
\[
    d''(c(\beta_0))
    = 4\pi\beta_0^2 + O(\beta_0^3) > 0 \qquad \textrm{for all } 0 < \beta_0 \ll 1.
\]
The mapping $\beta \mapsto d^{\prime\prime}(c(\beta))$ is smooth near $\beta_0$ and $d^{\prime\prime}(c(\beta_0)) > 0$, it follows that $d''(c(\beta)) > 0$ for all $|\beta - \beta_0| < \delta$ with $\delta$ sufficiently small.
\end{proof}

\subsection{Proof of orbital stability}
In this subsection, we verify Assumption~\ref{spectral assumptions}, which then allows us to apply the general stability results Theorem~\ref{abstract stability theorem} to the family of wave-borne point vortices discussed in the previous section. 

Recall that the traveling waves $\{U_c\}$ are critical points of the augmented Hamiltonian $E_c \colonequals E - cP$. Following the idea of Mielke~\cite{mielke2002energetic}, we observe that because $\varphi$ occurs quadratically in $E$, at a critical point, it can be eliminated in favor of $v \colonequals (\eta, \bfz) \in \Vspace_{1,3} \cap \nbhdO_{1,3}$, where
\[
    \Vspace_{1,3} \colonequals \Vspace_1 \times \Vspace_3,
    \qquad
    \nbhdO_{1,3} \colonequals \{(\eta, \bfz) \in \Xspace_1
    \times \Xspace_3 : \eta > -1, -1 < \bfz_2 < \eta(\bfz_1) < -\bfz_2\}.
\]
That is, because
\begin{equation}\label{eq:varphi-criticality}
    \langle D_\varphi E_c(u),\, \dot\varphi\rangle
    = \int_\R \dot\varphi\left(G(\eta)\varphi
    + \epsilon\nabla_\perp\Theta + c\eta'\right) dx_1,
\end{equation}
we set
\begin{equation}\label{eq:varphi-star}
    \varphi_*(v) \colonequals
    -G(\eta)^{-1}\big(c\eta' + \epsilon\nabla_\perp\Theta\big),
    \qquad
    u_*(v) \colonequals (\eta,\, \varphi_*(v),\, \bfz) \in \Vspace.
\end{equation}
The \emph{augmented potential} is then defined by \begin{equation}\label{eq:augmented-potential}
    \Vaug_c(v)\colonequals \min_{\varphi \in \Vspace_2} E_c(\eta, \varphi, \bfz) = E_c(u_*(v)), \qquad v \in \Vspace_{1,3} \cap \nbhdO_{1,3}.
\end{equation}
The minimum is attained precisely at $\varphi = \varphi_*(v)$, since $\varphi \mapsto E_c(\eta,\varphi,\bfz)$ is a strictly convex quadratic on $\Vspace_2$ with critical point $\varphi_*(v)$ by \eqref{eq:varphi-criticality}. Note that $\varphi_* \in C^\infty(\Vspace_{1,3} \cap \nbhdO_{1,3};\, \dot H^{3/2+}\cap \Xspace_2)$ and thus $u_* \in C^\infty(\Vspace_{1,3} \cap \nbhdO_{1,3};\, \Vspace)$. For later use, we also define the velocity traces
\begin{equation}\label{eq:a-b}
    \fraka = \fraka(v) \colonequals
    \big(\nabla(\mathscr{H}(\eta)\varphi_*)\big)\big|_\fluidS,
    \qquad
    \frakb = \frakb(v)
    \colonequals \fraka + \epsilon\nabla\Theta|_\fluidS - ce_1.
\end{equation}
Thus $\fraka$ is the irrotational part of the velocity field restricted to the surface, and $\frakb$ is the full relative velocity restricted to the surface. Note that $\frakb_2 = \eta'\,\frakb_1$: this is the kinematic boundary condition, and follows from \eqref{eq:varphi-star} since $u_*(v)$ is by construction a critical point of $E_c$ in $\varphi$.

The following three lemmas follow exactly the same computations as in Lemma~6.2, Lemma~6.3 and Lemma~6.5 of \cite{varholm2020stability}, and hence the proofs are omitted.

\begin{lemma}[Formula for $D^2\mathcal{V}_c^{\mathrm{aug}}$]\label{lem:D2Vaug}
For all $v \in \Vspace_{1,3} \cap \nbhdO_{1,3}$ and $\dot{v} = (\dot\eta, \dot{\bfz}) \in \Vspace_{1,3}$, we have
\begin{equation}\label{eq:D2Vaug}
    \jbracket{D^2\mathcal{V}_c^{\mathrm{aug}}(v)\dot{v},  \dot{v}
    }_{\Vspace_{1,3}^* \times \Vspace_{1,3}}
    = \jbracket{D_v^2 E_c(u_*(v))\dot{v},  \dot{v}
    }_{\Vspace_{1,3}^* \times \Vspace_{1,3}}
    - \jbracket{\mathscr{L}(v)\dot{v},
    G(\eta)^{-1}\mathscr{L}(v)\dot{v}
    }_{\Xspace_2^* \times \Xspace_2},
\end{equation}
where
\begin{equation}\label{eq:calL}
    \mathscr{L}(v)\dot{v}
    \colonequals G(\eta)(\mathfrak{a}_2\dot\eta)
    + (\mathfrak{b}_1\dot\eta)'
    + \epsilon\nabla_\perp\boldsymbol{\xi} \cdot \dot{\bfz}
\end{equation}
defines a bounded linear operator $\mathscr{L}(v) \in \Lin(\Xspace_{1,3};  \Xspace_2^*)$.
\end{lemma}

\begin{lemma}[Extension of $D^2\mathcal{V}_c^{\mathrm{aug}}$]
\label{lem:D2Vaug-extension}
For all $v \in \Vspace_{1,3} \cap \nbhdO_{1,3}$, there
is a self-adjoint linear operator
$A(v) \in \Lin(\Xspace_{1,3};  \Xspace_{1,3}^*)$
such that
\[
    \jbracket{D^2\mathcal{V}_c^{\mathrm{aug}}(v)\dot{v},  \dot{w}
    }_{\Vspace_{1,3}^* \times \Vspace_{1,3}}
    = \jbracket{A(v)\dot{v},   \dot{w}
    }_{\Xspace_{1,3}^* \times \Xspace_{1,3}}
\]
for all $\dot{v}, \dot{w} \in \Vspace_{1,3}$. Explicitly,
\begin{equation}\label{eq:A-matrix}
    A = \begin{pmatrix} A_{11} & A_{13} \\ A_{13}^* & A_{33} \end{pmatrix},
\end{equation}
with entries given by
\begin{align}
    A_{11}\dot\eta
    &\colonequals  (g + \mathfrak{b}_2'\mathfrak{b}_1)\dot\eta
    - \left(\frac{\sigmas}{\jbracket{\eta'}^3}\dot\eta'\right)'
    - \mathscr{M}\dot\eta, \label{eq:A11}\\
    A_{13}\dot{\bfz}
    &\colonequals   \epsilon\mathfrak{b}_1\nabla_\top
    \big(G(\eta)^{-1}\nabla_\perp\boldsymbol{\xi}
    - \boldsymbol{\xi}\big)\cdot\dot{\bfz},
    \label{eq:A13}\\
    A_{13}^*\dot\eta
    &\colonequals   \epsilon\int_\R
    \dot\eta \,\mathfrak{b}_1\nabla_\top
    \big(G(\eta)^{-1}\nabla_\perp\boldsymbol{\xi}
    - \boldsymbol{\xi}\big)\, dx_1,
    \label{eq:A13adj}\\
    A_{33}
    &\colonequals  D_{\bfz}^2 E_c(u_*)
    -  \epsilon^2\int_\R
    \nabla_\perp\boldsymbol{\xi} \odot
    G(\eta)^{-1}\nabla_\perp\boldsymbol{\xi}\, dx_1.
    \label{eq:A33}
\end{align}
Here $\mathscr{M}\dot\eta \colonequals  -\mathfrak{b}_1 \big(G(\eta)^{-1}(\mathfrak{b}_1\dot\eta)'\big)'$ and $x \odot y \colonequals  (x \otimes y + y \otimes x)/2$ is the symmetric outer product.
\end{lemma}

\begin{lemma}[Extension of $D^2 E_c$]\label{lem:Hc}
For all $v \in \Vspace_{1,3} \cap \nbhdO_{1,3}$, there is a self-adjoint operator
$H_c(v) \in \Lin(\Xspace;  \Xspace^*)$ such that
\begin{equation}\label{eq:Hc-extension}
    \jbracket{D^2 E_c(u_*(v))\dot{u},  \dot{w}
    }_{\Vspace^* \times \Vspace}
    = \jbracket{H_c(v)\dot{u},  \dot{w}
    }_{\Xspace^* \times \Xspace}
\end{equation}
for all $\dot{u}, \dot{w} \in \Vspace$. The operator is given by
\begin{equation}\label{eq:Hc}
    H_c(v)\dot{u}
    = \begin{pmatrix} \mathrm{Id}_{\Xspace_1^*} & 0 & 0 \\
    0 & 0 & \mathrm{Id}_{\Xspace_2^*} \\
    0 & \mathrm{Id}_{\R^2} & 0 \end{pmatrix}
    \begin{pmatrix} A(v) + \mathscr{L}(v)^* G(\eta)^{-1}\mathscr{L}(v)
    & -\mathscr{L}(v)^* \\
    -\mathscr{L}(v) & G(\eta) \end{pmatrix}
    \begin{bmatrix} \dot{v} \\ \dot\varphi \end{bmatrix},
\end{equation}
where $\mathscr{L}(v)^* \in \Lin(\Xspace_2;  \Xspace_{1,3}^*)$ is the adjoint of $\mathscr{L}(v)$, given by
\begin{equation}\label{eq:calL-adj}
    \mathscr{L}(v)^*\dot\varphi
    = \bigl(\mathfrak{a}_2 G(\eta)\dot\varphi
    - \mathfrak{b}_1\dot\varphi',
     \epsilon\jbracket{\nabla_\perp\boldsymbol{\xi},  \dot\varphi}\bigr).
\end{equation}
Moreover, we have
\begin{equation}\label{eq:Hc-quad}
    \jbracket{H_c\dot{u},  \dot{u}}_{\Xspace^* \times \Xspace}
    = \jbracket{A(v)\dot{v},  \dot{v}
    }_{\Xspace_{1,3}^* \times \Xspace_{1,3}}
    + \jbracket{G(\eta)(\dot\varphi - G(\eta)^{-1}\mathscr{L}(v)\dot{v}),
    \dot\varphi - G(\eta)^{-1}\mathscr{L}(v)\dot{v}
    }_{\Xspace_2^* \times \Xspace_2}
\end{equation}
for all $\dot{u} = (\dot{v}, \dot\varphi) \in \Xspace$.
\end{lemma}

We are now prepared to prove the main theorem.

\begin{proof}[Proof of Theorem~\ref{intro stability theorem}]
We have already confirmed that the wave-borne point vortex problem can be written as a Hamiltonian system that satisfies Assumptions~\ref{abstract interpolation assumption}--\ref{bound state assumption}. Moreover, in Theorem~\ref{thm:dpp}, we confirmed that $d^{\prime\prime}(c(\beta)) > 0$ when $0 < \beta_0 \ll 1$ and $|\beta-\beta_0| \ll 1$. Therefore, Theorem~\ref{intro stability theorem} follows directly from the abstract results Theorem~\ref{abstract stability theorem} as soon as we verify that Assumption~\ref{spectral assumptions} holds for such $\beta_0$ and $\beta$.

With that in mind, let $\beta_0$ and $\delta$ be as in Theorem~\ref{thm:dpp}, let $|\beta - \beta_0| < \delta$, and write
\[
	c \colonequals c(\beta), \qquad v_\beta \colonequals (\eta^{\beta,c}, b^{\beta,c}\mathbf{e}_2), \qquad H_c \colonequals H_{c}(v_\beta),
\]
we claim that the spectrum of $H_c = H_{c(\beta)}(v_\beta)$ has the form
\[
    \spectrum(I^{-1}H_c)
    = \{-\mu_c^2, ~0\} \cup \Sigma_c,
\]
where $-\mu_c^2 < 0$ is a simple eigenvalue with $\mu_c^2 \approx 4\pi c_0^2$, $0$ is a simple eigenvalue with eigenvector $T'(0)U_{c(\beta)}$, and $\Sigma_c \subset (0,\infty)$ is bounded away from $0$.

From~\eqref{eq:Hc}, we see that $I^{-1} H_c$ is as a relatively compact perturbation of the block operator $I^{-1} H_0$, where
\begin{equation}\label{eq:H0}
 H_0 \colonequals   \begin{pmatrix}
    1 - \sigmas\p_{x_1}^2 & -c\p_{x_1} & 0 \\
    c\p_{x_1} & \abs{\p_{x_1}}\tanh{\abs{\p_{x_1}}} & 0 \\
    0 & 0 & -4\pi c_0^2\, \mathbf{e}_2 \otimes \mathbf{e}_2 \end{pmatrix} \in \Lin(\Xspace, \Xspace^*).
\end{equation}
The upper-left block corresponds to the Hessian of $E_c$ with respect to $(\eta, \varphi)$ at the trivial solution with $\epsilon = 0$. The lower-right block is likewise is the leading-order part of $D_\bfz^2 E_c(U_c)$, which does not vanish as $\beta \to 0$, contributes the simple eigenvalues $0$, with eigenvector $(0,0,\mathbf{e}_1)$, and $-4\pi c_0^2$, with eigenvector $(0,0,\mathbf{e}_2)$. The $(\eta,\varphi)$-block is a Fourier multiplier. Therefore, to determine the spectrum of $I^{-1} H_0$, we need only examine the generalized eigenvalue problem on the Fourier side:
\[
     \begin{pmatrix} 1+\sigmas \xi^2 & ic\xi  \\ -ic\xi & \abs{\xi}\tanh{\abs{\xi}} 
    \end{pmatrix}
    \begin{pmatrix}\hat{\dot\eta}\\ \hat{\dot\varphi} \end{pmatrix}
    = \lambda
    \begin{pmatrix} 1+\xi^2 & 0  \\ 0 & \abs{\xi}\tanh{\abs{\xi}} 
    \end{pmatrix}
    \begin{pmatrix}\hat{\dot\eta}\\ \hat{\dot\varphi} \end{pmatrix}.
\]
This has a two-dimensional null space for any $\xi \in \mathbb{R}$. There are also a pair of nonzero eigenvalues $\lambda_1(\xi), \lambda_2(\xi)$ that satisfy
\[
        \lambda_1 + \lambda_2 = \frac{1+\sigmas \xi^2}{1+\xi^2} + 1 > 1,
    \qquad
    \lambda_1\lambda_2
    = \frac{1 + \sigmas \xi^2 - c^2\abs{\xi}\coth{\abs{\xi}}}{1+\xi^2},
\]
using $\abs{\xi}\coth{\abs{\xi}} \le 1 + \tfrac{\xi^2}{3}$,
\[
    \lambda_1\lambda_2 \ \ge\ \min\Big\{1-c^2,\ \sigmas-\frac{c^2}{3}\Big\}
    \ >\ 0
    \qquad\text{for } F^2 = c^2 < 1,\
    \Bs = \frac{\sigmas}{c^2} > \frac13.
\]
Thus, eigenvalues $\lambda_1(\xi)$ and $\lambda_2(\xi)$ are positive and bounded away from $0$ uniformly in $\xi$, meaning that the $(\eta,\varphi)$-block of $H_0$ is positive definite relative to $I$ with a lower bound $\lambda_* = \lambda_*(\Bs, F) > 0$. In summary,
\begin{equation}\label{eq:spec-H0}
	\spectrum(I^{-1}H_0) = \{-4\pi c_0^2, ~ 0\} \cup \Sigma_0, \qquad \Sigma_0 \subset [\lambda_*, \infty),
\end{equation}
with $-4\pi c_0^2$ and $0$ simple eigenvalues.

Now, because it is a relatively compact perturbation of $I^{-1} H_0$, for $0 < \beta_0 \ll 1$ and $\abs{\beta-\beta_0} \ll 1$, the spectrum of $I^{-1}H_c$ consists of a subset $\Sigma_c \subset [\lambda_*/2, \infty)$ along with two simple eigenvalues $\mu_0(\beta)$ and $\mu_1(\beta)$ that bifurcate from the $-4\pi c_0^2$ and $0$, respectively. But, we know that $0$ is an eigenvalue of $I^{-1} H_c$ with associated eigenvector $T^\prime(0) U_c$ by translation invariance, hence $\mu_1(\beta) = 0$ for all such $\beta$. Likewise, by continuity, for $|\beta-\beta_0| \ll 1$, we must have that $\mu_0(\beta) < 0$. This completes the proof of Assumption~\ref{spectral assumptions}, and hence the orbital stability of $U_c$.
\end{proof}

\section*{Acknowledgments}

The research of SW is supported in part by the NSF through DMS-2306243, and the Simons Foundation through award 960210.

\appendix

\bibliographystyle{siam}
\bibliography{projectdescription}

\end{document}